\documentclass[11pt]{article}

\usepackage[margin=1in]{geometry}
\usepackage[T1]{fontenc}
\usepackage[utf8]{inputenc}
\usepackage{lmodern}
\usepackage{amsmath,amssymb,amsthm,mathtools}
\usepackage{aliascnt}
\usepackage{bm}
\usepackage{enumitem}
\usepackage{algorithm}
\usepackage{algpseudocode}
\usepackage{graphicx}
\usepackage{hyperref}
\usepackage{booktabs}
\hypersetup{hidelinks,hypertexnames=false}
\usepackage{microtype}
\usepackage{mathrsfs}
\newcommand{\num}[1]{\ensuremath{\csname computed/#1\endcsname}}
\usepackage{xcolor}
\usepackage[nameinlink,capitalize]{cleveref}
\usepackage[bottom]{footmisc}
\usepackage{caption}

\usepackage{xcolor}
\usepackage{tikz}
\usepackage{caption}

\AtBeginDocument{%
  \setlength{\abovedisplayskip}{7pt plus 2pt minus 3pt}%
  \setlength{\belowdisplayskip}{7pt plus 2pt minus 3pt}%
  \setlength{\abovedisplayshortskip}{4pt plus 2pt minus 2pt}%
  \setlength{\belowdisplayshortskip}{4pt plus 2pt minus 2pt}%
}

\usepackage[compact]{titlesec}

\titlespacing*{\section}
  {0pt}{1.55ex plus 0.35ex minus 0.25ex}{0.75ex plus 0.15ex}

\titlespacing*{\subsection}
  {0pt}{1.25ex plus 0.30ex minus 0.20ex}{0.55ex plus 0.12ex}

\titlespacing*{\subsubsection}
  {0pt}{1.00ex plus 0.25ex minus 0.15ex}{0.45ex plus 0.10ex}

\titleformat{\section}
  {\normalfont\large\bfseries}{\thesection.}{0.5em}{}

\titleformat{\subsection}
  {\normalfont\normalsize\bfseries}{\thesubsection.}{0.5em}{}

\titleformat{\subsubsection}
  {\normalfont\normalsize\itshape}{\thesubsubsection.}{0.5em}{}

\setlist{
  topsep=2pt plus 1pt minus 1pt,
  itemsep=1pt plus 0.5pt,
  parsep=0pt,
  partopsep=0pt,
  leftmargin=*
}

\makeatletter
\def\thm@space@setup{%
  \thm@preskip=6pt plus 2pt minus 2pt
  \thm@postskip=6pt plus 2pt minus 2pt
}
\makeatother

\makeatletter
\renewcommand{\maketitle}{%
  \begin{center}
    {\Large\bfseries \@title\par}
    \vskip 0.45em
    {\normalsize
      \lineskip .5em
      \begin{tabular}[t]{c}
        \@author
      \end{tabular}\par}
    \vskip 0.35em
    {\small \@date\par}
  \end{center}
  \par\vskip 0.35em
  \@thanks
  \vskip 0.7em
}
\makeatother

\renewenvironment{abstract}
  {\small\begin{center}\bfseries Abstract\end{center}\vspace{-0.6em}\small}
  {\par\vspace{0.6em}}

\newtheorem{theorem}{Theorem}[section]

\newaliascnt{lemma}{theorem}
\newtheorem{lemma}[lemma]{Lemma}
\aliascntresetthe{lemma}

\newaliascnt{proposition}{theorem}
\newtheorem{proposition}[proposition]{Proposition}
\aliascntresetthe{proposition}

\newaliascnt{corollary}{theorem}
\newtheorem{corollary}[corollary]{Corollary}
\aliascntresetthe{corollary}

\theoremstyle{definition}

\newaliascnt{definition}{theorem}
\newtheorem{definition}[definition]{Definition}
\aliascntresetthe{definition}

\newaliascnt{assumption}{theorem}

\aliascntresetthe{assumption}

\theoremstyle{remark}

\newaliascnt{remark}{theorem}
\newtheorem{remark}[remark]{Remark}
\aliascntresetthe{remark}

\crefname{theorem}{theorem}{theorems}
\Crefname{theorem}{Theorem}{Theorems}

\crefname{lemma}{lemma}{lemmas}
\Crefname{lemma}{Lemma}{Lemmas}

\crefname{proposition}{proposition}{propositions}
\Crefname{proposition}{Proposition}{Propositions}

\crefname{corollary}{corollary}{corollaries}
\Crefname{corollary}{Corollary}{Corollaries}

\crefname{definition}{definition}{definitions}
\Crefname{definition}{Definition}{Definitions}

\crefname{assumption}{assumption}{assumptions}
\Crefname{assumption}{Assumption}{Assumptions}

\crefname{remark}{remark}{remarks}
\Crefname{remark}{Remark}{Remarks}

\newcounter{prooflink}
\crefname{prooflink}{proof}{proofs}
\Crefname{prooflink}{Proof}{Proofs}
\newcommand{\prooflink}[1]{\par\noindent{\small\hyperref[proof:#1]{Proof in the appendix.}}}

\newcommand{\R}{\mathbb R}

\newcommand{\tr}{\operatorname{tr}}
\newcommand{\spec}{\operatorname{spec}}
\newcommand{\End}{\operatorname{End}}
\newcommand{\TT}{\mathcal T}
\newcommand{\abs}[1]{\left\lvert #1\right\rvert}

\newcommand{\diag}{\operatorname{diag}}

\newcommand{\HS}{\mathrm{HS}}

\title{Chebyshev-Exact Acceleration under Hessian Variation,\\I: Sine-Jacobi Method}
\author{Dmitry Pasechnyuk-Vilensky\thanks{Mohamed bin Zayed University of Artificial Intelligence, Abu Dhabi, UAE. \texttt{dmivilensky1@gmail.com}}, Martin Tak\'{a}\v{c}\thanks{Mohamed bin Zayed University of Artificial Intelligence, Abu Dhabi, UAE.}}
\date{}

\begin{document}
\maketitle

\begin{abstract}
We study finite-horizon one-gradient realizations with the Chebyshev minimax
terminal residual on \([\mu,L]\).  Under time-dependent Hessian perturbations,
the terminal first variation is governed by a time-ordered spectral kernel
\(K_s(\lambda,\nu)\); its sharp \(\ell_2\) gain is \(A_N\).
For the prefix-exact Chebyshev recurrence,
\[
        A_N^{\rm pref}
        =\frac{\epsilon_N^\star}{L-\mu}
        \left(4N^2+16\sum_{m=1}^{N-1}m^2\right)^{1/2}
        =\frac4{\sqrt3}\frac{N^{3/2}}{L-\mu}\epsilon_N^\star(1+o(1)),
\]
and this is sharp in the causal two-term class with Chebyshev exactness at every prefix.
For terminal-only exactness, Jacobi coordinates give \(P_N=2^{1-N}T_N\): the spectrum is fixed at the midpoint Chebyshev nodes, while the spectral weights parametrize the realizations.  The sine weights give a final-exact Jacobi method with the same terminal residual and
\[
        A_N(J_N^{\sin})
        =2\sqrt{c_{\sin}}\frac{N^{3/2}}{L-\mu}\epsilon_N^\star(1+o(1)),
        \; 2\sqrt{c_{\sin}}\approx2.137936<4/\sqrt3.
\]
Thus the Chebyshev terminal polynomial does not determine the first-order Hessian-drift gain.  Momentum is necessary for this order: every final-Chebyshev-exact stepsize schedule, in any order, has \(A_N\ge2\sqrt{2/3}\,N^2\epsilon_N^\star/(L-\mu)\).
Beyond first order, the response to drift \(E_s=H_s-H\) of size \(\delta=\max_s\|E_s\|\) is a finite polynomial expansion, and for both realizations the terminal error is at most \(\epsilon_N^\star T_N(1+2\delta/(L-\mu))\), the Chebyshev bound of the interval widened by \(\delta\).  This bound is attained, holds for objectives with Lipschitz Hessian, and shows that drift reduces the accelerated rate exponent by the factor \(1-\sqrt{\delta/\mu}\).  The two realizations thus share the worst-case response and differ in the \(\ell_2\)-in-time response measured by \(A_N\).
The experiments show the finite-horizon effect: lower stochastic curvature overhead, larger admissible-block frontiers, accurate time-varying quadratic predictions, and lower restart cost on an endpoint-coupled smooth strongly convex GLM.
\end{abstract}

\section*{Introduction}

For a fixed quadratic with spectrum in \([\mu,L]\), the best degree-\(N\) residual is \(p_N^\star(\lambda)=T_N(z(\lambda))/T_N(\eta)\), where \(z(\lambda)=(L+\mu-2\lambda)/(L-\mu)\) and \(\eta=(L+\mu)/(L-\mu)\).  This terminal fixed-Hessian transfer function does not specify the ordered product of gradient-query factors that realizes it.  Under Hessian drift, that order matters.

Let the Hessian at step \(s\) be \(H_s=H+E_s\).  Writing \(E_s=\delta B_s\) with fixed directions \(B_s\) and a scale \(\delta\), an insertion from the \(\nu\)-eigenspace to the \(\lambda\)-eigenspace at time \(s\) enters the first variation with coefficient \(K_s(\lambda,\nu)\).  We measure the sharp spectral insertion gain by
\[
        A_N=\sup_{\lambda,\nu\in[\mu,L]}
        \left(\sum_{s=0}^{N-1}|K_s(\lambda,\nu)|^2\right)^{1/2}.
\]
For each spectral pair this is exactly the \(\ell_2\)-operator norm from Hessian insertions to the terminal first variation, hence a realization-level invariant.

For the prefix-exact Chebyshev recurrence we compute
\[
        A_N^{\rm pref}
        =\frac{\epsilon_N^\star}{L-\mu}
        \left(4N^2+16\sum_{m=1}^{N-1}m^2\right)^{1/2}
        =\frac{4}{\sqrt3}\frac{N^{3/2}}{L-\mu}
        \epsilon_N^\star(1+o(1)),
        \; \epsilon_N^\star=|T_N(\eta)|^{-1},
\]
and prove sharpness in the causal two-term class with Chebyshev exactness at every prefix.

We then pass to terminal-only exactness.  In Jacobi-normalized two-term coordinates, \(P_{s+1}(z)=(z-b_s)P_s(z)-a_s^2P_{s-1}(z)\), \(a_s>0\), and final Chebyshev exactness is \(P_N=2^{1-N}T_N\).  Thus the spectrum is fixed at the midpoint Chebyshev nodes, while the spectral weights remain free.  The first-variation kernel factorizes into a left principal polynomial and a right-tail polynomial.

The sine--Jacobi realization uses weights proportional to \(\sin((2j-1)\pi/(2N))\).  It is final-Chebyshev exact and satisfies
\[
        A_N(J_N^{\sin})
        =2\sqrt{c_{\sin}}\frac{N^{3/2}}{L-\mu}
        \epsilon_N^\star(1+o(1)),
        \;
        2\sqrt{c_{\sin}}\approx 2.137936
        <\frac4{\sqrt3}\approx2.309401.
\]
Thus the prefix and sine--Jacobi realizations share the same terminal Chebyshev polynomial but have different first-order Hessian-drift gains.

Without momentum the order is worse: every final-Chebyshev-exact
stepsize schedule, including the fractal ordering of the Chebyshev steps,
has gain at least of order \(N^2\epsilon_N^\star/(L-\mu)\)
(\Cref{thm:stepsize-lower}).

Beyond first order, the expansion in the drift \(E_s=H_s-H\) is finite
(\Cref{prop:finite-drift-linear}), and for both realizations the terminal
error under drift of size \(\delta=\max_s\|E_s\|\) is at most
\(\epsilon_N^\star T_N(1+2\delta/\Delta)\), the Chebyshev bound of the interval
widened by \(\delta\); this bound is attained
(\Cref{thm:finite-drift,cor:prefix-finite-drift,cor:sine-finite-drift}).
It applies to objectives with Lipschitz Hessian along the trajectory
(\Cref{cor:lipschitz-hessian}).  Equivalently, the drift multiplies the
accelerated exponent \(2N/\sqrt{\kappa}\), \(\kappa=L/\mu\), by
\(1-\sqrt{\delta/\mu}\) (\Cref{rem:drift-scale}).  In ill-conditioned problems,
drift that is negligible relative to \(L\) but a fraction of \(\mu\) therefore
changes the terminal error by orders of magnitude.  In operator norm the two
realizations have the same worst-case response; they differ in how the
response accumulates in \(\ell_2\) over time, which is what \(A_N\)
measures.

The experiments keep the terminal polynomial fixed and vary only the realization.  They compare finite-\(N\) drift constants, stochastic admissibility frontiers, time-varying quadratic packet responses, and finite-radius restarted blocks on an endpoint-coupled smooth strongly convex GLM.

We prove: (1) the time-ordered first-variation kernel and gain \(A_N\), together with an exact finite expansion in the drift whose linear part has size exactly \(A_N\); (2) a lower bound of order \(N^2\epsilon_N^\star/(L-\mu)\) for every final-Chebyshev-exact stepsize schedule; (3) the exact prefix gain; (4) its sharpness in the causal prefix-exact two-term class; (5) the Jacobi parametrization and kernel factorization for terminal-exact realizations; (6) the sine--Jacobi construction and its constant \(2\sqrt{c_{\sin}}\); (7) a finite-drift expansion in Jacobi coordinates with the attained bound \(\epsilon_N^\star T_N(1+2\delta/\Delta)\) for both realizations, and its consequence for objectives with Lipschitz Hessian; and (8) finite-horizon experiments comparing the two exact realizations.

The appendices contain the proofs, inverse-spectral input, sine-lattice estimates, the explicit evaluation of \(c_{\sin}\), the Fourier positivity of the sine--Jacobi blocks, and additional experimental diagnostics.

\paragraph{Related work.}
Chebyshev acceleration is the classical minimax-polynomial construction for a prescribed spectral interval; see \cite[Sec.~2.3, Eqs.~(2.12) and~(2.15)]{dAspremontScieurTaylor}.  The broader history of accelerated first-order methods includes \cite{Polyak1964,Nesterov1983}. The object studied here is not the terminal polynomial alone, but its ordered realization under Hessian variation: two methods can agree on every fixed Hessian and differ under time-varying Hessians.

Agarwal, Goel, and Zhang study fractal Chebyshev schedules, explicitly
crediting the underlying ordering to Lebedev and Finogenov
\cite[Sec.~1.2; Definition~5]{AgarwalGoelZhang2021}.
For a fixed positive-definite quadratic, permutations of the Chebyshev
steps preserve the terminal residual polynomial
\cite[Remark~2 and Sec.~2.3]{AgarwalGoelZhang2021},
but the ordering affects intermediate products and error propagation.
For power-of-two horizons and positive spectral bounds enclosing the
Hessian spectrum, they control products over contiguous time intervals
and the accumulation of uniformly bounded additive update errors
\cite[Theorems~9--10 and Corollary~11]{AgarwalGoelZhang2021}.
This provides a direct antecedent to the distinction between terminal
polynomial and realization-level stability. Our focus is the
two-spectral-variable first-variation kernel \(K_s(\lambda,\nu)\) for
Hessian perturbations and its sharp spectral-pair \(\ell_2\) gain \(A_N\),
which we compute and compare for prefix-exact and sine--Jacobi two-term
realizations, rather than for permutations of a gradient-descent schedule.

A complementary line of work studies acceleration of gradient descent
on a fixed, possibly nonquadratic objective by stepsize selection alone.
For \(L\)-smooth, \(\mu\)-strongly convex objectives, Altschuler and
Parrilo's Silver Stepsize Schedule gives iteration complexity
\(O(\kappa^{\log_{1+\sqrt{2}}2}\log(1/\varepsilon))\),
where \(\kappa=L/\mu\), for reducing the squared distance to the
minimizer by a factor \(\varepsilon\); the stated guarantee uses
power-of-two horizons
\cite[Theorem~1.1]{AltschulerParriloSilverI}.
Their companion paper gives a direct analysis for \(L\)-smooth convex
objectives admitting a minimizer, without requiring strong convexity:
the Silver schedule satisfies
\(f(x_n)-f(x^\star)
=O(L\|x_0-x^\star\|^2 n^{-\log_2(1+\sqrt{2})})\)
at horizons \(n=2^k-1\), \(k\ge1\)
\cite[Theorem~1]{AltschulerParriloSilverII}.
Repeating suitably chosen finite blocks also recovers the strongly
convex complexity above
\cite[Corollary~2]{AltschulerParriloSilverII}.

These exact-gradient guarantees are established through multi-step
descent and recursive gluing of consistency inequalities for a single
objective \cite[Sec.~3]{AltschulerParriloSilverII}.
They concern worst-case convergence beyond quadratics, rather than
sensitivity among realizations constrained to have the same terminal
Chebyshev residual. Here we impose that terminal constraint and compare
the sharp spectral-pair \(\ell_2\) gain \(A_N\) of the first variation
under time-dependent Hessian insertions. For final-Chebyshev-exact stepsize schedules, including every ordering
studied in \cite{AgarwalGoelZhang2021}, the drift gain is at least of order
\(N^2\epsilon_N^\star/\Delta\) (\Cref{thm:stepsize-lower}), whereas the
two-term realizations studied here have order \(N^{3/2}\epsilon_N^\star/\Delta\).
The Silver schedules repeat a stepsize (for \(N\ge3\) in the convex case
and dyadic \(N\ge4\) in the strongly convex case); they are therefore not
final-Chebyshev exact, and by the minimax property of \(p_N^\star\) their
terminal residual is strictly larger than \(\epsilon_N^\star\).

Performance estimation \cite{DroriTeboulle2014}, smooth strongly convex interpolation \cite[Theorem~4]{TaylorHendrickxGlineur2017}, and optimized first-order methods \cite[Theorem~2]{KimFessler2016} analyze and optimize worst-case oracle performance for fixed objective functions.  The criterion studied here is different: \(A_N\) is the sharp first variation of an ordered product under spectral Hessian insertions.

Related models include inexact first-order information controlled by \((\delta,L)\)-oracle inequalities \cite[Definition~1]{DevolderGlineurNesterov2014}, additive gradient noise \cite[Eq.~(1.2)]{CohenDiakonikolasOrecchia2018}, and IQC-based convergence certificates \cite[Theorem~4]{LessardRechtPackard2016}.  For fixed-parameter noisy methods on strongly convex quadratics, exact robustness formulas describe objective error \cite[Proposition~3.7]{AybatFallahGurbuzbalabanOzdaglar2020} and iterate variance \cite[Theorem~1]{MohammadiRazaviyaynJovanovic2021}.  Here the perturbation is multiplicative and spectral: an insertion from \(\nu\) to \(\lambda\) is propagated on the right at \(\nu\) and on the left at \(\lambda\), producing \(K_s(\lambda,\nu)\).

The Jacobi coordinates are the standard finite inverse-spectral coordinates for orthogonal polynomials and Jacobi matrices \cite[Sec.~2.5]{Teschl}.  Final Chebyshev exactness fixes the midpoint Chebyshev spectrum, while the spectral weights parametrize the remaining two-term realizations; the sine weights pick a persymmetric representative with a smaller drift constant.

\section{Realizations and first-order drift gain}
\label{sec:problem-setting}

Fix
$
0<\mu<L,
\;
\Delta:=L-\mu,
\;
\eta:=\frac{L+\mu}{L-\mu}>1,
\;
z(\lambda):=\eta-\frac{2\lambda}{\Delta}.
$
The terminal polynomial throughout the paper is
\[
        p_N^\star(\lambda)=\frac{T_N(z(\lambda))}{T_N(\eta)},
        \;
        \epsilon_N^\star:=\frac1{|T_N(\eta)|}.
\]
Here and below, \(T_k\) and \(U_k\) denote the Chebyshev polynomials of the
first and second kind, respectively.

Let \(X\) be a finite-dimensional real Hilbert space.  A length-\(N\)
one-gradient realization with memory \(m\) is a tuple
\[
\mathcal R=(m,b,\ell,\{P_s,g_s,q_s\}_{s=0}^{N-1}),
\;
b,\ell,g_s,q_s\in\R^m,
\quad
P_s\in\R^{m\times m},
\]
acting on \(Y_s\in\R^m\otimes X\) by
\[
Y_{s+1}
=(P_s\otimes I_X)Y_s
-(g_s\otimes I_X)\nabla f((q_s^\top\otimes I_X)Y_s),
\;
Y_0=b\otimes x_0,
\;
x_N=(\ell^\top\otimes I_X)Y_N.
\]
For a quadratic \(f_H(x)=\frac12\langle Hx,x\rangle\), define
\[
        C_s(\lambda):=P_s-\lambda g_sq_s^\top,
        \;
        p_{\mathcal R,N}(\lambda)
        :=\ell^\top C_{N-1}(\lambda)\cdots C_0(\lambda)b.
\]
The realization is final-Chebyshev exact if
$
        p_{\mathcal R,N}(\lambda)=p_N^\star(\lambda)
        \; (\lambda\in[\mu,L]).
$

Now let the Hessian vary along the trajectory.  Fix a reference Hessian
\(H=H^\top\) with \(\mu I\preceq H\preceq LI\), and let the gradient at step
\(s\) be that of the quadratic with Hessian \(H_s=H+E_s\), where
\(E_s\in\End(X)\) is the actual drift at step \(s\).  With
\(\mathsf C_s(A):=P_s\otimes I_X-g_sq_s^\top\otimes A\) on \(\R^m\otimes X\),
the output operator
\[
\TT_{\mathcal R}(H_0,\ldots,H_{N-1})
:=(\ell^\top\otimes I_X)\,
\mathsf C_{N-1}(H_{N-1})\cdots\mathsf C_0(H_0)\,(b\otimes I_X)
\]
maps \(x_0-x^\star\) to \(x_N-x^\star\), and
\(\TT_{\mathcal R}(H,\ldots,H)=p_{\mathcal R,N}(H)\).  To isolate the linear
response, fix directions \(B_0,\ldots,B_{N-1}\) and write \(E_s=\delta B_s\);
the scalar \(\delta\) only measures the size of the drift along these
directions.  If an insertion \(B_s\)
maps the \(\nu\)-eigenspace of \(H\) to the \(\lambda\)-eigenspace, the factors
to the right of the insertion propagate at \(\nu\), while the factors to the
left propagate at \(\lambda\).  This gives the scalar first-variation kernel
\[
K_{s,\mathcal R}(\lambda,\nu)
:=-\ell^\top
C_{N-1}(\lambda)\cdots C_{s+1}(\lambda)
 g_sq_s^\top
C_{s-1}(\nu)\cdots C_0(\nu)b,
\]
with empty products equal to the identity.

\begin{lemma}[First-order expansion]\label{lem:first-order-expansion}
Let \(H_s=H+\delta B_s\).  The derivative of
\(\delta\mapsto\TT_{\mathcal R}(H_0,\ldots,H_{N-1})\) at \(\delta=0\) is
\[
D\TT_{\mathcal R,H}[B_0,\ldots,B_{N-1}]
=
\sum_{s=0}^{N-1}\mathcal L_{s,H}(B_s),
\]
and on the spectral pair \((\lambda,\nu)\),
$
\mathcal L_{s,H}(B_s)_{\lambda\nu}
=
K_{s,\mathcal R}(\lambda,\nu)(B_s)_{\lambda\nu}.
$
\end{lemma}
\prooflink{lem:first-order-expansion}

\begin{definition}[First-order Hessian-drift gain]
The realization-level drift gain is
\[
A_N(\mathcal R)
:=
\sup_{\lambda,\nu\in[\mu,L]}
\left(
\sum_{s=0}^{N-1}|K_{s,\mathcal R}(\lambda,\nu)|^2
\right)^{1/2}.
\]
\end{definition}
For each fixed spectral pair, the quantity
\(\bigl(\sum_s|K_{s,\mathcal R}(\lambda,\nu)|^2\bigr)^{1/2}\)
is the sharp \(\ell_2\)-operator norm of
\((b_s)_s\mapsto\sum_sK_{s,\mathcal R}(\lambda,\nu)b_s\);
\(A_N(\mathcal R)\) is the supremum of these pairwise norms.

The first-order expansion is the linear part of an exact finite identity,
and \(A_N\) is precisely the size of that linear part.

\begin{proposition}[Finite drift and the size of the linear response]
\label{prop:finite-drift-linear}
Let \(\mu I\preceq H\preceq LI\) and let \(E_0,\ldots,E_{N-1}\in\End(X)\) be
arbitrary.  Then
\[
\TT_{\mathcal R}(H+E_0,\ldots,H+E_{N-1})
=p_{\mathcal R,N}(H)
+D\TT_{\mathcal R,H}[E_0,\ldots,E_{N-1}]
+\mathcal Q_{\mathcal R,H}[E_0,\ldots,E_{N-1}],
\]
where \(\mathcal Q_{\mathcal R,H}\) is a sum of \(k\)-linear expressions in
\((E_0,\ldots,E_{N-1})\) with \(2\le k\le N\).  Moreover
\[
\bigl\|D\TT_{\mathcal R,H}[E_0,\ldots,E_{N-1}]\bigr\|_{\HS}
\le
A_N(\mathcal R)
\Bigl(\sum_{s=0}^{N-1}\|E_s\|_{\HS}^2\Bigr)^{1/2},
\]
and the constant \(A_N(\mathcal R)\) cannot be decreased.
\end{proposition}
\prooflink{prop:finite-drift-linear}

Thus \(A_N\) is the exact Lipschitz constant of the linear response to drift,
measured in the Hilbert--Schmidt norm and summed in \(\ell_2\) over time, and
the expansion terminates at degree \(N\).  For Jacobi realizations,
Section~\ref{sec:jacobi-main} bounds the higher-order part
\(\mathcal Q_{\mathcal R,H}\) in terms of \(\delta=\max_s\|E_s\|\); relative to the
bound on the linear part it is of order \(N^2\delta/\Delta\).

The simplest final-Chebyshev-exact realizations use no momentum.  Gradient
descent with stepsizes \(\gamma_0,\ldots,\gamma_{N-1}>0\),
\[
x_{s+1}=x_s-\gamma_s\nabla f(x_s),
\]
is the realization with \(m=1\), \(b=\ell=P_s=q_s=1\) and \(g_s=\gamma_s\).
Its residual polynomial is \(\prod_s(1-\gamma_s\lambda)\), so it is
final-Chebyshev exact if and only if the inverse stepsizes
\(\gamma_s^{-1}\) run through the roots of \(p_N^\star\),
\[
\lambda_j=\frac{\Delta}{2}(\eta-\zeta_j),
\qquad
\zeta_j=\cos\frac{(2j-1)\pi}{2N},
\qquad j=1,\ldots,N,
\]
in some order.  The order does not affect the fixed-Hessian residual; the
fractal Chebyshev schedule \cite[Definition~5]{AgarwalGoelZhang2021} is one
such order.  No order makes the drift gain small.

\begin{theorem}[Drift gain of exact stepsize schedules]
\label{thm:stepsize-lower}
Let \(N\ge1\).  For every final-Chebyshev-exact stepsize schedule,
\[
A_N
\ge
\frac{2\epsilon_N^\star}{\Delta}
\sqrt{\frac{2N^4+N^2}{3}}
=
2\sqrt{\frac23}\,\frac{N^2}{\Delta}\,\epsilon_N^\star
\bigl(1+O(N^{-2})\bigr).
\]
\end{theorem}
\prooflink{thm:stepsize-lower}

Thus, without momentum, the drift gain is at least of order
\(N^2\epsilon_N^\star/\Delta\).  The rest of the paper computes and compares
\(A_N\) for final-Chebyshev-exact two-term realizations, for which the order
improves to \(N^{3/2}\epsilon_N^\star/\Delta\).

\section{Prefix-exact Chebyshev recurrence and exact gain}
\label{sec:cheb-upper}

We recall the prefix-exact two-term Chebyshev recurrence and compute its Hessian-drift first variation.  For a fixed symmetric Hessian with \(\mu I\preceq H\preceq LI\) and \(0<\mu<L\), the corresponding classical recurrence and residual bound appear in \cite[Eq.~(2.16); Algorithm~2 and Theorem~2.1]{dAspremontScieurTaylor}.  Their spectral coordinate is the negative of ours; the normalized residual is unchanged because \(T_s(-x)=(-1)^sT_s(x)\).  That fixed-Hessian result is not used as a bound for time-dependent perturbations. Define
$
\beta_t:=\frac{T_{t-1}(\eta)}{T_{t+1}(\eta)},
\; t\ge1.
$
The method is initialized by
$
        x_1=x_0-\frac{2}{L+\mu}\nabla f(x_0),
$
and, for \(t\ge1\), evolves by
\[
        x_{t+1}
        =x_t+\beta_t(x_t-x_{t-1})
        -(1+\beta_t)\frac{2}{L+\mu}\nabla f(x_t).
\]

\begin{theorem}[Prefix-exact Chebyshev recurrence]\label{thm:prefix-exact-chebyshev-recurrence}
On every quadratic objective \(f(x)=\frac12\langle Hx,x\rangle\) with
\(\spec(H)\subset[\mu,L]\), the recurrence satisfies
$
        x_t-x^\star=p_t^\star(H)(x_0-x^\star)
        \;
        \forall t\ge0.
$
\end{theorem}
\prooflink{thm:prefix-exact-chebyshev-recurrence}

\begin{theorem}[Exact prefix first-order kernel]\label{thm:prefix-kernel}
Let \(N\ge1\).  For the prefix-exact Chebyshev recurrence,
\[
K_0^{\rm pref}(\lambda,\nu)
=
-\frac{2}{\Delta T_N(\eta)}U_{N-1}(z(\lambda)),
\]
and, for \(1\le s\le N-1\),
\[
K_s^{\rm pref}(\lambda,\nu)
=
-\frac{4}{\Delta T_N(\eta)}
U_{N-s-1}(z(\lambda))T_s(z(\nu)).
\]
Consequently
\[
A_N^{\rm pref}
=
\frac{\epsilon_N^\star}{\Delta}
\left(4N^2+16\sum_{m=1}^{N-1}m^2\right)^{1/2}
=
\frac{2\epsilon_N^\star}{\Delta}
\sqrt{N^2+\frac23N(N-1)(2N-1)},
\]
and
\[
A_N^{\rm pref}
=
\frac{4}{\sqrt3}\frac{N^{3/2}}{\Delta}\epsilon_N^\star(1+o(1)).
\]
\end{theorem}
\prooflink{thm:prefix-kernel}

In units of \(4(\epsilon_N^\star)^2/\Delta^2\), the squared prefix gain is
\(N^2+\frac23N(N-1)(2N-1)\), while the squared lower bound of
\Cref{thm:stepsize-lower} is \((2N^4+N^2)/3\).  Their difference is
\[
\frac{2N^4+N^2}{3}-N^2-\frac23N(N-1)(2N-1)
=\frac23N(N-1)(N^2-N+1),
\]
which is positive for \(N\ge2\).  Hence every final-Chebyshev-exact
stepsize schedule has a strictly larger drift gain than the prefix-exact
recurrence, by a factor at least \(\sqrt{N/2}\,(1+o(1))\).

\section{Sharpness within the prefix-exact Chebyshev class}
\label{sec:prefix-lower-main}

The exact prefix value is also the best possible value under the stronger
constraint that the two-term recurrence be Chebyshev-exact at every prefix.
The auxiliary uniqueness statement for such recurrences is part of the proof
in Appendix~\ref{app:proofs-sec3}.

\begin{theorem}[Sharp lower bound for causal two-term every-step Chebyshev methods]\label{thm:prefix-lower}
Let \(N\ge1\).  Consider a causal two-term one-gradient recurrence initialized by
$
x_1=x_0-\gamma_0\nabla f(x_0),
$
and, for \(1\le t\le N-1\), evolved by
\(x_{t+1}=a_tx_t+b_tx_{t-1}-\gamma_t\nabla f(x_t)\).
Assume that, on every scalar quadratic objective with curvature
\(\lambda\in[\mu,L]\), its residual polynomials are prefix-exact Chebyshev:
$
r_t(\lambda)=p_t^\star(\lambda),
\;
0\le t\le N.
$
Then
\[
A_N
\ge
\frac{\epsilon_N^\star}{\Delta}
\left(4N^2+16\sum_{m=1}^{N-1}m^2\right)^{1/2}.
\]
Equality is attained by the standard prefix-exact Chebyshev recurrence.
Consequently this recurrence is optimal inside the causal two-term
prefix-exact class, and
\(A_N\ge \frac{4}{\sqrt3}\frac{N^{3/2}}{\Delta}\epsilon_N^\star(1-o(1))\).
\end{theorem}
\prooflink{thm:prefix-lower}

The next step is to replace prefix exactness by terminal exactness and
parametrize the resulting two-term recurrences.

\section{Jacobi coordinates and first-variation factorization}
\label{sec:jacobi-main}

We work in Jacobi-normalized coordinates for terminal-exact two-term recurrences:
\[
P_{-1}=0,
\;
P_0=1,
\;
P_{s+1}(z)=(z-b_s)P_s(z)-a_s^2P_{s-1}(z),
\; a_s>0.
\]
In these coordinates, final Chebyshev exactness is the terminal condition
$
        P_N(z)=2^{1-N}T_N(z),
$
so the Jacobi spectrum is fixed to
\(\zeta_j=\cos\frac{(2j-1)\pi}{2N}, \; j=1,\ldots,N\).
The free parameters are the spectral weights on these fixed nodes,
or equivalently the intermediate Jacobi coefficients.  The two structural facts
below give the coordinate representation and the time-ordered kernel in this
parametrization.

\begin{theorem}[Jacobi method]
\label{thm:jacobi-method}
Let
$
P_{-1}:=0,
\; P_0:=1,
\; a_0:=0,
$
and let
\(P_{s+1}(z)=(z-b_s)P_s(z)-a_s^2P_{s-1}(z), \; s=0,\ldots,N-1\),
with \(a_s>0\) for \(s=1,\ldots,N-1\).  Assume
$
P_s(\eta)>0,
\; s=0,\ldots,N.
$
Define
$
r_s(\lambda):=\frac{P_s(z(\lambda))}{P_s(\eta)}.
$
The associated two-term method is
\[
x_1
=
\frac1{P_1(\eta)}
\left[
(\eta-b_0)x_0-\frac{2}{\Delta}\nabla f(x_0)
\right],
\]
and, for \(s=1,\ldots,N-1\),
\[
x_{s+1}
=
\frac{P_s(\eta)}{P_{s+1}(\eta)}
\left[
(\eta-b_s)x_s-\frac{2}{\Delta}\nabla f(x_s)
\right]
-
a_s^2\frac{P_{s-1}(\eta)}{P_{s+1}(\eta)}x_{s-1}.
\]
On every quadratic with \(\mu I\preceq H\preceq LI\), its errors satisfy
$
e_s=r_s(H)e_0,
\; s=0,\ldots,N.
$
In particular, if \(P_N(z)=2^{1-N}T_N(z)\), then
\(e_N=p_N^\star(H)e_0\).
\end{theorem}
\prooflink{thm:jacobi-method}

Let \(J\) be the \(N\times N\) Jacobi matrix with diagonal entries
\(b_0,\ldots,b_{N-1}\) and off-diagonal entries \(a_1,\ldots,a_{N-1}\).
For the right tail block \(J_{s+1:N-1}\), the principal submatrix on the sites
\(s+1,\ldots,N-1\), set
\(Q_{N-s-1}^{(s+1)}(z)=\det(zI-J_{s+1:N-1}), \; Q_0^{(N)}:=1\).

\begin{theorem}[Jacobi first-variation kernel]\label{thm:jacobi-first-variation-kernel}
Let \(\mathcal R_J\) be the Jacobi realization of \Cref{thm:jacobi-method}.
For \(x,y\in[-1,1]\), write
\[
K_s(x,y):=K_{s,\mathcal R_J}(\lambda(x),\lambda(y)),
\;
\lambda(u):=\frac{L+\mu}{2}-\frac{\Delta}{2}u.
\]
Then
\[
K_s(x,y)
=
-\frac{2}{\Delta P_N(\eta)}
Q_{N-s-1}^{(s+1)}(x)P_s(y),
\; s=0,\ldots,N-1.
\]
Consequently
\[
A_N(J)
=
\frac{2}{\Delta |P_N(\eta)|}
\sup_{x,y\in[-1,1]}
\left(
\sum_{s=0}^{N-1}
|Q_{N-s-1}^{(s+1)}(x)P_s(y)|^2
\right)^{1/2}.
\]
\end{theorem}
\prooflink{thm:jacobi-first-variation-kernel}

The factorization extends from the first variation to finite drift.  For
\(0\le a\le b+1\le N\) let \(J_{a:b}\) be the principal block of \(J\) on the
sites \(a,\ldots,b\), and set
\[
\widetilde D_{a:b}(x):=2^{\,b-a+1}\det(xI-J_{a:b}),
\qquad
\widetilde D_{a:a-1}:=1,
\qquad
g_{a:b}:=\max_{x\in[-1,1]}\bigl|\widetilde D_{a:b}(x)\bigr|.
\]
The factor \(2^{\,b-a+1}\) normalizes the blocks: if all couplings in the
block equal \(1/2\) and the diagonal vanishes, then
\(\widetilde D_{a:b}=U_{b-a+1}\).

\begin{theorem}[Finite Hessian drift in Jacobi coordinates]
\label{thm:finite-drift}
Let \(\mathcal R_J\) be a final-Chebyshev-exact Jacobi realization, let
\(H=H^\top\) with \(\mu I\preceq H\preceq LI\), put \(Z:=z(H)\), let
\(E_0,\ldots,E_{N-1}\in\End(X)\), and set \(\delta:=\max_s\|E_s\|\).
\begin{enumerate}[label=(\roman*),leftmargin=2em]
\item The output operator has the exact expansion
\(\TT_{\mathcal R_J}(H+E_0,\ldots,H+E_{N-1})=\sum_{k=0}^N\TT^{(k)}\), where
\[
\TT^{(k)}
=\frac{\epsilon_N^\star}{2}\Bigl(-\frac4\Delta\Bigr)^{k}
\!\!\sum_{0\le s_1<\cdots<s_k\le N-1}\!\!
\widetilde D_{s_k+1:N-1}(Z)\,E_{s_k}\,
\widetilde D_{s_{k-1}+1:s_k-1}(Z)\cdots
E_{s_1}\,\widetilde D_{0:s_1-1}(Z).
\]
Here \(\TT^{(0)}=p_N^\star(H)\) and
\(\TT^{(1)}=D\TT_{\mathcal R_J,H}[E_0,\ldots,E_{N-1}]\).
\item For every \(k\),
\[
\bigl\|\TT^{(k)}\bigr\|
\le\frac{\epsilon_N^\star}{2}\Bigl(\frac{4\delta}{\Delta}\Bigr)^{k}G_k(J),
\qquad
G_k(J):=\!\!\sum_{0\le s_1<\cdots<s_k\le N-1}\!\!
g_{0:s_1-1}\,g_{s_1+1:s_2-1}\cdots g_{s_k+1:N-1}.
\]
\item Always \(G_k(J)\ge2^{1-k}[t^k]\,T_N(1+t)\), with equality for every
\(k\) if each \(\widetilde D_{a:b}\) attains its maximum on \([-1,1]\) at
\(x=1\).  In that case, with
\(\TT=\TT_{\mathcal R_J}(H+E_0,\ldots,H+E_{N-1})\),
\[
\begin{aligned}
\bigl\|\TT-p_N^\star(H)\bigr\|
&\le\epsilon_N^\star\Bigl(T_N\bigl(1+\tfrac{2\delta}{\Delta}\bigr)-1\Bigr),
\qquad
\bigl\|D\TT_{\mathcal R_J,H}[E]\bigr\|\le\frac{2N^2\epsilon_N^\star\delta}{\Delta},\\
\bigl\|\TT-p_N^\star(H)-D\TT_{\mathcal R_J,H}[E]\bigr\|
&\le\epsilon_N^\star\Bigl(T_N\bigl(1+\tfrac{2\delta}{\Delta}\bigr)-1
-\tfrac{2N^2\delta}{\Delta}\Bigr)
\le\frac23\,\epsilon_N^\star N^4\Bigl(\frac\delta\Delta\Bigr)^2
\cosh\Bigl(2N\sqrt{\tfrac\delta\Delta}\Bigr).
\end{aligned}
\]
\item Conversely, let \(\mathcal R\) be any final-Chebyshev-exact
realization, \(\mu\in\spec(H)\), and \(E_s=-\delta I\) for all \(s\).  Then
\(\|\TT_{\mathcal R}-p_N^\star(H)\|\ge\epsilon_N^\star(T_N(1+2\delta/\Delta)-1)\),
\(\|D\TT_{\mathcal R,H}[E]\|\ge2N^2\epsilon_N^\star\delta/\Delta\), and
\(\|\TT_{\mathcal R}-p_N^\star(H)-D\TT_{\mathcal R,H}[E]\|\ge
\epsilon_N^\star(T_N(1+2\delta/\Delta)-1-2N^2\delta/\Delta)\), where
\(\TT_{\mathcal R}=\TT_{\mathcal R}(H+E_0,\ldots,H+E_{N-1})\).  Hence all
bounds in (iii) are attained.
\end{enumerate}
\end{theorem}
\prooflink{thm:finite-drift}

The right-hand side in (iii) is the Chebyshev bound for the widened interval
\([\mu-\delta,L+\delta]\).  For realizations whose blocks attain their
maxima at \(x=1\), time-varying drift of size \(\delta\) is therefore no worse,
in operator norm and at every order, than the static shift
\(H\mapsto H-\delta I\); by (iv), no final-Chebyshev-exact realization does
better.  The prefix-exact recurrence has this property, and so does the
sine--Jacobi chain of Section~\ref{sec:sine-upper-main}.

\begin{corollary}[Finite drift for the prefix-exact recurrence]
\label{cor:prefix-finite-drift}
The prefix-exact recurrence of \Cref{thm:prefix-exact-chebyshev-recurrence}
is the Jacobi realization with \(b_s=0\), \(a_1^2=\frac12\), and
\(a_s^2=\frac14\) for \(s\ge2\).  Its blocks are
\(\widetilde D_{0:b}=2T_{b+1}\) and \(\widetilde D_{a:b}=U_{b-a+1}\) for
\(a\ge1\), and each attains its maximum on \([-1,1]\) at \(x=1\).  Hence all
bounds of \Cref{thm:finite-drift}(iii) hold for the prefix-exact recurrence,
and they are attained.
\end{corollary}
\prooflink{cor:prefix-finite-drift}

Thus the linear part is at most \(2N^2\epsilon_N^\star\delta/\Delta\) in
operator norm, the higher-order part is smaller by a factor of order
\(N^2\delta/\Delta\), and \(\delta\sim\Delta/N^2\) is the scale at which
drift changes the terminal error by a constant factor.  This scale is small
relative to \(L\) but not relative to \(\mu\); see \Cref{rem:drift-scale}.

\section{The sine--Jacobi chain and its upper bound}
\label{sec:sine-upper-main}

We pick a canonical point in the final-Chebyshev fibre by imposing reversal symmetry on the endpoint Green profile; the finite sine-lattice identities are proved in Appendix~\ref{app:proofs-sec5}.  Let
\[
P_{s+1}(z)=(z-b_s)P_s(z)-a_s^2P_{s-1}(z),
\;
P_{-1}=0,\quad P_0=1,\quad a_s>0,
\]
be the Jacobi recurrence associated with a nondegenerate two-term realization.
Final Chebyshev exactness imposes
\(P_N(z)=2^{1-N}T_N(z)\).
Hence the spectrum of the \(N\times N\) Jacobi matrix is fixed:
\[
\zeta_j=\cos\theta_j,
\;
\theta_j=\frac{(2j-1)\pi}{2N},
\;
j=1,\ldots,N.
\]
The remaining freedom in the final-Chebyshev class is the choice of spectral
weights
\[
\rho=\sum_{j=1}^N w_j\delta_{\zeta_j},
\;
w_j>0,
\;
\sum_{j=1}^N w_j=1.
\]

For a Jacobi realization, the kernel is
\(K_s(x,y) = -\frac{2}{\Delta P_N(\eta)} Q^{(s+1)}_{N-s-1}(x)P_s(y)\),
where \(Q^{(s+1)}_{N-s-1}\) is the characteristic polynomial of the right tail.
At the endpoint \(x=y=1\),
\[
\frac{P_s(1)Q^{(s+1)}_{N-s-1}(1)}{P_N(1)}
=
\bigl[(I-J)^{-1}\bigr]_{ss}
=:d_s.
\]
Thus the endpoint contribution is controlled by the positive Green profile \(d=(d_0,\ldots,d_{N-1})\).  Its total mass is spectral:
\[
\sum_{s=0}^{N-1}d_s
=
\operatorname{tr}(I-J)^{-1}
=
\sum_{j=1}^N\frac{1}{1-\zeta_j}.
\]
The spectral weights determine how this fixed Green mass is distributed along
the chain, while the drift estimate contains the squared mass
\(\sum_{s=0}^{N-1}d_s^2\).

The continuum endpoint profile used below is defined here.  Let \(f\) be the positive non-logarithmic solution of
\[
        f''(t)+\frac{\pi^2}{4\sin^2(\pi t)}f(t)=0,
        \;
        f(t)\sim\sqrt t\quad(t\downarrow0),
\]
and define
\[
        I_{\sin}:=\int_0^1 f(u)f(1-u)\,du,
        \;
        \varphi_{\sin}(t):=\frac{f(t)f(1-t)}{I_{\sin}},
        \;
        c_{\sin}:=\int_0^1\varphi_{\sin}(t)^2dt.
\]
The normalization of \(f\) cancels in \(\varphi_{\sin}\).  Numerically,
\(c_{\sin}\approx1.1426922\) and \(2\sqrt{c_{\sin}}\approx2.137936\).

We use the following selection rule: choose the canonical representative in
the Chebyshev isospectral fibre whose endpoint Green profile is invariant
under reversal,
\(d_s=d_{N-1-s}\).
At the Jacobi-matrix level this is enforced by persymmetry,
\(RJR=J, \; Re_s=e_{N-1-s}\).
For a Jacobi matrix with prescribed simple eigenvalues
\(\zeta_1,\ldots,\zeta_N\), persymmetry is equivalent to the inverse-spectral
weights
\(w_j=\frac{c}{|P_N'(\zeta_j)|}\),
where \(c\) is the normalizing constant.

For \(P_N(z)=2^{1-N}T_N(z)\), we have
\(P_N'(\zeta_j) = 2^{1-N}N\,U_{N-1}(\zeta_j)\).
Since \(\zeta_j=\cos\theta_j\) and
\[
U_{N-1}(\cos\theta_j)
=
\frac{\sin(N\theta_j)}{\sin\theta_j},
\;
|\sin(N\theta_j)|=1,
\]
the persymmetric weights are
\(w_j = \frac{\sin\theta_j}{\sum_{k=1}^N\sin\theta_k}\).
These are the sine weights.

Let \(J_N^{\sin}\) denote the Jacobi matrix associated with these weights.  It
has zero diagonal and explicit off-diagonal coefficients
\[
a_{s,N}^2
=
\frac14
\frac{\sin^2(\pi s/N)}
{\sin\bigl(\pi(s-\tfrac12)/N\bigr)
 \sin\bigl(\pi(s+\tfrac12)/N\bigr)},
\;
s=1,\ldots,N-1.
\]
The resulting two-term method is the sine--Jacobi Chebyshev realization. It is
final-Chebyshev exact by construction. Moreover, for this chain,
\[
A_N(J_N^{\sin})
=
\frac{2\epsilon_N^\star}{\Delta}
\left(\sum_{s=0}^{N-1}d_s^2\right)^{1/2}
=
2\sqrt{c_{\sin}}\,
\frac{N^{3/2}}{\Delta}\epsilon_N^\star(1+o(1)).
\]
The method uses the positive endpoint values
$
        p_0=1,
        \;
        p_1=\eta,
        \;
        p_{s+1}=\eta p_s-a_{s,N}^2p_{s-1}.
$
These are positive because they are leading principal characteristic polynomials evaluated at \(\eta>1\), with zeros in \([-1,1]\) by interlacing.  Initialize by
$
        x_1=x_0-\frac{2}{L+\mu}\nabla f(x_0),
$
and, for \(s=1,\ldots,N-1\), updated by
\[
        x_{s+1}
        =\frac{\eta p_s}{p_{s+1}}x_s
        -\frac{a_{s,N}^2p_{s-1}}{p_{s+1}}x_{s-1}
        -\frac{2p_s}{\Delta p_{s+1}}\nabla f(x_s).
\]

\noindent
\begin{minipage}[t]{0.49\textwidth}
\vspace{0pt}
\captionof{algorithm}{Prefix-exact Chebyshev recurrence}
\label{alg:prefix-chebyshev}
\begin{algorithmic}[1]
\Require initial point \(x_0\), gradient oracle \(\nabla f\), parameters \(0<\mu<L\), horizon \(N\ge1\)
\State \(\Delta\gets L-\mu\), \(\eta\gets (L+\mu)/\Delta\)
\State \(c_0\gets 1\), \(c_1\gets \eta\)
\For{\(t=1,\ldots,N\)}
    \State \(c_{t+1}\gets 2\eta c_t-c_{t-1}\)
\EndFor
\State \(x_1\gets x_0-\dfrac{2}{L+\mu}\nabla f(x_0)\)
\For{\(t=1,\ldots,N-1\)}
    \State \(\beta_t\gets c_{t-1}/c_{t+1}\)
    \State \(x_{t+1}\gets x_t+\beta_t(x_t-x_{t-1})\)
    \Statex \hspace{\algorithmicindent}
    \(\displaystyle -(1+\beta_t)\frac{2}{L+\mu}\nabla f(x_t)\)
\EndFor
\State \Return \(x_N\)
\end{algorithmic}
\end{minipage}
\hfill
\begin{minipage}[t]{0.49\textwidth}
\vspace{0pt}
\captionof{algorithm}{Sine-Jacobi Chebyshev method}
\label{alg:sine-jacobi}
\begin{algorithmic}[1]
\Require initial point \(x_0\), gradient oracle \(\nabla f\), parameters \(0<\mu<L\), horizon \(N\ge1\)
\State \(\Delta\gets L-\mu\), \(\eta\gets (L+\mu)/\Delta\)
\State \(p_0\gets 1\), \(p_1\gets \eta\)
\For{\(s=1,\ldots,N-1\)}
    \State \(\displaystyle
    a_s^2\gets \frac14
    \frac{\sin^2(\pi s/N)}
    {\sin(\pi(s-\frac12)/N)\sin(\pi(s+\frac12)/N)}
    \)
    \State \(p_{s+1}\gets \eta p_s-a_s^2p_{s-1}\)
\EndFor
\State \(x_1\gets x_0-\dfrac{2}{L+\mu}\nabla f(x_0)\)
\For{\(s=1,\ldots,N-1\)}
    \State \(\displaystyle
    x_{s+1}\gets
    \frac{\eta p_s}{p_{s+1}}x_s
    -\frac{a_s^2p_{s-1}}{p_{s+1}}x_{s-1}
    \)
    \Statex \hspace{\algorithmicindent}
    \(\displaystyle
    -\frac{2p_s}{\Delta p_{s+1}}\nabla f(x_s)
    \)
\EndFor
\State \Return \(x_N\)
\end{algorithmic}
\end{minipage}

The endpoint-profile theorem is the following.

\begin{theorem}[Sine-Jacobi endpoint profile]
\label{thm:sine-endpoint-profile}
For \(D_N(t):=d_{\lfloor tN\rfloor}/N\), \(0\le t<1\),
$
        D_N\to\varphi_{\sin}
        \;\text{in }L^2(0,1),
$
$
        \frac1{N^3}\sum_{s=0}^{N-1}d_s^2\to c_{\sin}.
$
\end{theorem}
\prooflink{thm:sine-endpoint-profile}

\begin{theorem}[Sine-Jacobi gain]
\label{thm:sine-upper}
The sine-Jacobi Chebyshev method is final-Chebyshev exact and satisfies the
finite identity
\[
        A_N(J_N^{\sin})
        =
        \frac{2\epsilon_N^\star}{\Delta}
        \left(\sum_{s=0}^{N-1}d_s^2\right)^{1/2}.
\]
Consequently,
\[
        A_N(J_N^{\sin})
        =
        2\sqrt{c_{\sin}}
        \frac{N^{3/2}}{\Delta}
        \epsilon_N^\star(1+o(1)).
\]
\end{theorem}
\prooflink{thm:sine-upper}

The same expansion controls the sine--Jacobi chain under finite drift.  Its
blocks inherit the positive trigonometric expansions of the Rogers
polynomials, so the sine chain has the same block property as the
prefix-exact recurrence.

\begin{corollary}[Finite drift for the sine--Jacobi chain]
\label{cor:sine-finite-drift}
Every block \(\widetilde D_{a:b}\) of \(J_N^{\sin}\) has nonnegative Fourier
coefficients in \(\theta\), \(x=\cos\theta\); in particular it attains its
maximum on \([-1,1]\) at \(x=1\).  Hence all bounds of
\Cref{thm:finite-drift}(iii) hold for the sine--Jacobi chain, with the same
constants as for the prefix-exact recurrence, and they are attained.
\end{corollary}
\prooflink{cor:sine-finite-drift}

Thus the prefix-exact and sine--Jacobi realizations have exactly the same
worst-case response to finite drift in operator norm, at every order: both
coincide with the static shift of the spectral interval.  They differ in how
the linear response accumulates in \(\ell_2\) over time and spectral pairs,
which is measured by \(A_N\) (\Cref{prop:finite-drift-linear}) and is smaller
for the sine--Jacobi chain (\Cref{thm:sine-upper}).

The drift bounds apply directly to smooth objectives.  For a method that
queries the gradient at the current iterate, the gradient at \(x_s\) equals
\(\bar H_s(x_s-x^\star)\) with the averaged Hessian
\(\bar H_s:=\int_0^1\nabla^2f(x^\star+t(x_s-x^\star))\,dt\), so the error after
\(N\) steps is exactly \(\TT_{\mathcal R}(\bar H_0,\ldots,\bar H_{N-1})\)
applied to \(x_0-x^\star\).

\begin{corollary}[Objectives with Lipschitz Hessian]
\label{cor:lipschitz-hessian}
Let \(f\) be twice differentiable on a convex set containing \(x^\star\) and
the iterates, with \(\nabla f(x^\star)=0\),
\(\mu I\preceq\nabla^2f(x^\star)\preceq LI\), and
\(\|\nabla^2f(x)-\nabla^2f(y)\|\le\rho\|x-y\|\).  Run the prefix-exact or the
sine--Jacobi method for \(N\) steps and suppose that
\(\|x_s-x^\star\|\le R\) for \(0\le s\le N-1\).  Then
\[
\|x_N-x^\star\|\le\epsilon_N^\star\,
T_N\Bigl(1+\frac{\rho R}{\Delta}\Bigr)\|x_0-x^\star\|.
\]
\end{corollary}
\prooflink{cor:lipschitz-hessian}

A Hessian that varies by \(\rho R\) over the region visited by the iterates
thus acts, for both methods and to all orders, like a widening of the
spectral interval by \(\rho R/2\) on each side. It changes the Chebyshev
factor \(\epsilon_N^\star\) by a constant factor once \(\rho R\) is of order
\(\Delta/N^2\), which is a fraction of \(\mu\) that is only logarithmically
small, as the following remark quantifies.

\begin{remark}[The drift scale is set by \(\mu\)]
\label{rem:drift-scale}
Let \(\kappa=L/\mu\).  Since \(T_N(\cosh t)=\cosh(Nt)\) and
\(\operatorname{arccosh}(1+2u)=2\operatorname{arcsinh}\sqrt u\), the common
bound of \Cref{cor:prefix-finite-drift,cor:sine-finite-drift} reads
\[
\epsilon_N^\star\,T_N\Bigl(1+\frac{2\delta}{\Delta}\Bigr)
=\frac{\cosh\bigl(2N\operatorname{arcsinh}\sqrt{\delta/\Delta}\bigr)}
{\cosh\bigl(2N\operatorname{arcsinh}\sqrt{\mu/\Delta}\bigr)} .
\]
For \(0\le\delta\le\mu\), concavity of \(\operatorname{arcsinh}\) gives
\(\operatorname{arcsinh}\sqrt{\mu/\Delta}-\operatorname{arcsinh}\sqrt{\delta/\Delta}
\ge(\sqrt\mu-\sqrt\delta)/\sqrt L\), and therefore, for both realizations,
\[
\bigl\|\TT(H+E_0,\ldots,H+E_{N-1})\bigr\|
\le2\exp\Bigl(-\frac{2N}{\sqrt\kappa}\Bigl(1-\sqrt{\delta/\mu}\Bigr)\Bigr).
\]
Drift of size \(\delta\) thus multiplies the accelerated exponent
\(2N/\sqrt\kappa\) by \(1-\sqrt{\delta/\mu}\), that is, it acts as the
condition number \(\kappa/(1-\sqrt{\delta/\mu})^2\).  By
\Cref{thm:finite-drift}(iv) this cannot be improved, and at \(\delta=\mu\)
the static shift reaches \(\lambda=0\), where every final-Chebyshev-exact
realization has residual \(1\).  The bounds therefore cover the whole range
in which acceleration survives.

The threshold \(\Delta/N^2\) is \(\mu\) in disguise.  Reaching accuracy
\(\epsilon\) with the factor \(\epsilon_N^\star\) takes
\(N\approx\frac{\sqrt\kappa}{2}\ln\frac2\epsilon\), and then
\(\Delta/N^2\approx4\mu/\ln^2(2/\epsilon)\); the drift amplifies the
terminal error by about \(\frac12(2/\epsilon)^{\sqrt{\delta/\mu}}\).  For
\(\kappa=10^4\) and \(N=725\), so that \(\epsilon_N^\star\approx10^{-6}\), the
exact amplification \(T_N(1+2\delta/\Delta)\) equals \(2.3\) for
\(\delta=0.01\mu=10^{-6}L\), \(13\) for \(\delta=0.05\mu\), and \(700\) for
\(\delta=0.25\mu\), and these values are attained by \(E_s=-\delta I\).  In
ill-conditioned problems, Hessian variation that is negligible relative to
\(L\), a small fraction of \(\mu\), thus changes the terminal error of an
accelerated method by orders of magnitude.  Shorter blocks, for which
\(\Delta/N^2\) is larger, are correspondingly more robust, consistent with
the shortened blocks accepted in the restart experiments of
Section~\ref{sec:experiments}.
\end{remark}

\section{Performance experiments}
\label{sec:experiments}

This section compares the prefix-exact and sine--Jacobi realizations while keeping the terminal Chebyshev transfer \(p_N^\star(\lambda)=T_N(z(\lambda))/T_N(\eta)\), \(z(\lambda)=(L+\mu-2\lambda)/(L-\mu)\), fixed.  The measured quantities are therefore invisible to the fixed-Hessian residual alone: time-ordered response, stochastic curvature overhead, finite-radius degradation, and restart cost.  All runs use Algorithms~\ref{alg:prefix-chebyshev} and~\ref{alg:sine-jacobi}.

The tests are: finite constants and stochastic frontiers on diagonal quadratics with one endpoint-coupled stochastic Hessian channel; time-varying two-dimensional quadratics for kernel validation; finite-radius/restart tests on an endpoint-coupled logistic regression GLM; and, in the appendix, a smooth log-cosh nonlinear-transfer family.

The source code is available at \url{https://github.com/dmivilensky/sine-Jacobi}.

\subsection{Finite constants and stochastic block admissibility}
\label{subsec:experiments-frontier}

First we convert the asymptotic constants into finite-horizon quantities:
\[
        C_N^{\rm pref}
        :=\frac{\Delta A_N^{\rm pref}}{N^{3/2}\epsilon_N^\star},
        \;
        C_N^{\rm sine}
        :=\frac{2}{N^{3/2}}
        \bigg(\sum_{s=0}^{N-1}\bigl[(I-J_N^{\sin})^{-1}\bigr]_{ss}^2\bigg)^{1/2},
\]
where \(\Delta=L-\mu\).  The prefix expression uses Theorem~\ref{thm:prefix-kernel}; the sine expression uses the finite Jacobi matrix of Section~\ref{sec:sine-upper-main}.  Figure~\ref{fig:exp-theory-constants} uses \((\mu,L)=(1,120)\).  The finite-\(N\) squared ratio is already close to its limit: at \(N=44\),
\[
        \left(C_{44}^{\rm sine}/C_{44}^{\rm pref}\right)^2=0.8711,
\]
and at \(N=96\) it is \(0.8636\).  Equivalently, at the same first-order stochastic overhead, the sine--Jacobi realization requires about \(12.9\%\) less batch budget at \(N=44\) and about \(13.6\%\) less at \(N=96\).

\begin{figure}[ht!]
    \centering
    \includegraphics[width=0.75\linewidth]{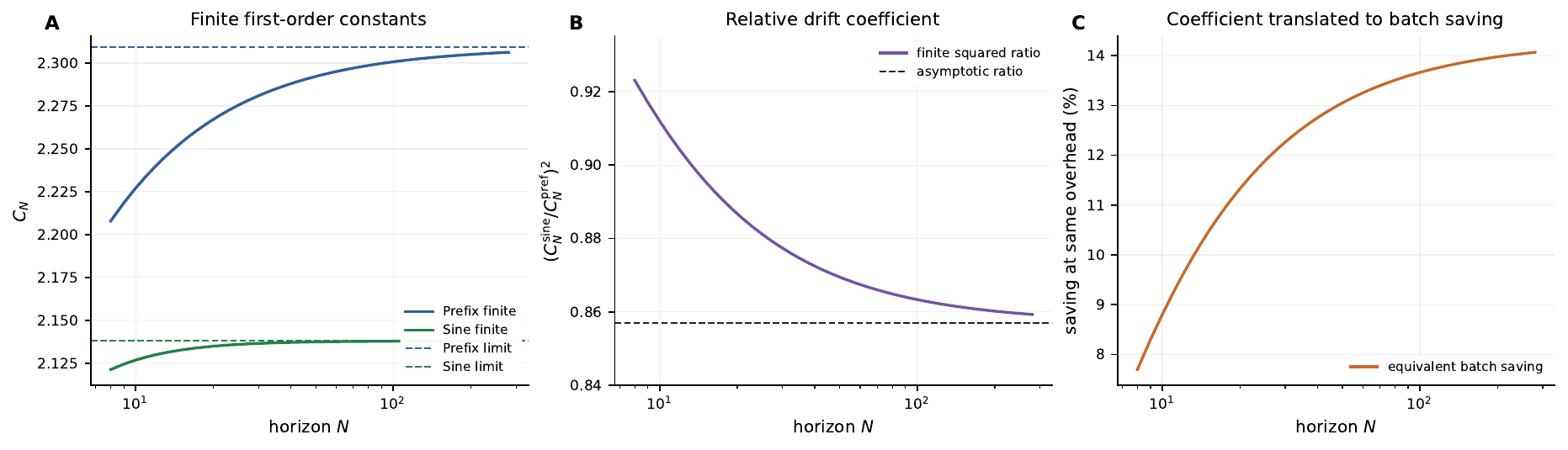}
    \caption{Finite first-order drift constants.  Panel A compares the finite constants \(C_N^{\rm pref}\) and \(C_N^{\rm sine}\) with their limits \(4/\sqrt3\) and \(2\sqrt{c_{\sin}}\).  Panel B shows the finite squared coefficient ratio.  Panel C rewrites the same coefficient ratio as the equivalent reduction in batch budget at a fixed first-order overhead.}
    \label{fig:exp-theory-constants}
\end{figure}

We next turn the same coefficient comparison into block selection for
\(f_H(x)=\frac12 x^\top Hx,\; H=\operatorname{diag}(1,\sqrt{120},120)\).
The minimizer is fixed; all degradation comes from multiplicative Hessian-vector perturbations.  The diagonal Hessian has eigenvalues \(1,\sqrt{120},120\), and the endpoint-coupled oracle is
\[
        g_t(x)=Hx+\frac{\sigma}{\sqrt b}z_t Bx,
        \;
        B=e_1e_3^\top+e_3e_1^\top,
        \; z_t\sim N(0,1),
\]
with \(\sigma=2\).  The stochastic predictions below assume independent
draws \(z_t\) across oracle times.  The initial vector is proportional to \((1,0.25,1)\), so all three eigendirections are present while the noise couples only the endpoints.  For each \(b\), \(N\in\{16,18,\ldots,100\}\), and realization \(R\), we estimate
\[
        \mathrm{oh}_R(N,b)
        :=\mathbb E\frac{\|x_N^R\|^2}{\|x_N^R\|^2_{\delta=0}}-1
\]
from \(900\) common-random-number trials, paired between prefix and sine at each \((N,b)\).  The admissibility rule is
\(\mathrm{oh}_R(N,b)\le 0.25\).
Among admissible horizons, the selected horizon maximizes the empirical per-gradient log decrease.

\begin{figure}[ht!]
    \centering
    \includegraphics[width=0.75\linewidth]{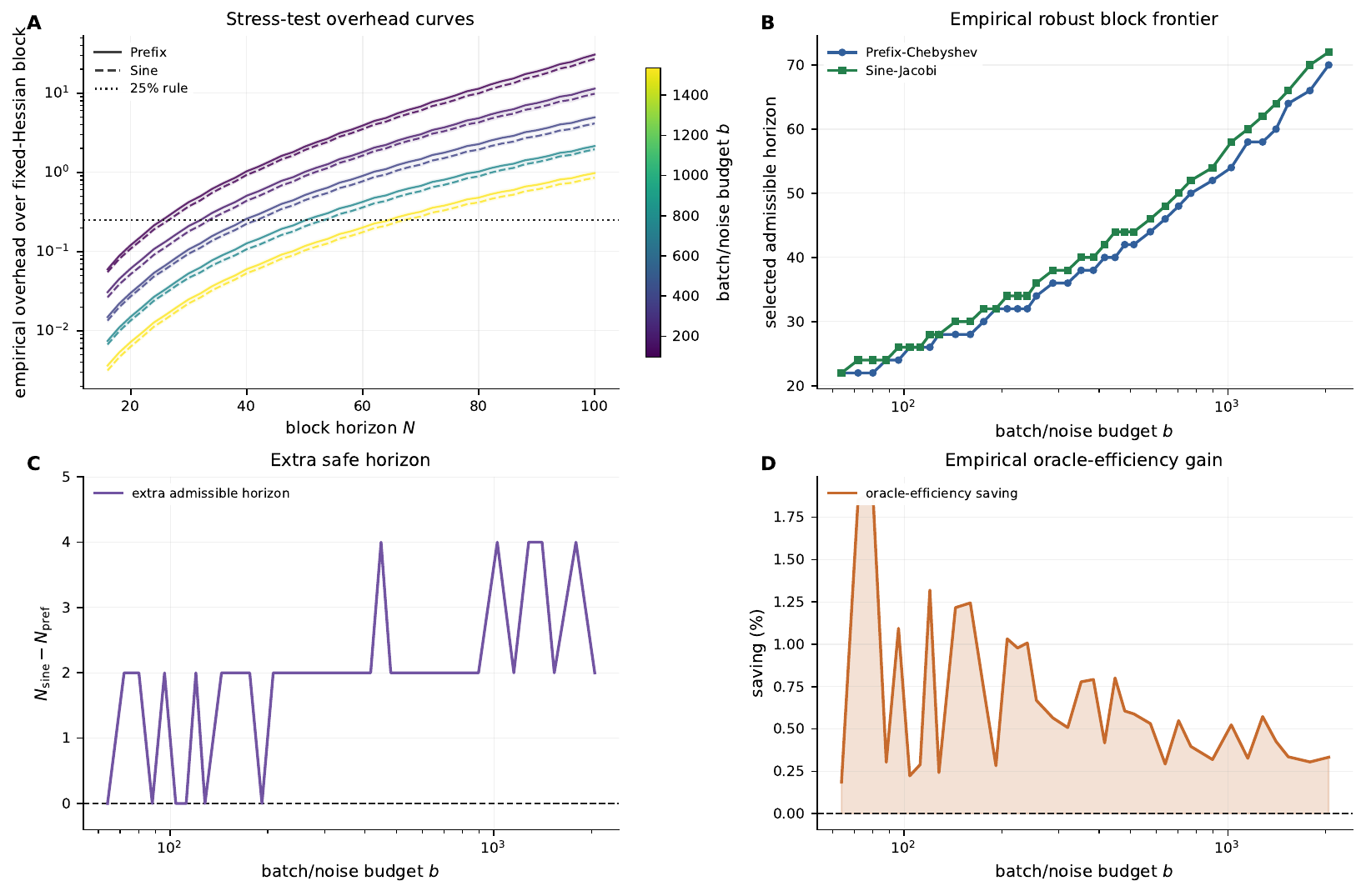}
    \caption{Empirical stochastic block frontier for the endpoint-coupled multiplicative curvature oracle.  Panel A reports the measured overhead curves \(\mathrm{oh}_R(N,b)\) for several batch/noise budgets.  Panel B gives the empirically selected admissible horizon under the \(25\%\) overhead rule.  Panel C shows the extra admissible horizon obtained by the sine--Jacobi realization.  Panel D reports the corresponding oracle-efficiency gain on the selected frontier.}
    \label{fig:exp-optimization-frontier}
\end{figure}

The empirical frontier follows the finite-constant calculation.  Across the tested batch grid, the median extra admissible horizon is \(2\), and the maximum extra admissible horizon is \(4\).  The per-gradient efficiency improvement is smaller than the coefficient reduction because the chosen horizon is integer-valued and because the two methods have the same fixed-Hessian Chebyshev residual once a horizon is fixed.  The gain therefore appears mainly as the ability to use the next admissible Chebyshev block before violating the curvature-noise overhead constraint.

\subsection{Kernel-level validation on time-varying quadratics}
\label{subsec:experiments-tvq}

To test the kernel directly, use the non-autonomous quadratic sequence
\(f_s(x)=\frac12 x^\top H_s x\),
where the Hessian changes at the same times at which the gradient oracle is queried.  The model is two-dimensional,
\[
        H_s=\begin{pmatrix}\lambda & a_s\\ a_s & \nu\end{pmatrix},
        \; \lambda=200,
        \; \nu=8000,
        \; x_0=(0,1).
\]
Both realizations are calibrated on \([\mu,L]=[\lambda,\nu]=[200,8000]\), so the terminal Chebyshev polynomial is identical.  The off-diagonal perturbation is a random-sign packet
\(a_s=\rho p_s(c,h)\xi_s, \; \rho=180, \; \xi_s\in\{-1,1\}\),
where \(p(c,h)\) is a normalized Gaussian time profile.  For the predictor,
assume that the signs are independent and that each sign is equally likely.
We use \(N=96\), a grid of \(49\) centers and \(9\) widths, and \(384\) sign
samples per packet.  The exact first-order RMS predictor is
\[
        \mathrm{RMS}_{\rm pred}^R(c,h)
        =\rho\Big(\sum_{s=0}^{N-1}p_s(c,h)^2|K_s^R(\lambda,\nu)|^2\Big)^{1/2}.
\]
The measured response is the RMS terminal first coordinate, normalized by
\(\rho\), and it is compared with \(\mathrm{RMS}_{\rm pred}^R(c,h)/\rho\).
Since the unperturbed start is \((0,1)\), the unperturbed terminal first
coordinate vanishes.  The kernel predicts the first-order transverse
response; the finite-\(\rho\) measurement also includes higher insertions.
Here \([\mu,L]\) is the calibration interval for the unperturbed diagonal
Hessian, not a claimed spectral enclosure for every perturbed \(H_s\).

\begin{figure}[ht!]
    \centering
    \includegraphics[width=0.75\linewidth]{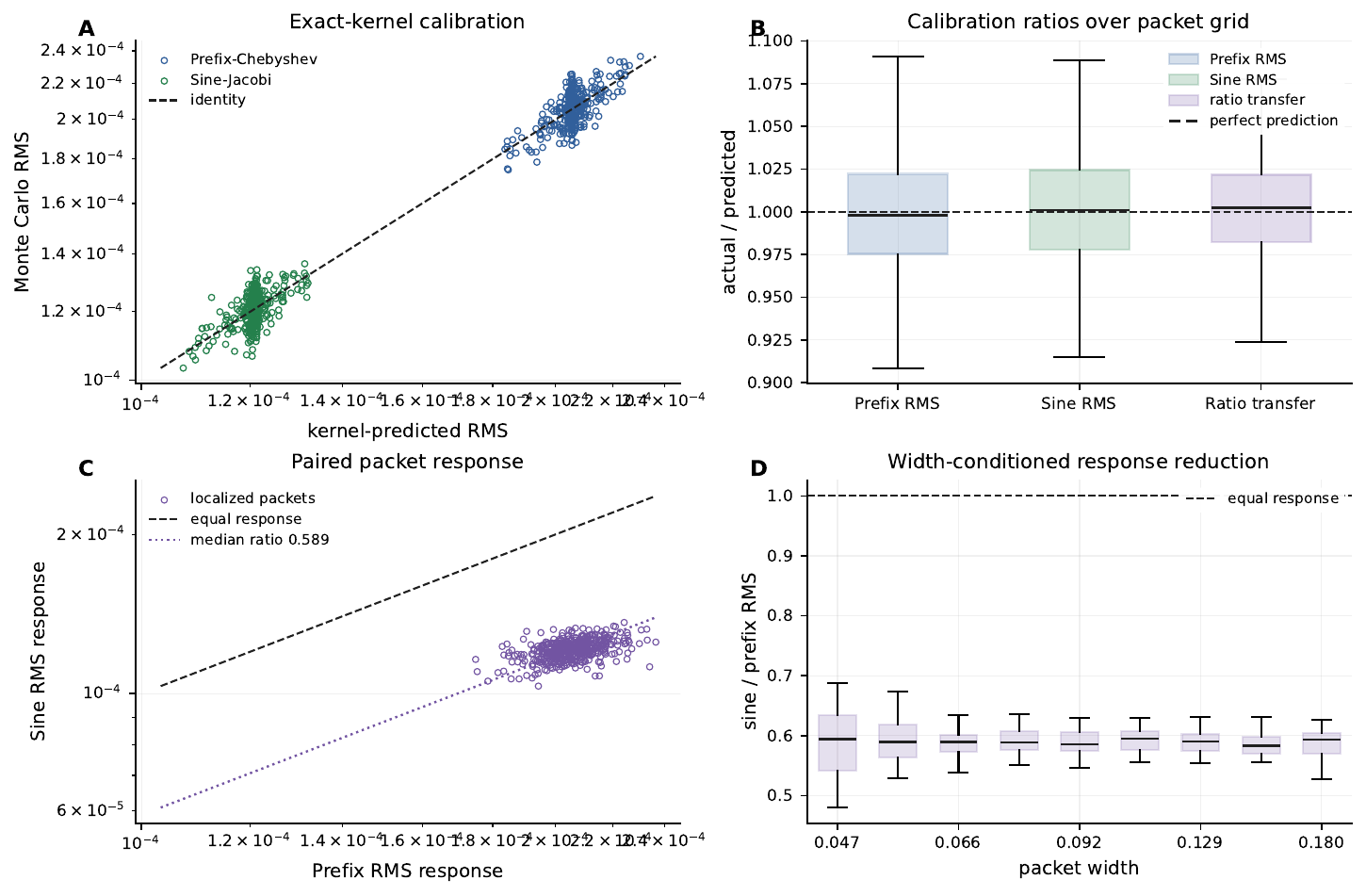}
    \caption{Time-varying quadratic validation of the first-variation kernel.  Panel A compares the first-order RMS predictor with the Monte Carlo response, both normalized by \(\rho\).  Panel B summarizes calibration ratios over the packet grid.  Panel C compares the paired prefix and sine packet responses.  Panel D groups the sine/prefix RMS response ratio by packet width.}
    \label{fig:exp-tvq-kernel}
\end{figure}

Figure~\ref{fig:exp-tvq-kernel} shows quantitative agreement.  Median actual/predicted ratios are \(0.998\) for prefix, \(1.001\) for sine, and \(1.002\) for the transfer ratio; the central \(10\)--\(90\%\) ranges are \([0.958,1.044]\), \([0.958,1.045]\), and \([0.968,1.040]\).  The paired comparison is almost entirely below the diagonal: the median sine/prefix RMS ratio is \(0.589\), with \(10\)--\(90\%\) range \([0.557,0.624]\), and width-conditioned boxes stay near \(0.59\).

\subsection{Stochastic curvature at a fixed horizon}
\label{subsec:experiments-stochastic}

At fixed horizon \(N=44\), using the same endpoint-coupled oracle, the overhead is
\(\mathrm{oh}_R(b) =\mathbb E\frac{\|x_N^R\|^2}{\|x_N^R\|^2_{\delta=0}}-1\).
The tail fit uses the pre-specified large-batch set
\(b\in\{1280,1408,1536,1792,2048\}\)
with through-origin fit \(\mathrm{oh}_R(b)\simeq a_N^R/b\).  The fitted coefficients are
\[
        a_{44}^{\rm pref}=125.53\pm 2.42,
        \;
        a_{44}^{\rm sine}=108.14\pm 1.25,
\]
where the intervals are \(95\%\) standard-error intervals from the design-tail fit.  The fitted ratio is \(0.861\), close to the finite first-order squared ratio \(0.871\) in Figure~\ref{fig:exp-theory-constants}.

\begin{figure}[ht!]
    \centering
    \includegraphics[width=0.75\linewidth]{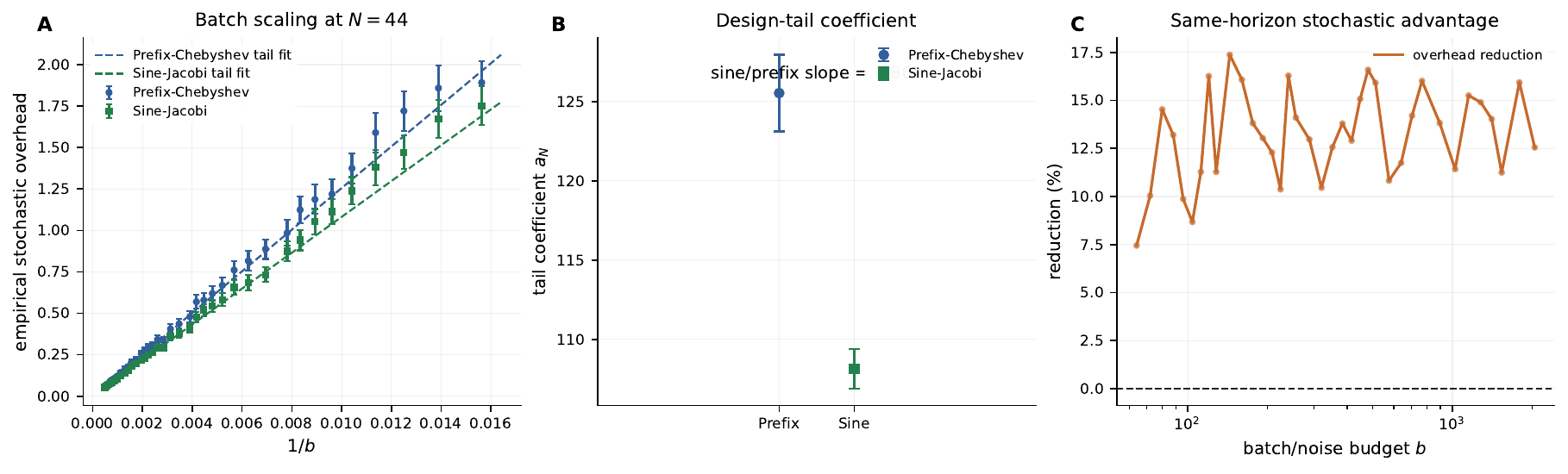}
    \caption{Fixed-horizon stochastic curvature response at \(N=44\).  Panel A plots the measured overhead against \(1/b\) together with the pre-specified design-tail fits.  Panel B compares the fitted tail coefficients.  Panel C shows the same-horizon overhead reduction as a function of the batch/noise budget.}
    \label{fig:exp-stochastic}
\end{figure}

This noisy first-variation test has leading squared-error contribution proportional to \(1/b\).  The coefficient ordering matches the theory: at the same Chebyshev factor and \(N=44\), sine--Jacobi has the smaller time-ordered curvature coefficient.

\subsection{Finite-radius logistic GLM and restarted blocks}
\label{subsec:experiments-glm}

The deterministic finite-radius test uses the regularized logistic objective
\[
        f(x)=\frac1n\sum_{i=1}^n\log\bigl(1+\exp(-y_i a_i^\top x)\bigr)
        +\frac{\lambda_{\rm reg}}2\|x\|^2,\; \lambda_{\rm reg}=0.003 .
\]
The design has anisotropic Gaussian covariance with target condition number \(300\), augmented by endpoint-coupling structured rows.  Labels come from a logistic model with signal direction proportional to \(q_1+0.85q_d+0.35q_{d/2}\).  The instance is smooth and strongly convex, with Hessian derivative near \(x_\star\) coupling the spectral endpoints; after augmentation, \(d=32\) and \(n=2296\).  At the minimizer,
\(\mu=0.015231, \; L=1.4702, \; \kappa=96.53, \; \Delta=1.455\).
The horizon is \(N=36\), with \(\epsilon_N^\star=1.28026\cdot 10^{-3}\).  Starting points are
\(x_0=x_\star+r u_j\),
where the \(64\) directions \(u_j\) combine low-, high-, and middle-curvature eigendirections of \(\nabla^2 f(x_\star)\), plus a small random component.  The block degradation statistic is
\(D_R(r,u_j) :=\frac{\|x_N^R-x_\star\|}{r\epsilon_N^\star}\).
A direction is accepted if \(D_R(r,u_j)\le 2\).  The radius grid contains \(96\) values between \(3\cdot10^{-3}\) and \(1\).

\begin{figure}[ht!]
    \centering
    \includegraphics[width=0.75\linewidth]{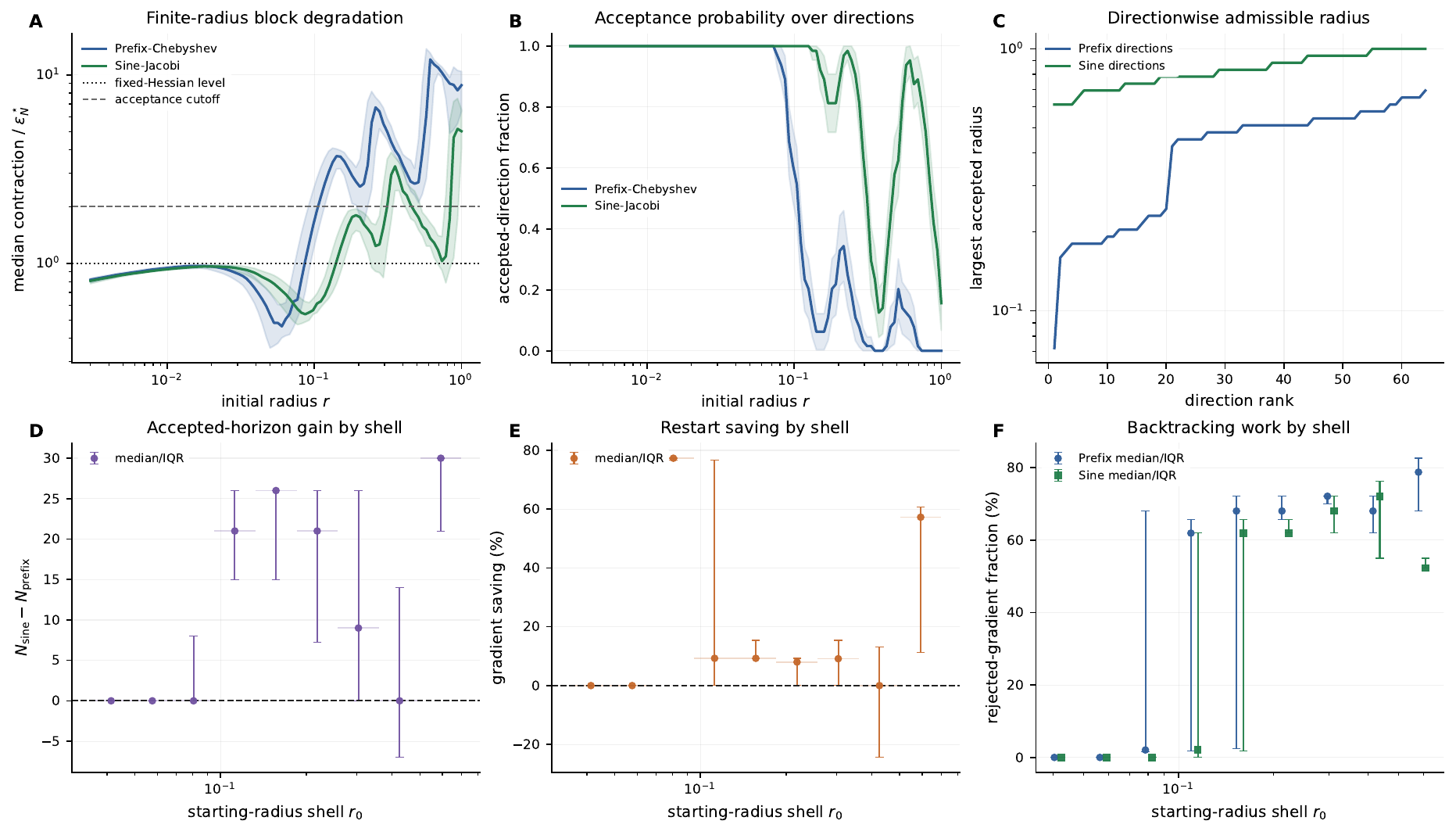}
    \caption{Finite-radius logistic GLM and restarted blocks.  Panels A--C show one-block behavior over endpoint-coupled directions: normalized contraction, accepted-direction fraction, and directionwise admissible radius.  Panels D--F summarize restarted blocks over logarithmic initial-radius bins.  Each point is the median over all tested starting radii and endpoint directions in the bin; vertical bars show the interquartile range, and horizontal ticks show the bin width.}
    \label{fig:exp-glm}
\end{figure}

The one-block finite-radius GLM test shows a separation between the two realizations.  At \(r\approx0.0979\), the prefix median normalized contraction is \(1.58\) and the accepted-direction fraction is \(0.609\), while the sine median is \(0.563\) and all tested directions are accepted.  At \(r\approx0.150\), the prefix accepted fraction drops to \(0.0625\), while the sine accepted fraction remains \(0.922\).  The median largest accepted radius over directions is \(0.495\) for prefix and \(0.832\) for sine, a factor \(1.68\).

The nonmonotone shape in Panels A and B is consistent with the deterministic finite-radius expansion.  With \(H_\star=\nabla^2f(x_\star)\), the terminal block error expands as
\[
        x_N^R(x_\star+r u)-x_\star
        =r p_N^\star(H_\star)u+r^2 Z_R(u)+O(r^3),
        \; r\downarrow0.
\]
To specify the second-order term, write
\(DH_\star[v]:=D(\nabla^2 f)(x_\star)[v]\), and let
\(u_s^R\) be the first-order error trajectory initialized at \(u\).
Both realizations query the gradient at the current iterate.  The Taylor
expansion
\[
\nabla f(x_\star+e)
=H_\star e+\tfrac12 DH_\star[e]e+O(\|e\|^3)
\]
therefore gives
\[
Z_R(u)=\frac12\sum_{s=0}^{N-1}
\mathcal L_{s,H_\star}^R(DH_\star[u_s^R])u,
\]
where \(\mathcal L_{s,H_\star}^R\) is the insertion operator of
\Cref{lem:first-order-expansion}.  In particular,
\[
        \frac{\|x_N^R(x_\star+r u)-x_\star\|}{r}
        =\|a(u)+r b_R(u)+O(r^2)\|,
        \quad a(u)=p_N^\star(H_\star)u,\quad b_R(u)=Z_R(u).
\]
When \(a(u)\ne0\), the right derivative at the origin is
\(\langle a(u),b_R(u)\rangle/\|a(u)\|\), whose sign depends on direction
and phase.  When \(a(u)=0\), the right derivative is instead \(\|b_R(u)\|\).
Thus no common monotonicity in \(r\) is implied.  As \(r\) varies,
directions can enter or leave the acceptance set
\(\|a(u_j)+r b_R(u_j)+O(r^2)\|\le 2\epsilon_N^\star\).
Endpoint-coupled logistic rows couple low and high curvature.  In this
experiment the sine--Jacobi realization exhibits a larger accepted-radius
range; a smaller \(A_N\) alone does not imply such an ordering for every
nonlinear trajectory.  The logistic loss has a Lipschitz Hessian, so
\Cref{cor:lipschitz-hessian} applies to this test: for iterates within
distance \(R\) of the minimizer it gives the same worst-case bound
\(\epsilon_N^\star T_N(1+\rho R/\Delta)\) for both realizations.  The separation observed here is therefore not a worst-case
effect: it is a typical-direction effect, for which the smaller
\(\ell_2\)-in-time gain \(A_N\) of the sine--Jacobi chain is the relevant
quantity.

The restart test uses the same GLM and directions, with radii denser near the transition \(r_0\simeq0.1\).  The solver uses \(N_{\max}=44\), \(\eta_{\rm acc}=0.23\), and target radius \(2\cdot10^{-5}\); a trial is accepted if the actual contraction is at most \((1+4\eta_{\rm acc})\epsilon_N^\star\), otherwise the candidate block length is reduced.  Panels D--F report medians and interquartile ranges over fixed logarithmic initial-radius shells.  In the shell containing \(r_0\simeq0.1\), the median first accepted horizon gain is \(21\), median gradient saving is about \(9.3\%\), and the median rejected-gradient fraction drops from \(61.9\%\) for prefix to \(2.1\%\) for sine.  At larger radii both methods backtrack, but sine continues to admit longer blocks.

Overall, with the terminal Chebyshev polynomial fixed, sine--Jacobi admits larger stochastic horizons, has smaller fitted noise coefficients, matches the time-varying quadratic kernel prediction, and gives larger GLM admissible radii with less rejected-gradient work.  The differences therefore come from the ordered realization, not from a different fixed-Hessian residual.

\section*{Conclusion}

The first-order response of a Chebyshev method to Hessian drift is a property of the ordered realization, not only of the terminal polynomial.  The kernel \(K_s(\lambda,\nu)\) and its spectral \(\ell_2\) norm \(A_N\) give the sharp finite-horizon insertion gain.

For the prefix-exact Chebyshev recurrence,
\[
        A_N^{\rm pref}
        =\frac{\epsilon_N^\star}{L-\mu}
        \left(4N^2+16\sum_{m=1}^{N-1}m^2\right)^{1/2}
        =\frac4{\sqrt3}\frac{N^{3/2}}{L-\mu}\epsilon_N^\star(1+o(1)),
\]
and this is optimal in the causal two-term prefix-exact class.  In the larger terminal-exact class, Jacobi coordinates fix the terminal spectrum but leave spectral weights free.  The sine weights preserve the same terminal residual and yield \(2\sqrt{c_{\sin}}\approx2.137936<4/\sqrt3\). Thus the terminal Chebyshev polynomial determines the fixed-Hessian minimax transfer, but not the Hessian-drift gain of its realizations.  Momentum is needed for the order \(N^{3/2}\): every final-Chebyshev-exact stepsize schedule has gain at least of order \(N^2\epsilon_N^\star/(L-\mu)\).

Beyond first order, the response to Hessian drift is a finite polynomial expansion.  For both the prefix-exact and the sine--Jacobi realizations, drift of size \(\delta\) in operator norm is exactly as harmful as the static widening of \([\mu,L]\) by \(\delta\): the terminal error is at most \(\epsilon_N^\star T_N(1+2\delta/(L-\mu))\), this bound is attained, and it reduces the accelerated rate exponent by the factor \(1-\sqrt{\delta/\mu}\).  The same bound holds for objectives with Lipschitz Hessian along the trajectory.  Drift that is a small fraction of \(\mu\), negligible relative to \(L\) in ill-conditioned problems, therefore matters, and the realizations differ in how it accumulates in \(\ell_2\) over time, which is what \(A_N\) measures.

The experiments show the same separation at finite horizons: smaller stochastic overhead coefficients, larger admissible-block frontiers, accurate time-varying quadratic predictions, larger GLM admissible radii, longer accepted restarts, and less rejected-gradient work, all with the terminal residual held fixed.

% \section*{Acknowledgments}

% We thank Slavom\'ir Tak\'a\v{c} for sharing two unpublished manuscripts
% co-authored with Martin Tak\'a\v{c}. The earlier one, ``On Optimal First Order Methods for Quadratic Objective Function'',
% studies fixed-coefficient linear iterative algorithms for quadratics and
% derives a $\sqrt{\kappa}$ contraction bound that holds uniformly over
% the iteration index, in contrast to the finite-dimensional analysis of
% \cite{nesterov2004introductory}; it further shows that a
% single-sequence gradient method with non-monotone, time-dependent step
% sizes already attains this bound, i.e.\ outside the stationarity
% assumption required by \cite{arjevani2015lower} for their $1$-linear
% lower bound. The later one, ``Acceleration by Step-Size Scheduling: Exact Quadratic Theory and Local Perturbations'', places the same
% schedule inside a finite-horizon minimax problem, separates
% terminal-exact from prefix-exact realizations of a fixed Chebyshev
% filter, and raises the question of how a chosen realization responds to
% Hessian variation, with a local Hessian--Lipschitz perturbation bound
% and a restart-based experimental template. Both drafts shaped the
% direction of the present work.

\appendix

\section{Proofs for Section~\ref{sec:problem-setting}}
\label{app:proofs-sec1}

This appendix proves the first-order expansion used to define the drift gain \(A_N\), the finite-drift identity of \Cref{prop:finite-drift-linear}, and the lower bound for exact stepsize schedules.

\subsection{First-order Hessian-drift gain}

\begin{proof}[Proof of \Cref{lem:first-order-expansion}]\refstepcounter{prooflink}\label{proof:lem:first-order-expansion}
For an operator \(A\in\End(X)\), set
\[
\mathsf C_s(A)
:=
P_s\otimes I_X-g_sq_s^\top\otimes A
\quad
\text{on } \R^m\otimes X.
\]
Thus, for \(H_s(\delta):=H+\delta B_s\),
\[
\mathsf C_s(H_s(\delta))
=
\mathsf C_s(H)-\delta\,\mathsf G_s(B_s),
\;
\mathsf G_s(B_s):=g_sq_s^\top\otimes B_s.
\]
The output operator corresponding to the time-varying Hessian sequence is
\[
\TT_{\mathcal R}(\delta)
=
(\ell^\top\otimes I_X)
\mathsf C_{N-1}(H_{N-1}(\delta))\cdots
\mathsf C_0(H_0(\delta))
(b\otimes I_X),
\]
where \(b\otimes I_X:X\to \R^m\otimes X\) denotes \(x\mapsto b\otimes x\).

Since every factor is affine in \(\delta\), the product is differentiable at \(\delta=0\).  Applying the product rule to the ordered product gives
\[
\begin{aligned}
\frac{d}{d\delta}\bigg|_{\delta=0}
&
\mathsf C_{N-1}(H_{N-1}(\delta))\cdots
\mathsf C_0(H_0(\delta))
\\
&=
-\sum_{s=0}^{N-1}
\mathsf C_{N-1}(H)\cdots \mathsf C_{s+1}(H)\,
\mathsf G_s(B_s)\,
\mathsf C_{s-1}(H)\cdots \mathsf C_0(H),
\end{aligned}
\]
with the convention that an empty product is the identity on \(\R^m\otimes X\).  Therefore
\[
D\TT_{\mathcal R,H}[B_0,\ldots,B_{N-1}]
=
\sum_{s=0}^{N-1}\mathcal L_{s,H}(B_s),
\]
where
\[
\mathcal L_{s,H}(B_s)
=
-(\ell^\top\otimes I_X)
\mathsf C_{N-1}(H)\cdots \mathsf C_{s+1}(H)\,
(g_sq_s^\top\otimes B_s)\,
\mathsf C_{s-1}(H)\cdots \mathsf C_0(H)
(b\otimes I_X).
\]

It remains to identify the spectral blocks of \(\mathcal L_{s,H}\).  Let \(\Pi_\lambda\) be the orthogonal projector onto the eigenspace \(X_\lambda\) of \(H\).  If \(x\in X_\rho\), then
\[
\mathsf C_j(H)(y\otimes x)
=
(P_jy)\otimes x-(g_jq_j^\top y)\otimes Hx
=
(P_j-\rho g_jq_j^\top)y\otimes x
=
C_j(\rho)y\otimes x.
\]
Hence \(\mathsf C_j(H)\) preserves each fiber \(\R^m\otimes X_\rho\), and its restriction to that fiber is \(C_j(\rho)\otimes I_{X_\rho}\).

Fix two eigenvalues \(\lambda,\nu\) and take \(x\in X_\nu\).  The right part of the product gives
\[
\mathsf C_{s-1}(H)\cdots \mathsf C_0(H)(b\otimes x)
=
\bigl(C_{s-1}(\nu)\cdots C_0(\nu)b\bigr)\otimes x.
\]
After the insertion,
\[
(g_sq_s^\top\otimes B_s)
\left[
\bigl(C_{s-1}(\nu)\cdots C_0(\nu)b\bigr)\otimes x
\right]
=
\bigl(g_sq_s^\top C_{s-1}(\nu)\cdots C_0(\nu)b\bigr)\otimes B_sx.
\]
Projecting to the \(\lambda\)-fiber replaces \(B_sx\) by \(\Pi_\lambda B_sx\).  The left part of the product then acts with spectral parameter \(\lambda\):
\[
\begin{aligned}
&
\mathsf C_{N-1}(H)\cdots \mathsf C_{s+1}(H)
\left[
\bigl(g_sq_s^\top C_{s-1}(\nu)\cdots C_0(\nu)b\bigr)
\otimes \Pi_\lambda B_sx
\right]
\\
&\; =
\left[
C_{N-1}(\lambda)\cdots C_{s+1}(\lambda)
g_sq_s^\top
C_{s-1}(\nu)\cdots C_0(\nu)b
\right]
\otimes \Pi_\lambda B_sx.
\end{aligned}
\]
Finally, applying \(\ell^\top\otimes I_X\) gives
\[
\Pi_\lambda\mathcal L_{s,H}(B_s)x
=
-\ell^\top
C_{N-1}(\lambda)\cdots C_{s+1}(\lambda)
g_sq_s^\top
C_{s-1}(\nu)\cdots C_0(\nu)b
\,
\Pi_\lambda B_sx.
\]
By the definition of \(K_{s,\mathcal R}(\lambda,\nu)\), this is exactly
\[
\Pi_\lambda\mathcal L_{s,H}(B_s)\Pi_\nu
=
K_{s,\mathcal R}(\lambda,\nu)\,
\Pi_\lambda B_s\Pi_\nu.
\]
Equivalently, on the spectral pair \((\lambda,\nu)\),
\[
\mathcal L_{s,H}(B_s)_{\lambda\nu}
=
K_{s,\mathcal R}(\lambda,\nu)(B_s)_{\lambda\nu}.
\]
This proves the asserted first-order expansion.
\end{proof}

\subsection{Finite drift and the linear response}

\begin{proof}[Proof of \Cref{prop:finite-drift-linear}]\refstepcounter{prooflink}\label{proof:prop:finite-drift-linear}
With \(\mathsf C_s\) and \(\mathsf G_s\) as in the proof of
\Cref{lem:first-order-expansion},
\(\mathsf C_s(H+E_s)=\mathsf C_s(H)-\mathsf G_s(E_s)\) is affine in \(E_s\).
Expanding the ordered product,
\[
\TT_{\mathcal R}(H+E_0,\ldots,H+E_{N-1})
=\sum_{S\subseteq\{0,\ldots,N-1\}}(-1)^{|S|}
(\ell^\top\otimes I_X)\,
\mathsf F_{N-1}^{S}\cdots\mathsf F_0^{S}\,(b\otimes I_X),
\]
where \(\mathsf F_s^{S}=\mathsf G_s(E_s)\) for \(s\in S\) and
\(\mathsf F_s^{S}=\mathsf C_s(H)\) for \(s\notin S\).  The term with
\(S=\varnothing\) is \(\TT_{\mathcal R}(H,\ldots,H)=p_{\mathcal R,N}(H)\).  The
terms with \(|S|=1\) form the part linear in \((E_s)\); by
\Cref{lem:first-order-expansion} with \(\delta B_s=E_s\), it equals
\(D\TT_{\mathcal R,H}[E_0,\ldots,E_{N-1}]\).  The term indexed by \(S\) is
\(|S|\)-linear, which gives the remainder \(\mathcal Q_{\mathcal R,H}\).

Let \(\Pi_\lambda\) be the spectral projectors of \(H\).  By
\Cref{lem:first-order-expansion},
\(\Pi_\lambda D\TT_{\mathcal R,H}[E]\Pi_\nu
=\sum_sK_{s,\mathcal R}(\lambda,\nu)\Pi_\lambda E_s\Pi_\nu\).  The
Cauchy--Schwarz inequality and the definition of \(A_N\) give
\[
\bigl\|\Pi_\lambda D\TT_{\mathcal R,H}[E]\Pi_\nu\bigr\|_{\HS}^2
\le\sum_s K_{s,\mathcal R}(\lambda,\nu)^2
\sum_s\|\Pi_\lambda E_s\Pi_\nu\|_{\HS}^2
\le A_N(\mathcal R)^2\sum_s\|\Pi_\lambda E_s\Pi_\nu\|_{\HS}^2.
\]
Summing over all pairs \((\lambda,\nu)\) and using
\(\sum_{\lambda,\nu}\|\Pi_\lambda E\Pi_\nu\|_{\HS}^2=\|E\|_{\HS}^2\) proves the
bound.

For sharpness, the kernel is continuous on \([\mu,L]^2\), so the supremum
defining \(A_N(\mathcal R)\) is attained at some \((\lambda_*,\nu_*)\).  Take
\(X=\R^2\), \(H=\diag(\lambda_*,\nu_*)\), and
\(E_s=K_{s,\mathcal R}(\lambda_*,\nu_*)\,e_1e_2^\top\).  Then the \((1,2)\)
entry of \(D\TT_{\mathcal R,H}[E]\) is
\(\sum_sK_{s,\mathcal R}(\lambda_*,\nu_*)^2=A_N(\mathcal R)^2\), while
\((\sum_s\|E_s\|_{\HS}^2)^{1/2}=A_N(\mathcal R)\).  Hence equality holds.
\end{proof}

\begin{proof}[Proof of \Cref{thm:stepsize-lower}]\refstepcounter{prooflink}\label{proof:thm:stepsize-lower}
For gradient descent, \(C_s(\lambda)=1-\gamma_s\lambda\), and the kernel is
\[
K_s(\lambda,\nu)
=-\gamma_s\prod_{i=s+1}^{N-1}(1-\gamma_i\lambda)
\prod_{i=0}^{s-1}(1-\gamma_i\nu).
\]
Both \(\prod_s(1-\gamma_s\lambda)\) and \(p_N^\star\) have degree \(N\) and
value \(1\) at \(\lambda=0\), and \(p_N^\star\) has the simple roots
\(\lambda_j\); hence they coincide exactly when
\(\{\gamma_s^{-1}\}=\{\lambda_j\}\) as multisets.  Let
\(\gamma_s^{-1}=\lambda_{j(s)}\).  Since
\(\lambda_j-\mu=\Delta(1-\zeta_j)/2>0\), the factor
\(1-\gamma_s\mu\) does not vanish, and at \(\lambda=\nu=\mu\)
\[
K_s(\mu,\mu)
=-\frac{\gamma_s}{1-\gamma_s\mu}\,p_N^\star(\mu)
=-\frac{p_N^\star(\mu)}{\lambda_{j(s)}-\mu}
=-\frac{2\epsilon_N^\star}{\Delta(1-\zeta_{j(s)})},
\]
because \(z(\mu)=1\) and \(p_N^\star(\mu)=T_N(1)/T_N(\eta)=\epsilon_N^\star\).
Therefore
\[
\sum_{s=0}^{N-1}K_s(\mu,\mu)^2
=\frac{4(\epsilon_N^\star)^2}{\Delta^2}
\sum_{j=1}^N\frac1{(1-\zeta_j)^2},
\]
independently of the order of the stepsizes.

Since \(T_N(z)=2^{N-1}\prod_j(z-\zeta_j)\) and \(T_N(1)=1\),
\[
\sum_{j=1}^N\frac1{(1-\zeta_j)^2}
=-\frac{d}{dz}\frac{T_N'(z)}{T_N(z)}\Big|_{z=1}
=T_N'(1)^2-T_N''(1).
\]
The Chebyshev equation \((1-z^2)T_N''-zT_N'+N^2T_N=0\) at \(z=1\) gives
\(T_N'(1)=N^2\).  Its derivative,
\((1-z^2)T_N'''-3zT_N''+(N^2-1)T_N'=0\), at \(z=1\) gives
\(T_N''(1)=N^2(N^2-1)/3\).  Hence
\[
\sum_{j=1}^N\frac1{(1-\zeta_j)^2}
=N^4-\frac{N^2(N^2-1)}{3}
=\frac{2N^4+N^2}{3}.
\]
The gain \(A_N\) is a supremum over \(\lambda,\nu\in[\mu,L]\), so it is at
least the value at \((\mu,\mu)\).  This proves the bound; the asymptotic
form follows from \(\sqrt{2N^4+N^2}=\sqrt2\,N^2(1+O(N^{-2}))\).
\end{proof}

\section{Proofs for Section~\ref{sec:cheb-upper}}
\label{app:proofs-sec2}

This appendix contains the proofs of the prefix-exact recurrence identities,
the prefix first-variation kernel, and the Chebyshev upper-bound statements.
For completeness, we include the proof of the prefix-exact recurrence in the normalization used for the first-variation calculation.

\subsection{Prefix-exact recurrence}

\begin{proof}[Proof of \Cref{thm:prefix-exact-chebyshev-recurrence}]\refstepcounter{prooflink}\label{proof:thm:prefix-exact-chebyshev-recurrence}
Since \(\spec(H)\subset[\mu,L]\) and \(0<\mu<L\), the operator \(H\) is
positive definite.  Thus the minimizer of
\(f(x)=\frac12\langle Hx,x\rangle\)
is \(x^\star=0\).  We write
\(e_t:=x_t-x^\star=x_t\).
By the spectral theorem, it is enough to prove the claimed identity on each
eigenspace of \(H\).  Fix an eigenvalue \(\lambda\in\spec(H)\), and let
\(e_t(\lambda)\) denote the corresponding scalar spectral component.  On this
component the recurrence becomes
\[
e_1(\lambda)
=
\left(1-\frac{2\lambda}{L+\mu}\right)e_0(\lambda)
\]
and, for \(t\ge1\),
\[
e_{t+1}(\lambda)
=
(1+\beta_t)
\left(1-\frac{2\lambda}{L+\mu}\right)e_t(\lambda)
-
\beta_t e_{t-1}(\lambda).
\]

We shall prove by induction that
\(e_t(\lambda)=p_t^\star(\lambda)e_0(\lambda) \; \text{for all }t\ge0\).
First,
\(p_0^\star(\lambda)=\frac{T_0(z(\lambda))}{T_0(\eta)}=1\),
so the identity holds at \(t=0\).  Also,
\[
p_1^\star(\lambda)
=
\frac{T_1(z(\lambda))}{T_1(\eta)}
=
\frac{z(\lambda)}{\eta}.
\]
Since
\[
z(\lambda)
=
\frac{L+\mu-2\lambda}{\Delta}
=
\eta\left(1-\frac{2\lambda}{L+\mu}\right),
\]
we get
\(p_1^\star(\lambda) = 1-\frac{2\lambda}{L+\mu}\).
Hence
\(e_1(\lambda)=p_1^\star(\lambda)e_0(\lambda)\).

Assume now that the identity holds at times \(t-1\) and \(t\), with \(t\ge1\).
We prove it at time \(t+1\).  Since \(\eta>1\), one has
\(T_k(\eta)>0 \; \text{for every }k\ge0\),
and therefore the coefficient
\(\beta_t=\frac{T_{t-1}(\eta)}{T_{t+1}(\eta)}\)
is well-defined.  The Chebyshev identity
\(T_{t+1}(w)=2wT_t(w)-T_{t-1}(w)\)
applied at \(w=z(\lambda)\) gives
\[
p_{t+1}^\star(\lambda)
=
\frac{2z(\lambda)T_t(z(\lambda))-T_{t-1}(z(\lambda))}
     {T_{t+1}(\eta)}.
\]
Using
\[
T_t(z(\lambda))=T_t(\eta)p_t^\star(\lambda),
\;
T_{t-1}(z(\lambda))=T_{t-1}(\eta)p_{t-1}^\star(\lambda),
\]
we obtain
\[
p_{t+1}^\star(\lambda)
=
\frac{2z(\lambda)T_t(\eta)}{T_{t+1}(\eta)}
p_t^\star(\lambda)
-
\frac{T_{t-1}(\eta)}{T_{t+1}(\eta)}
p_{t-1}^\star(\lambda).
\]
The second coefficient is \(\beta_t\).  For the first coefficient, use
\(T_{t+1}(\eta)+T_{t-1}(\eta)=2\eta T_t(\eta)\),
which is the same Chebyshev recurrence at \(w=\eta\).  Hence
\[
\frac{2z(\lambda)T_t(\eta)}{T_{t+1}(\eta)}
=
\frac{z(\lambda)}{\eta}
\frac{T_{t+1}(\eta)+T_{t-1}(\eta)}{T_{t+1}(\eta)}
=
\left(1-\frac{2\lambda}{L+\mu}\right)(1+\beta_t).
\]
Therefore
\[
p_{t+1}^\star(\lambda)
=
(1+\beta_t)
\left(1-\frac{2\lambda}{L+\mu}\right)
p_t^\star(\lambda)
-
\beta_t p_{t-1}^\star(\lambda).
\]
Using the induction hypothesis in the scalar recurrence for \(e_t(\lambda)\),
we get
\(e_{t+1}(\lambda) = p_{t+1}^\star(\lambda)e_0(\lambda)\).
This closes the induction on the spectral component \(\lambda\).

Since the spectral decomposition of \(H\) is orthogonal and the above identity
holds on every eigenspace, recombining the spectral components gives
\(e_t=p_t^\star(H)e_0 \; \text{for all }t\ge0\).
Equivalently,
\(x_t-x^\star=p_t^\star(H)(x_0-x^\star)\),
which proves the theorem.
\end{proof}

\subsection{First-order kernel and exact gain}

\begin{lemma}[Chebyshev tail fundamental solution]\label{lem:cheb-tail-fundamental}
Fix \(1\le r\le N\).  For the prefix-exact Chebyshev recurrence, consider the scalar tail problem from time \(r\) to terminal time \(N\), with initial tail data
\(y_{r-1}=0, \; y_r=1\).
Then the terminal value is
\(G_{N,r}(z)=\frac{T_r(\eta)}{T_N(\eta)}U_{N-r}(z)\),
where \(U_m\) is the Chebyshev polynomial of the second kind.  In particular, for the tail after the insertion at time \(s\), \(0\le s\le N-1\),
\(G_{N,s+1}(z)=\frac{T_{s+1}(\eta)}{T_N(\eta)}U_{N-s-1}(z)\).
For \(s=0\), this specializes to
\[
G_{N,1}(z)=\frac{T_1(\eta)}{T_N(\eta)}U_{N-1}(z)
=
\frac{\eta}{T_N(\eta)}U_{N-1}(z).
\]
\end{lemma}

\begin{proof}[Proof of \Cref{lem:cheb-tail-fundamental}]\refstepcounter{prooflink}\label{proof:lem:cheb-tail-fundamental}
Fix \(1\le r\le N\).  The scalar prefix-exact Chebyshev recurrence in the
variable \(z\) is
\[
y_{t+1}
=
2z\,\frac{T_t(\eta)}{T_{t+1}(\eta)}\,y_t
-
\frac{T_{t-1}(\eta)}{T_{t+1}(\eta)}\,y_{t-1},
\; t\ge1.
\]
Since \(\eta>1\), there exists \(a>0\) such that \(\eta=\cosh a\).  Hence
\(T_t(\eta)=\cosh(ta)>0 \; \text{for all }t\ge0\).
Thus all normalizing factors in the recurrence are nonzero.

Define
\(Y_t:=T_t(\eta)y_t\).
Multiplying the recurrence by \(T_{t+1}(\eta)\), we obtain
\(Y_{t+1}=2zY_t-Y_{t-1}\).
The tail initial data
\(y_{r-1}=0, \; y_r=1\)
become
\(Y_{r-1}=0, \; Y_r=T_r(\eta)\).

Set
\(A:=T_r(\eta)\).
We prove by induction on \(k\ge0\) that
\(Y_{r+k}=A\,U_k(z)\).
For \(k=0\), this is \(Y_r=A=A\,U_0(z)\).  For \(k=1\), the recurrence gives
\(Y_{r+1}=2zY_r-Y_{r-1}=2zA=A\,U_1(z)\).
Assume the identity holds for \(k-1\) and \(k\).  Then
\[
Y_{r+k+1}
=
2zY_{r+k}-Y_{r+k-1}
=
A\bigl(2zU_k(z)-U_{k-1}(z)\bigr)
=
A\,U_{k+1}(z),
\]
where the last equality is the defining three-term recurrence for the
Chebyshev polynomials of the second kind.  Therefore
\(Y_{r+k}=T_r(\eta)U_k(z) \; \text{for all }k\ge0\).
Taking \(k=N-r\), we get
\(Y_N=T_r(\eta)U_{N-r}(z)\).
Dividing by \(T_N(\eta)\) gives
\(y_N = \frac{Y_N}{T_N(\eta)} = \frac{T_r(\eta)}{T_N(\eta)}U_{N-r}(z)\).
This proves the formula for \(G_{N,r}\).  Substituting \(r=s+1\) gives the
displayed formula for the tail after an insertion at time \(s\).  The displayed
specialization for \(s=0\) follows from \(T_1(\eta)=\eta\).
\end{proof}

\begin{proof}[Proof of \Cref{thm:prefix-kernel}]\refstepcounter{prooflink}\label{proof:thm:prefix-kernel}
Fix spectral values
\(\lambda,\nu\in[\mu,L], \; z:=z(\lambda), \; y:=z(\nu)\).
The two variables have different roles: \(y\) is the spectral parameter before
the insertion, and \(z\) is the spectral parameter after the insertion.

We compute the scalar first variation associated with one Hessian insertion at
time \(s\).  Suppose the input error lies in the \(\nu\)-spectral mode and the
inserted perturbation maps that mode into the \(\lambda\)-spectral mode.  The
kernel \(K_s^{\rm pref}(\lambda,\nu)\) is the product of three scalar factors:
the prefix transfer up to the query time \(s\), the negative gradient
coefficient at time \(s\), and the tail transfer from the newly created
\(\lambda\)-mode to time \(N\).

For \(s=0\), the prefix transfer is the identity.  The first step of the
recurrence is
\(e_1=e_0-\frac{2}{L+\mu}H_0e_0\).
Under the perturbation \(H_0=H+\delta B_0\), the derivative at
\(\delta=0\) contributes
\(-\frac{2}{L+\mu}B_0e_0\)
to the state at time \(1\).  Thus the insertion coefficient at \(s=0\) is
\(\gamma_0=\frac{2}{L+\mu}\).
Since
\(L+\mu=\Delta\eta\),
we also have
\(\gamma_0=\frac{2}{\Delta\eta}\).
By \Cref{lem:cheb-tail-fundamental}, the tail transfer from time \(1\) to time
\(N\), with previous state fixed at zero and current state equal to one, is
\[
G_{N,1}(z)
=
\frac{T_1(\eta)}{T_N(\eta)}U_{N-1}(z)
=
\frac{\eta}{T_N(\eta)}U_{N-1}(z).
\]
Therefore
\[
K_0^{\rm pref}(\lambda,\nu)
=
-\gamma_0G_{N,1}(z)
=
-\frac{2}{\Delta T_N(\eta)}U_{N-1}(z).
\]

Now let \(1\le s\le N-1\).  Before the insertion, the \(\nu\)-mode evolves
according to the prefix-exact residual.  Hence, by
\Cref{thm:prefix-exact-chebyshev-recurrence}, the scalar prefix transfer to
time \(s\) is
\(r_s(\nu) = p_s^\star(\nu) = \frac{T_s(y)}{T_s(\eta)}\).
At time \(s\), the recurrence is
\[
e_{s+1}
=
(1+\beta_s)\left(I-\frac{2}{L+\mu}H_s\right)e_s
-
\beta_s e_{s-1},
\;
\beta_s=\frac{T_{s-1}(\eta)}{T_{s+1}(\eta)}.
\]
Under the perturbation \(H_s=H+\delta B_s\), the derivative at
\(\delta=0\) contributes
\(-(1+\beta_s)\frac{2}{L+\mu}B_se_s\)
to the state at time \(s+1\).  Thus the insertion coefficient is
\(\gamma_s = (1+\beta_s)\frac{2}{L+\mu}\).
Using
\(1+\beta_s = \frac{T_{s+1}(\eta)+T_{s-1}(\eta)}{T_{s+1}(\eta)}\)
and the Chebyshev identity
\(T_{s+1}(\eta)+T_{s-1}(\eta)=2\eta T_s(\eta)\),
we obtain
\[
\gamma_s
=
\frac{2}{L+\mu}
\frac{2\eta T_s(\eta)}{T_{s+1}(\eta)}
=
\frac{4T_s(\eta)}{\Delta T_{s+1}(\eta)}.
\]
By \Cref{lem:cheb-tail-fundamental}, the tail transfer from time \(s+1\) to
time \(N\) in the \(\lambda\)-mode is
\(G_{N,s+1}(z) = \frac{T_{s+1}(\eta)}{T_N(\eta)}U_{N-s-1}(z)\).
Multiplying prefix, insertion coefficient, and tail, with the negative sign
from the perturbed gradient step, gives
\[
\begin{aligned}
K_s^{\rm pref}(\lambda,\nu)
&=
-G_{N,s+1}(z)\,\gamma_s\,r_s(\nu) \\
&=
-
\frac{T_{s+1}(\eta)}{T_N(\eta)}U_{N-s-1}(z)
\frac{4T_s(\eta)}{\Delta T_{s+1}(\eta)}
\frac{T_s(y)}{T_s(\eta)} \\
&=
-\frac{4}{\Delta T_N(\eta)}
U_{N-s-1}(z)T_s(y).
\end{aligned}
\]
This proves the stated kernel formulas.

It remains to compute the sharp gain \(A_N^{\rm pref}\).  By definition,
\[
A_N^{\rm pref}
=
\sup_{\lambda,\nu\in[\mu,L]}
\left(
\sum_{s=0}^{N-1}
\abs{K_s^{\rm pref}(\lambda,\nu)}^2
\right)^{1/2}.
\]
The affine map \(\lambda\mapsto z(\lambda)\) sends \([\mu,L]\) onto
\([-1,1]\), and the same is true for \(\nu\mapsto z(\nu)\).  Hence the
supremum may equivalently be taken over \(z,y\in[-1,1]\).

For \(z\in[-1,1]\), write \(z=\cos\theta\).  Then
\(U_m(z)=U_m(\cos\theta)=\frac{\sin((m+1)\theta)}{\sin\theta}\)
for \(0<\theta<\pi\), with endpoint values obtained by continuity.  The
inequality
\(|\sin((m+1)\theta)|\le (m+1)|\sin\theta|\)
gives
\(\abs{U_m(z)}\le m+1, \; z\in[-1,1]\).
Also,
\(\abs{T_s(y)}\le1, \; y\in[-1,1]\).
Therefore, for every \(z,y\in[-1,1]\),
\[
\sum_{s=0}^{N-1}
\abs{K_s^{\rm pref}(\lambda,\nu)}^2
\le
\frac{1}{\Delta^2T_N(\eta)^2}
\left(
4N^2+
16\sum_{s=1}^{N-1}(N-s)^2
\right).
\]
The right-hand side is attained at
\(z=1, \; y=1\),
because
\(U_m(1)=m+1, \; T_s(1)=1\).
Consequently
\[
A_N^{\rm pref}{}^2
=
\frac{1}{\Delta^2T_N(\eta)^2}
\left(
4N^2+
16\sum_{s=1}^{N-1}(N-s)^2
\right).
\]
Since
\(\epsilon_N^\star=\frac1{T_N(\eta)}\),
and since the change of variables \(m=N-s\) gives
\(\sum_{s=1}^{N-1}(N-s)^2 = \sum_{m=1}^{N-1}m^2\),
we obtain
\[
A_N^{\rm pref}
=
\frac{\epsilon_N^\star}{\Delta}
\left(
4N^2+
16\sum_{m=1}^{N-1}m^2
\right)^{1/2}.
\]
Using
\(\sum_{m=1}^{N-1}m^2 = \frac{(N-1)N(2N-1)}6\),
this becomes
\[
A_N^{\rm pref}
=
\frac{2\epsilon_N^\star}{\Delta}
\sqrt{
N^2+\frac23N(N-1)(2N-1)
}.
\]
Finally,
\(N^2+\frac23N(N-1)(2N-1) = \frac43N^3+O(N^2)\),
and therefore
\[
A_N^{\rm pref}
=
\frac{2\epsilon_N^\star}{\Delta}
\sqrt{\frac43N^3+O(N^2)}
=
\frac{4}{\sqrt3}\frac{N^{3/2}}{\Delta}\epsilon_N^\star(1+o(1)).
\]
\end{proof}

\section{Proofs for Section~\ref{sec:prefix-lower-main}}
\label{app:proofs-sec3}

This appendix proves the uniqueness and lower-bound statements for the
causal two-term prefix-exact class.

The following auxiliary theorem is used only inside the proof of the prefix lower bound.

\begin{theorem}[Local uniqueness of causal two-term prefix realization]\label{thm:prefix-realization-uniqueness}
Fix \(M\ge1\).  Consider a causal two-term one-gradient recurrence initialized by
\(x_1=x_0-\frac{2}{L+\mu}\nabla f(x_0)\),
and, for \(1\le t\le M-1\), evolved by
\(x_{t+1}=a_tx_t+b_tx_{t-1}-\gamma_t\nabla f(x_t)\).
Assume that on every quadratic objective
\[
f(x)=\frac12\langle H(x-x^\star),x-x^\star\rangle,
\;
\spec(H)\subset[\mu,L],
\]
the residuals are prefix-exact through time \(M\):
\(x_j-x^\star=p_j^\star(H)(x_0-x^\star) \; \text{for }0\le j\le M\).
Then, for every \(1\le t\le M-1\),
\[
a_t=1+\beta_t,
\;
b_t=-\beta_t,
\;
\gamma_t=(1+\beta_t)\frac{2}{L+\mu}.
\]
\end{theorem}

\begin{proof}[Proof of \Cref{thm:prefix-realization-uniqueness}]\refstepcounter{prooflink}\label{proof:thm:prefix-realization-uniqueness}
Fix \(1\le t\le M-1\).  We prove that the coefficients at this time are uniquely
forced by the prefix-exact identities at times \(t-1,t,t+1\), which are among the assumed identities because \(0\le t-1<t<t+1\le M\).

It is enough to work on one-dimensional quadratic objectives.  Let
\(f_\lambda(x)=\frac12\lambda x^2, \; \lambda\in[\mu,L]\).
On this scalar mode, the recurrence gives
\[
e_{t+1}(\lambda)
=
(a_t-\gamma_t\lambda)e_t(\lambda)
+
b_t e_{t-1}(\lambda),
\]
where \(e_j(\lambda)=x_j-x^\star\).  By prefix exactness,
\(e_j(\lambda)=p_j^\star(\lambda)e_0 \; \text{for }j=t-1,t,t+1\).
Since the identity holds for every initial scalar error \(e_0\), we obtain the
polynomial identity
\[
p_{t+1}^\star(\lambda)
=
(a_t-\gamma_t\lambda)p_t^\star(\lambda)
+
b_t p_{t-1}^\star(\lambda).
\]

We now write the known Chebyshev recurrence for the normalized residuals.  By
the proof of \Cref{thm:prefix-exact-chebyshev-recurrence}, for every
\(\lambda\),
\[
p_{t+1}^\star(\lambda)
=
(1+\beta_t)
\left(1-\frac{2\lambda}{L+\mu}\right)
p_t^\star(\lambda)
-
\beta_t p_{t-1}^\star(\lambda).
\]
Subtracting this identity from the previous one gives
\[
\left[
a_t-(1+\beta_t)
-
\left(\gamma_t-(1+\beta_t)\frac{2}{L+\mu}\right)\lambda
\right]p_t^\star(\lambda)
+
(b_t+\beta_t)p_{t-1}^\star(\lambda)
=
0.
\]

The three polynomials
\[
\lambda p_t^\star(\lambda),
\;
p_t^\star(\lambda),
\;
p_{t-1}^\star(\lambda)
\]
are linearly independent.  Indeed, \(p_k^\star\) has degree \(k\), because
\(p_k^\star(\lambda)=\frac{T_k(z(\lambda))}{T_k(\eta)}\)
and \(z(\lambda)\) is a nonconstant affine function of \(\lambda\), while
\(T_k\) has degree \(k\) and \(T_k(\eta)\ne0\).  Therefore
\[
\deg(\lambda p_t^\star)=t+1,
\;
\deg(p_t^\star)=t,
\;
\deg(p_{t-1}^\star)=t-1.
\]
A nontrivial linear combination of these three polynomials cannot vanish
identically, because the highest nonzero degree term would be uncancelled.

Hence all three scalar coefficients in the displayed identity must vanish.
Thus
\(\gamma_t-(1+\beta_t)\frac{2}{L+\mu}=0\),
\(a_t-(1+\beta_t)=0\),
and
\(b_t+\beta_t=0\).
Equivalently,
\[
a_t=1+\beta_t,
\;
b_t=-\beta_t,
\;
\gamma_t=(1+\beta_t)\frac{2}{L+\mu}.
\]
Since \(1\le t\le M-1\) was arbitrary, the conclusion holds for every \(1\le t\le M-1\).
\end{proof}

\begin{proof}[Proof of \Cref{thm:prefix-lower}]\refstepcounter{prooflink}\label{proof:thm:prefix-lower}
We first show that the coefficients of any recurrence in the stated class are
forced by prefix exactness.

On a scalar quadratic objective with curvature \(\lambda\), the first update is
\(e_1(\lambda)=(1-\gamma_0\lambda)e_0\).
Hence the first residual polynomial is
\(r_1(\lambda)=1-\gamma_0\lambda\).
By prefix exactness,
\(r_1(\lambda)=p_1^\star(\lambda)\).
Since
\(p_1^\star(\lambda) = 1-\frac{2\lambda}{L+\mu}\),
we obtain the polynomial identity
\(1-\gamma_0\lambda = 1-\frac{2\lambda}{L+\mu}\).
Therefore
\(\gamma_0=\frac{2}{L+\mu}\).

The proof of \Cref{thm:prefix-realization-uniqueness} uses only centered scalar
quadratics, so its coefficient comparison applies under the present scalar
prefix-exactness hypotheses.  Taking \(M=N\), for every \(1\le t\le N-1\)
the coefficients at time \(t\) are uniquely forced:
\[
a_t=1+\beta_t,
\;
b_t=-\beta_t,
\;
\gamma_t=(1+\beta_t)\frac{2}{L+\mu}.
\]
Thus, through terminal time \(N\), the recurrence has exactly the same scalar
update coefficients as the standard prefix-exact Chebyshev recurrence.

It remains to compare the first-order gains.  The first-order Hessian-drift
kernel of a causal two-term one-gradient recurrence is determined by its scalar
prefix transfers, its gradient insertion coefficients, and its scalar tail
fundamental solutions.  All three objects are the same for the present
recurrence and for the standard prefix-exact Chebyshev recurrence, because the
update coefficients agree at every time \(0,\ldots,N-1\).  Therefore
\(K_s(\lambda,\nu)=K_s^{\rm pref}(\lambda,\nu), \; s=0,\ldots,N-1\),
for all \(\lambda,\nu\in[\mu,L]\).  Consequently
\[
A_N=A_N^{\rm pref}.
\]

By \Cref{thm:prefix-kernel},
\[
A_N^{\rm pref}
=
\frac{\epsilon_N^\star}{\Delta}
\left(4N^2+16\sum_{m=1}^{N-1}m^2\right)^{1/2}.
\]
This proves the displayed lower bound.  The same theorem shows that the
standard prefix-exact Chebyshev recurrence attains this value, and hence it is
optimal in the stated causal two-term prefix-exact class.

Finally,
\(\sum_{m=1}^{N-1}m^2 = \frac{(N-1)N(2N-1)}6\),
so
\(4N^2+16\sum_{m=1}^{N-1}m^2 = \frac{16}{3}N^3+O(N^2)\).
Taking square roots gives
\[
\frac{\epsilon_N^\star}{\Delta}
\left(4N^2+16\sum_{m=1}^{N-1}m^2\right)^{1/2}
=
\frac{4}{\sqrt3}\frac{N^{3/2}}{\Delta}\epsilon_N^\star(1+o(1)).
\]
The stated asymptotic lower bound follows.
\end{proof}

\section{Proofs for Section~\ref{sec:jacobi-main}}
\label{app:proofs-sec4}

This appendix gives the proofs of the Jacobi method, the Jacobi
first-variation factorization, and the finite-drift results stated in
Section~\ref{sec:jacobi-main}.

\subsection{Jacobi chains and the Chebyshev fibre}

\begin{proof}[Proof of \Cref{thm:jacobi-method}]\refstepcounter{prooflink}\label{proof:thm:jacobi-method}
Set
\[
A_s:=\frac{P_s(\eta)}{P_{s+1}(\eta)},
\;
B_s:=a_s^2\frac{P_{s-1}(\eta)}{P_{s+1}(\eta)},
\; s=0,\ldots,N-1.
\]
The assumption \(P_{s+1}(\eta)>0\) makes all displayed coefficients
well defined.  Since \(a_0=0\) and \(P_{-1}=0\), one has \(B_0=0\).

Let
\[
f(x)=f_\star+\frac12\langle H(x-x_\star),x-x_\star\rangle,
\;
e_s:=x_s-x_\star.
\]
Then \(\nabla f(x_s)=He_s\).  Evaluating the Jacobi recurrence at \(z=\eta\)
gives
\(P_{s+1}(\eta) = (\eta-b_s)P_s(\eta)-a_s^2P_{s-1}(\eta)\).
After division by \(P_{s+1}(\eta)\), this is the fixed-point identity
\(A_s(\eta-b_s)-B_s=1\).
Subtracting \(x_\star\) from the method recurrence therefore cancels the
constant part exactly and gives
\[
e_{s+1}
=
A_s\left(\eta-b_s-\frac{2H}{\Delta}\right)e_s
-
B_s e_{s-1},
\; s=0,\ldots,N-1,
\]
where the \(e_{-1}\)-term is absent at \(s=0\) because \(B_0=0\).

We now identify the scalar residual recurrence.  Put \(r_{-1}:=0\).  For every
\(\lambda\in[\mu,L]\), using
\(z(\lambda)=\eta-\frac{2\lambda}{\Delta}\),
the defining recurrence for \(P_s\) gives
\[
P_{s+1}(z(\lambda))
=
\left(\eta-b_s-\frac{2\lambda}{\Delta}\right)P_s(z(\lambda))
-
a_s^2P_{s-1}(z(\lambda)).
\]
Dividing by \(P_{s+1}(\eta)\) and rewriting the result in terms of
\(r_j(\lambda)=P_j(z(\lambda))/P_j(\eta)\), we obtain
\[
r_{s+1}(\lambda)
=
A_s\left(\eta-b_s-\frac{2\lambda}{\Delta}\right)r_s(\lambda)
-
B_s r_{s-1}(\lambda),
\; s=0,\ldots,N-1,
\]
with \(r_0(\lambda)=1\) and \(r_{-1}(\lambda)=0\).

By the spectral theorem, the operator recurrence for \(e_s\) is obtained from
this scalar recurrence by functional calculus with \(\lambda\) replaced by
\(H\).  Since \(e_0=r_0(H)e_0\), induction over \(s\) gives
\(e_s=r_s(H)e_0, \; s=0,\ldots,N\).

If \(P_N(z)=2^{1-N}T_N(z)\), then
\[
r_N(\lambda)
=
\frac{P_N(z(\lambda))}{P_N(\eta)}
=
\frac{2^{1-N}T_N(z(\lambda))}{2^{1-N}T_N(\eta)}
=
p_N^\star(\lambda).
\]
Applying functional calculus once more yields
\(e_N=r_N(H)e_0=p_N^\star(H)e_0\).
Finally, at \(s=0\) we have \(P_0=1\), \(a_0=0\), and \(P_{-1}=0\), so the
general recurrence reduces to the displayed formula for \(x_1\).
\end{proof}

\subsection{First-variation factorization}

\begin{proof}[Proof of \Cref{thm:jacobi-first-variation-kernel}]\refstepcounter{prooflink}\label{proof:thm:jacobi-first-variation-kernel}
Fix \(s\in\{0,\ldots,N-1\}\).  We compute the three scalar factors in the
definition of the first-variation kernel: the prefix before the insertion, the
insertion coefficient, and the tail after the insertion.

First, the unperturbed Jacobi residual at time \(s\) is
\(r_s(\nu)=\frac{P_s(z(\nu))}{P_s(\eta)}\).
Thus, in \(z\)-coordinates and with \(y=z(\nu)\), the prefix factor multiplying
the input direction before the Hessian insertion at time \(s\) is
\(\frac{P_s(y)}{P_s(\eta)}\).

Second, the method of \Cref{thm:jacobi-method} uses the gradient coefficient
\(\frac{2}{\Delta}\frac{P_s(\eta)}{P_{s+1}(\eta)}\).
A perturbation \(H_s=H+\delta B_s\) therefore contributes, at first order,
the additive term
\(-\frac{2}{\Delta}\frac{P_s(\eta)}{P_{s+1}(\eta)} B_s e_s\)
to \(e_{s+1}\).  Hence the scalar insertion coefficient in the kernel is
\(-\frac{2}{\Delta}\frac{P_s(\eta)}{P_{s+1}(\eta)}\).

It remains to compute the tail transfer from a perturbation injected into
\(e_{s+1}\) to the terminal error \(e_N\).  For \(t\ge s+1\), let
\(\Phi_{N,s+1}(x)\) denote the scalar terminal value of the homogeneous
linearized recurrence in \(z\)-coordinate \(x\), with tail initial data
\(u_s=0,\; u_{s+1}=1\).
The recurrence for \(u_t\) is
\[
u_{t+1}
=
\frac{P_t(\eta)}{P_{t+1}(\eta)}(x-b_t)u_t
-
a_t^2\frac{P_{t-1}(\eta)}{P_{t+1}(\eta)}u_{t-1},
\; t=s+1,\ldots,N-1.
\]
We claim that
\[
\Phi_{N,s+1}(x) = \frac{P_{s+1}(\eta)}{P_N(\eta)} Q_{N-s-1}^{(s+1)}(x).
\]
Indeed, define for \(t=s+1,\ldots,N\)
\(u_t(x) := \frac{P_{s+1}(\eta)}{P_t(\eta)} Q_{t-s-1}^{(s+1)}(x)\),
where
\(Q_{t-s-1}^{(s+1)}(x) = \det(xI-J_{s+1:t-1})\)
and the empty determinant is \(Q_0^{(s+1)}=1\).  Then \(u_{s+1}=1\), while
\(u_s=0\) is the imposed left tail datum.  The tail determinants satisfy the
Jacobi recurrence
\[
Q_{t-s}^{(s+1)}(x)
=
(x-b_t)Q_{t-s-1}^{(s+1)}(x)
-
a_t^2Q_{t-s-2}^{(s+1)}(x),
\; t=s+1,\ldots,N-1,
\]
with \(Q_{-1}^{(s+1)}:=0\).  After multiplication by
\(P_{s+1}(\eta)/P_{t+1}(\eta)\), this is exactly the displayed recurrence for
\(u_{t+1}\).  Hence the claimed tail transfer formula holds.

Multiplying the prefix factor, insertion coefficient, and tail transfer gives
\[
\begin{aligned}
K_s(x,y)
&=
\left(
\frac{P_{s+1}(\eta)}{P_N(\eta)}
Q_{N-s-1}^{(s+1)}(x)
\right)
\left(
-\frac{2}{\Delta}\frac{P_s(\eta)}{P_{s+1}(\eta)}
\right)
\left(
\frac{P_s(y)}{P_s(\eta)}
\right)  \\
&=
-\frac{2}{\Delta P_N(\eta)}
Q_{N-s-1}^{(s+1)}(x)P_s(y).
\end{aligned}
\]
This proves the kernel formula.

Finally, the affine map \(u\mapsto\lambda(u)\) is a bijection from \([-1,1]\)
onto \([\mu,L]\).  Therefore the definition of \(A_N(\mathcal R_J)\) gives
\[
A_N(J)
=
\sup_{x,y\in[-1,1]}
\left(
\sum_{s=0}^{N-1}\abs{K_s(x,y)}^2
\right)^{1/2}.
\]
Substituting the kernel formula and taking the absolute value of the common
normalizing factor \(P_N(\eta)\) yields
\[
A_N(J)
=
\frac{2}{\Delta |P_N(\eta)|}
\sup_{x,y\in[-1,1]}
\left(
\sum_{s=0}^{N-1}
\abs{Q_{N-s-1}^{(s+1)}(x)P_s(y)}^2
\right)^{1/2}.
\]
\end{proof}

\subsection{Finite Hessian drift}

Throughout, \(D_{a:b}(x):=\det(xI-J_{a:b})\) for \(0\le a\le b\le N-1\), with
\(D_{a:a-1}:=1\) and \(D_{a:a-2}:=0\); thus
\(\widetilde D_{a:b}=2^{\,b-a+1}D_{a:b}\).  Expansion along the last row gives
\begin{equation}\label{eq:block-recurrence}
D_{a:t}(x)=(x-b_t)D_{a:t-1}(x)-a_t^2D_{a:t-2}(x),
\qquad t\ge a,
\end{equation}
where the last term vanishes for \(t=a\).

\begin{proof}[Proof of \Cref{thm:finite-drift}]\refstepcounter{prooflink}\label{proof:thm:finite-drift}
\noindent\emph{(i)}  Write \(x_s-x^\star=W_s(x_0-x^\star)\) for the method of
\Cref{thm:jacobi-method} when the gradient at step \(s\) is
\(H_s(x_s-x^\star)\), \(H_s=H+E_s\).  Multiplying the update of
\Cref{thm:jacobi-method} by \(P_{s+1}(\eta)\) shows that
\(\Pi_s:=P_s(\eta)W_s\) satisfies \(\Pi_{-1}=0\), \(\Pi_0=I\), and
\[
\Pi_{s+1}=(Z_s-b_s)\Pi_s-a_s^2\Pi_{s-1},
\qquad
Z_s:=\eta I-\tfrac2\Delta H_s=Z-\tfrac2\Delta E_s,
\]
with \(a_0:=0\).  Hence
\(\TT_{\mathcal R_J}(H_0,\ldots,H_{N-1})=W_N=\Pi_N/P_N(\eta)\).

We claim that, for \(0\le t\le N\),
\begin{equation}\label{eq:jacobi-duhamel}
\Pi_t=D_{0:t-1}(Z)-\frac2\Delta\sum_{r=0}^{t-1}D_{r+1:t-1}(Z)\,E_r\,\Pi_r .
\end{equation}
For \(t=0\) both sides equal \(I\), and the same formula with \(t=-1\) gives
\(\Pi_{-1}=0\).  Assume \eqref{eq:jacobi-duhamel} for \(t-1\) and \(t\).
In the formula for \(\Pi_{t-1}\) the sum may be extended to \(r=t-1\),
because \(D_{t:t-2}=0\).  Substituting both formulas into the recurrence and
using \eqref{eq:block-recurrence} for the blocks starting at \(0\) and at
\(r+1\le t\),
\[
\Pi_{t+1}
=D_{0:t}(Z)-\frac2\Delta\sum_{r=0}^{t-1}D_{r+1:t}(Z)E_r\Pi_r
-\frac2\Delta E_t\Pi_t ,
\]
which is \eqref{eq:jacobi-duhamel} for \(t+1\), since \(D_{t+1:t}=1\).

Every \(\Pi_r\) on the right of \eqref{eq:jacobi-duhamel} has \(r<t\).
Substituting the formula into itself until no \(\Pi_r\) remains, which takes
at most \(t\) rounds, gives
\[
\Pi_N=\sum_{k=0}^N\Bigl(-\frac2\Delta\Bigr)^{k}
\!\!\sum_{0\le s_1<\cdots<s_k\le N-1}\!\!
D_{s_k+1:N-1}(Z)\,E_{s_k}\,D_{s_{k-1}+1:s_k-1}(Z)\cdots
E_{s_1}\,D_{0:s_1-1}(Z).
\]
In each term the blocks have total length \(N-k\), so the product of the
\(D\)'s equals \(2^{-(N-k)}\) times the product of the \(\widetilde D\)'s.  Since
\(P_N(\eta)=2^{1-N}T_N(\eta)\) and \(\epsilon_N^\star=1/T_N(\eta)\),
\[
\frac{(-2/\Delta)^k\,2^{-(N-k)}}{P_N(\eta)}
=\frac{\epsilon_N^\star}{2}\Bigl(-\frac4\Delta\Bigr)^{k},
\]
which is the expansion in (i).  For \(k=0\),
\(\frac{\epsilon_N^\star}{2}\widetilde D_{0:N-1}(Z)
=\frac{\epsilon_N^\star}{2}\,2^N\,2^{1-N}T_N(Z)=p_N^\star(H)\).  The term
\(k=1\) is the part of the expansion linear in \((E_s)\), which is
\(D\TT_{\mathcal R_J,H}[E]\) by \Cref{prop:finite-drift-linear}.

\noindent\emph{(ii)}  The operator \(Z\) is symmetric with spectrum in
\([-1,1]\), so \(\|\widetilde D_{a:b}(Z)\|\le g_{a:b}\).  Submultiplicativity
and \(\|E_s\|\le\delta\) bound each term of (i), and summation over
\(s_1<\cdots<s_k\) gives (ii).

\noindent\emph{(iii)}  For real \(t_0,\ldots,t_{N-1}\), the determinant
\(\det(\diag(1+t_s)-J)\) is affine in each \(t_s\), and the coefficient of
\(\prod_{s\in S}t_s\) is the principal minor of \(I-J\) obtained by deleting
the rows and columns in \(S\).  Since \(J\) is tridiagonal, deleting the sites
\(s_1<\cdots<s_k\) leaves the blocks between consecutive deleted sites, so
this minor is
\(D_{0:s_1-1}(1)D_{s_1+1:s_2-1}(1)\cdots D_{s_k+1:N-1}(1)\).  Setting all
\(t_s=t\) and using \(\det((1+t)I-J)=P_N(1+t)=2^{1-N}T_N(1+t)\),
\[
\sum_{0\le s_1<\cdots<s_k\le N-1}
\widetilde D_{0:s_1-1}(1)\cdots\widetilde D_{s_k+1:N-1}(1)
=2^{N-k}\,2^{1-N}[t^k]\,T_N(1+t)
=2^{1-k}[t^k]\,T_N(1+t).
\]
Since \(g_{a:b}\ge|\widetilde D_{a:b}(1)|\), this gives
\(G_k(J)\ge2^{1-k}[t^k]T_N(1+t)\), with equality when every block attains
its maximum at \(x=1\).  In that case, with \(u:=2\delta/\Delta\),
(i) and (ii) give
\[
\bigl\|\TT-p_N^\star(H)\bigr\|
\le\sum_{k=1}^N\frac{\epsilon_N^\star}{2}\Bigl(\frac{4\delta}{\Delta}\Bigr)^{k}
2^{1-k}[t^k]T_N(1+t)
=\epsilon_N^\star\bigl(T_N(1+u)-1\bigr).
\]
The term \(k=1\) is \(D\TT_{\mathcal R_J,H}[E]\), and
\([t]\,T_N(1+t)=T_N'(1)=N^2\) gives
\(\|D\TT_{\mathcal R_J,H}[E]\|\le\epsilon_N^\star N^2u\).  The terms
\(k\ge2\) give \(\epsilon_N^\star(T_N(1+u)-1-N^2u)\).  For its \(\cosh\)
form, \(\cosh\sqrt{2u}\ge1+u\) gives
\(\operatorname{arccosh}(1+u)\le\sqrt{2u}\), hence
\(T_N(1+u)=\cosh(N\operatorname{arccosh}(1+u))\le\cosh y\) with
\(y=N\sqrt{2u}\).  Since \(y^2/2=N^2u\) and
\(\cosh y-1-\frac{y^2}{2}=\sum_{j\ge2}\frac{y^{2j}}{(2j)!}
\le\frac{y^4}{24}\cosh y\),
\[
T_N(1+u)-1-N^2u\le\frac{N^4u^2}{6}\cosh\bigl(N\sqrt{2u}\bigr)
=\frac23N^4\Bigl(\frac\delta\Delta\Bigr)^2
\cosh\Bigl(2N\sqrt{\tfrac\delta\Delta}\Bigr).
\]

\noindent\emph{(iv)}  Let \(E_s=-\delta I\) for all \(s\).  Then
\(\TT_{\mathcal R}(H-\delta I,\ldots,H-\delta I)=p_{\mathcal R,N}(H-\delta I)\),
and final exactness identifies the polynomials \(p_{\mathcal R,N}\) and
\(p_N^\star\), which agree on \([\mu,L]\).  Thus
\(\TT_{\mathcal R}-p_N^\star(H)=p_N^\star(H-\delta I)-p_N^\star(H)\), and its
linear part in \(\delta\) is \(D\TT_{\mathcal R,H}[E]=-\delta\,(p_N^\star)'(H)\).
All three operators are functions of \(H\); on the \(\mu\)-eigenspace they
equal
\[
p_N^\star(\mu-\delta)-p_N^\star(\mu)=\epsilon_N^\star(T_N(1+u)-1),
\qquad
-\delta(p_N^\star)'(\mu)=\epsilon_N^\star N^2u,
\]
and their difference \(\epsilon_N^\star(T_N(1+u)-1-N^2u)\), because
\(z(\mu-\delta)=1+u\), \(z(\mu)=1\), and
\((p_N^\star)'(\mu)=-\frac2\Delta\epsilon_N^\star T_N'(1)\).  The operator norms
are at least these values.
\end{proof}

\begin{proof}[Proof of \Cref{cor:prefix-finite-drift}]\refstepcounter{prooflink}\label{proof:cor:prefix-finite-drift}
Put \(P_0=1\) and \(P_s=2^{1-s}T_s\) for \(s\ge1\).  Then \(P_1=x\),
\(P_2=xP_1-\frac12P_0\), and \(P_{s+1}=xP_s-\frac14P_{s-1}\) for \(s\ge2\) by
\(2xT_s=T_{s+1}+T_{s-1}\); so these are the leading polynomials of the
Jacobi matrix with \(b_s=0\), \(a_1^2=\frac12\), \(a_s^2=\frac14\)
(\(s\ge2\)).  In \Cref{thm:jacobi-method}, \(P_1(\eta)=\eta\) gives
\(x_1=x_0-\frac{2}{L+\mu}\nabla f(x_0)\).  For \(s\ge1\), the coefficients of
\(x_s\), \(x_{s-1}\), and \(\nabla f(x_s)\) are
\(2\eta T_s(\eta)/T_{s+1}(\eta)=1+\beta_s\), \(-\beta_s\), and
\(-\frac{4T_s(\eta)}{\Delta T_{s+1}(\eta)}=-(1+\beta_s)\frac{2}{L+\mu}\),
by \(2\eta T_s=T_{s+1}+T_{s-1}\) and \(\Delta\eta=L+\mu\).  This is the
prefix-exact recurrence.

Its blocks are explicit.  A block starting at site \(0\) is a leading block,
so \(\widetilde D_{0:b}=2^{b+1}P_{b+1}=2T_{b+1}\).  A block starting at
\(a\ge1\) has all couplings equal to \(\frac14\); hence
\(y_j=\widetilde D_{a:a+j-1}\) satisfies \(y_0=1\), \(y_1=2x\),
\(y_{j+1}=2xy_j-y_{j-1}\), that is, \(\widetilde D_{a:b}=U_{b-a+1}\).  Both
\(|T_n|\le1=T_n(1)\) and \(|U_n|\le n+1=U_n(1)\) hold on \([-1,1]\), so every
block attains its maximum at \(x=1\).  \Cref{thm:finite-drift}(iii) and (iv)
then give the bounds and their attainment.
\end{proof}

\section{Proofs for Section~\ref{sec:sine-upper-main}}
\label{app:proofs-sec5}

This appendix records the finite inverse-spectral and sine-lattice inputs used in Section~\ref{sec:sine-upper-main}, including the Fourier positivity of all blocks of the sine--Jacobi matrix used for finite drift.

\phantomsection\label{app:auxiliary-details}
\phantomsection\label{app:finite-inverse-spectral}
We use the finite Favard--Stieltjes inverse-spectral theorem: a positive probability measure on \(N\) distinct real points determines a unique \(N\times N\) Jacobi matrix with positive off-diagonal entries and that first spectral measure.  The relevant finite-support construction is in \cite[Sec.~2.5]{Teschl}; Theorem~2.13 there is stated for infinite support and is not invoked here.  In the present setting every nonzero polynomial of degree less than \(N\) has positive squared norm, since it cannot vanish at all \(N\) support points.  Gram--Schmidt therefore gives \(N\) orthonormal polynomials with positive leading coefficients, and multiplication by the coordinate gives the unique Jacobi matrix; the construction terminates at degree \(N\).

\subsection{Sine spectral measure, method, and algorithm}

We record the orthogonal-polynomial inputs for the sine-Jacobi upper bound and specialize them to the finite sine lattice.  The root-of-unity Rogers identities give the finite recurrence and maximum principle; the endpoint profile follows from a direct asymptotic analysis of the finite Rogers convolution.  The rest is finite-dimensional linear algebra.

\begin{theorem}[Finite sine-lattice package]
\label{thm:sine-lattice-package}
Let \(J_N^{\sin}\) be the Jacobi matrix associated with
\[
\rho_N^{\sin}=\sum_{j=1}^N w_j\delta_{\zeta_j},\;
\zeta_j=\cos\theta_j,\;
\theta_j=\frac{(2j-1)\pi}{2N},\;
w_j=\frac{\sin\theta_j}{\sum_{k=1}^N\sin\theta_k}.
\]
Let \(P_s^{(N)}\), \(s=0,\ldots,N\), be its monic leading polynomials.  For
\(0\le s\le N-1\), let
\(Q_{N-s-1}^{(s+1)}(z)\)
denote the monic characteristic polynomial of the right tail block on the sites
\(s+1,\ldots,N-1\), with the convention \(Q_0^{(N)}(z)=1\).  Then:
\begin{enumerate}[label=(\roman*),leftmargin=2em]
\item
\(b_{s,N}=0, \; s=0,\ldots,N-1\),
and
\(P_N^{(N)}(z)=\det(zI-J_N^{\sin})=2^{1-N}T_N(z)\).
\item For \(1\le s\le N-1\),
\[
a_{s,N}^{2}
=
\frac14\frac{\sin^2(\pi s/N)}
{\sin(\pi(s-\frac12)/N)\sin(\pi(s+\frac12)/N)}.
\]
\item For every compact \(K\subset(0,1)\), uniformly for \(s/N\in K\),
\[
N^2\left(a_{s,N}^2-\frac14\right)\to
V_{\sin}(t):=\frac{\pi^2}{16\sin^2(\pi t)},\; t=\frac{s}{N}.
\]
More precisely,
\(a_{s,N}^2=\frac14+\frac1{N^2}\frac{\pi^2}{16\sin^2(\pi t)}+O_K(N^{-4})\).
\item The endpoint-normalized polynomials satisfy the finite maximum principle
\[
|P_s^{(N)}(x)|\le P_s^{(N)}(1),
\;
-1\le x\le1,
\;
0\le s\le N-1.
\]
Moreover,
\[
\sup_{x,y\in[-1,1]}
\sum_{s=0}^{N-1}|P_s^{(N)}(y)Q_{N-s-1}^{(s+1)}(x)|^2
=
\sum_{s=0}^{N-1}P_s^{(N)}(1)^2Q_{N-s-1}^{(s+1)}(1)^2.
\]
\end{enumerate}
\end{theorem}
\begin{lemma}[Root-of-unity Rogers/q-ultraspherical specialization]
\label{lem:root-unity-rogers}
Let \(N\ge2\), set
\[
h=\frac{\pi}{N},
\;
q_0=\mathrm{e}^{2ih},
\;
\beta=\mathrm{e}^{ih}=q_0^{1/2},
\]
and define
\[
\theta_j=\frac{(2j-1)\pi}{2N},
\;
\zeta_j=\cos\theta_j,
\;
j=1,\ldots,N.
\]
Let
\[
d\rho_N(x)=\sum_{j=1}^N w_j\delta_{\zeta_j},
\;
w_j=\frac{\sin\theta_j}{\sum_{k=1}^N\sin\theta_k}.
\]
The monic orthogonal polynomials for \(d\rho_N\) are the finite
root-of-unity specialization of the continuous \(q\)-ultraspherical, equivalently
Rogers, family with parameters \((q_0,\beta)\).  In monic normalization they
satisfy
\(P_{s+1}(x)=xP_s(x)-a_{s,N}^2P_{s-1}(x), \; 1\le s\le N-1\),
where
\[
a_{s,N}^{2}
=
\frac14
\frac{\sin^2(sh)}
{\sin((s-\frac12)h)\sin((s+\frac12)h)}.
\]
Moreover, for \(0\le s<N\), the endpoint-normalized polynomial has a positive
trigonometric expansion
\[
\frac{P_s(\cos\theta)}{P_s(1)}
=
\sum_{k=0}^s c_{s,k}^{(N)}\cos((s-2k)\theta),
\;
c_{s,k}^{(N)}\ge0,
\;
\sum_{k=0}^s c_{s,k}^{(N)}=1.
\]
\end{lemma}

\begin{proof}[Proof of \Cref{lem:root-unity-rogers}]\refstepcounter{prooflink}\label{proof:lem:root-unity-rogers}
We use the finite root-of-unity recurrence and orthogonality in
\cite[Eqs.~(1)--(6),Eqs.~(22) and~(24)]{SpiridonovZhedanov}.
Their real exponent parameter, denoted here by \(b\) to distinguish it from
our \(\beta\), is \(b=1/2\), so that \(\beta=q_0^b=\mathrm e^{ih}\).
Their half-integer series has \(j=2b=1\) and
\(M=N+1-j=N\); its condition \(1\le j\le N-1\) holds for every \(N\ge2\).
For this series their nodes are
\(X_s=2\cos((s+\frac12)\pi/N)\), \(0\le s<N\).
The product in their weight formula (22) is empty when \(j=1\), leaving
strictly positive weights proportional to \(\sin((s+\frac12)\pi/N)\).
Thus the index shift \(m=s+1\), coordinate change \(x=X/2\), and normalization
of the weights give exactly \(d\rho_N\).

If \(\widetilde P_s(X)\) denotes their monic polynomial, then our monic
polynomial is \(P_s(x)=2^{-s}\widetilde P_s(2x)\), so its recurrence
coefficient is \(u_s/4\).  In terms of the multiplicative parameter
\(\beta=q^b\), Eqs.~(1)--(3) of that source therefore give
\[
A_s(q,\beta)
=
\frac{(1-q^s)(1-\beta^2q^{s-1})}
{4(1-\beta q^{s-1})(1-\beta q^s)}.
\]
Substituting \(q=q_0=\mathrm{e}^{2ih}\) and \(\beta=\mathrm{e}^{ih}\), we get
\[
\beta^2q_0^{s-1}=q_0^s,
\;
\beta q_0^{s-1}=\mathrm{e}^{i(2s-1)h},
\;
\beta q_0^s=\mathrm{e}^{i(2s+1)h}.
\]
Therefore
\[
A_s(q_0,\beta)
=
\frac{(1-\mathrm{e}^{2ish})^2}
{4(1-\mathrm{e}^{i(2s-1)h})(1-\mathrm{e}^{i(2s+1)h})}.
\]
Using
\(1-\mathrm{e}^{2iu}=-2i\mathrm{e}^{iu}\sin u\),
we obtain
\[
A_s(q_0,\beta)
=
\frac14
\frac{\sin^2(sh)}
{\sin((s-\frac12)h)\sin((s+\frac12)h)}.
\]
This is the displayed value of \(a_{s,N}^2\).

Finally, use the finite Rogers Fourier formula in
\cite[Sec.~3.10.1]{KS98}.
For each fixed degree this is a finite rational identity in \((q,\beta)\),
so it extends from its usual parameter range wherever its denominators are
nonzero.  Here \(s<N\), and every factor \(1-q_0^j\), \(1\le j\le s\),
is nonzero; the factors \(1-\beta q_0^{s-1}\) and \(1-\beta q_0^s\)
in the recurrence are also nonzero for \(1\le s<N\).
No continuous orthogonality formula or infinite \(q\)-product at
\(\lvert q_0\rvert=1\) is being used.  The finite formula gives
\[
C_s(\cos\theta;\beta\mid q_0)
=
\sum_{k=0}^s
\frac{(\beta;q_0)_k(\beta;q_0)_{s-k}}
{(q_0;q_0)_k(q_0;q_0)_{s-k}}
\mathrm{e}^{i(s-2k)\theta}.
\]
For \(0\le k\le s<N\),
\[
\frac{(\beta;q_0)_k}{(q_0;q_0)_k}
=
\mathrm{e}^{-ikh/2}
\prod_{r=0}^{k-1}
\frac{\sin((r+\frac12)h)}{\sin((r+1)h)}.
\]
Hence the product of the \(k\)-part and the \((s-k)\)-part has a common phase
independent of \(k\), multiplied by a positive real number.  Since the zeros of
\(P_s\), \(s<N\), lie in the convex hull of the support, the monic endpoint value
\(P_s(1)\) is positive.  Dividing by this endpoint value at \(\theta=0\) gives a
cosine expansion
\[
\frac{P_s(\cos\theta)}{P_s(1)}
=
\sum_{k=0}^s c_{s,k}^{(N)}\cos((s-2k)\theta)
\]
with
\(c_{s,k}^{(N)}\ge0, \; \sum_{k=0}^s c_{s,k}^{(N)}=1\).
\end{proof}

\begin{theorem}[Direct endpoint asymptotics for the sine hard-edge recurrence]
\label{thm:discrete-hard-edge-LG}
For \(N\ge2\), let \(F_s^{(N)}\), \(s=0,\ldots,N\), be the sequence determined by
\[
F_{s+1}^{(N)}-2F_s^{(N)}+F_{s-1}^{(N)}
=
-\frac{\sin^2(\pi/(2N))}
{\sin^2(\pi s/N)-\sin^2(\pi/(2N))}F_{s-1}^{(N)},
\;
1\le s<N,
\]
with \(F_0^{(N)}=1\) and \(F_1^{(N)}=2\).  Write
\[
\mathrm K(k):=\int_0^{\pi/2}\frac{d\theta}{\sqrt{1-k^2\sin^2\theta}},
\qquad
f(t):=\frac{2}{\pi^{3/2}}\sqrt{\sin(\pi t)}\,
\mathrm K\!\left(\sin\frac{\pi t}{2}\right),
\quad 0<t<1.
\]
Then \(f\) is positive and is the unique solution satisfying
\[
f''(t)+\frac{\pi^2}{4\sin^2(\pi t)}f(t)=0,
\qquad
f(t)\sim\sqrt t\quad(t\downarrow0).
\]
For every compact \(K\subset(0,1)\),
\[
\sup_{t\in K}\left|
\frac{F_{\lfloor Nt\rfloor}^{(N)}}{\sqrt{\pi N}}-f(t)
\right|\longrightarrow0.
\]
There is an absolute constant \(M\), for example \(M=3\pi^2/2\), such that
\[
0<F_s^{(N)}\le M\sqrt N,
\qquad 0\le s\le N,\quad N\ge2.
\]
Consequently, the step functions
\[
G_N(t):=
\frac{F_{\lfloor Nt\rfloor}^{(N)}F_{N-\lfloor Nt\rfloor-1}^{(N)}}{\pi N},
\qquad 0\le t<1,
\]
converge to \(f(t)f(1-t)\) in \(L^p(0,1)\) for every \(1\le p<\infty\),
in particular in \(L^2(0,1)\).
\end{theorem}

\begin{proof}[Proof of \Cref{thm:discrete-hard-edge-LG}]
\noindent 1)
Set \(h=\pi/N\), \(r_{0,N}=1\), and
\[
r_{k,N}:=\prod_{j=1}^k
\frac{\sin((j-\tfrac12)h)}{\sin(jh)},
\qquad 1\le k<N.
\]
We first show that
\[
F_s^{(N)}=\frac1{r_{s,N}}\sum_{k=0}^s r_{k,N}r_{s-k,N},
\qquad 0\le s<N.
\]
Indeed, for \(q_0=\mathrm e^{2ih}\) and \(\beta=\mathrm e^{ih}\),
\Cref{lem:root-unity-rogers} gives
\[
\frac{(\beta;q_0)_k}{(q_0;q_0)_k}
=\mathrm e^{-ikh/2}r_{k,N}.
\]
All denominators are nonzero for \(k<N\).  Thus the finite Rogers formula
used in \Cref{lem:root-unity-rogers} specializes to this root of unity as a
finite rational identity; no convergence of infinite \(q\)-products is used.  It gives
\[
C_s(\cos\theta;\beta\mid q_0)
=\mathrm e^{-ish/2}\sum_{k=0}^s
r_{k,N}r_{s-k,N}\mathrm e^{i(s-2k)\theta}.
\]
For \(s\ge1\), the terms \(k=0,s\) contribute
\(2\mathrm e^{-ish/2}r_{s,N}T_s(x)\), while the remaining paired terms have lower
polynomial degree in \(x=\cos\theta\).  The leading coefficient is therefore
\(2^s\mathrm e^{-ish/2}r_{s,N}\); at \(s=0\), it is \(1\).
Dividing by the leading coefficient and setting \(x=1\) proves the displayed
convolution for \(2^sP_s(1)\).  The finite monic recurrence in
\Cref{lem:root-unity-rogers}, multiplied by \(2^{s+1}\), is precisely the
recurrence in the statement, with the same initial values.  Its denominators
are positive for \(1\le s<N\).  Uniqueness of the recursively determined
sequence gives \(F_s^{(N)}=2^sP_s(1)\) for \(0\le s\le N\), and hence the
convolution for \(s<N\).

\medskip\noindent 2)
Let
\[
b_k:=\prod_{j=1}^k\frac{j-\tfrac12}{j}
=\frac{\binom{2k}{k}}{4^k},\quad b_0=1,
\qquad
g(x):=\log\frac{\sin(\pi x)}{\pi x},\quad g(0)=0.
\]
Fix \(0<b<1\).  The function \(g\) is \(C^2\) on \([0,b]\).  Uniformly for
\(0\le k\le bN\), Taylor's formula on each mesh interval gives
\[
\begin{aligned}
\log\frac{r_{k,N}}{b_k}
&=\sum_{j=1}^k\left[
 g\!\left(\frac{j-1/2}{N}\right)-g\!\left(\frac jN\right)\right] \\
&=-\frac1{2N}\sum_{j=1}^k g'\!\left(\frac jN\right)+O_b(N^{-1})
=-\frac12g\!\left(\frac kN\right)+O_b(N^{-1}).
\end{aligned}
\]
The Taylor error per term is \(O_b(N^{-2})\), and the error in the Riemann
sum for \(g'\) is \(O_b(N^{-1})\), uniformly in its upper endpoint.  Stirling's
formula gives \(b_k\sqrt{\pi k}\to1\) and \(b_k\sim(k+1)^{-1/2}\).
Consequently,
\[
\sqrt N\,r_{\lfloor Nu\rfloor,N}
\longrightarrow\frac1{\sqrt{\sin(\pi u)}}
\quad\text{locally uniformly for }0<u<1,
\]
and, for each \(0<b<1\), there are constants \(c_b,C_b>0\) such that
\[
\frac{c_b}{\sqrt{k+1}}\le r_{k,N}\le\frac{C_b}{\sqrt{k+1}},
\qquad 0\le k\le bN,\quad N\ge2.
\]

\medskip\noindent 3)
Put \(S_{s,N}:=\sum_{k=0}^s r_{k,N}r_{s-k,N}\).  We claim that, uniformly for
\(s/N\in[a,b]\subset(0,1)\),
\[
S_{s,N}-I(s/N)\longrightarrow0,
\qquad
I(t):=\int_0^t\frac{du}{\sqrt{\sin(\pi u)\sin(\pi(t-u))}}.
\]
Fix \(0<\delta<a/3\) and first restrict the sum to
\(\delta\le k/N\le s/N-\delta\).  By Step~2, each summand is \(N^{-1}\)
times the displayed integrand, with an error \(o(N^{-1})\) uniform in this
region.  The corresponding truncated Riemann sums converge uniformly,
because the integrand is uniformly continuous on
\(\{(t,u):a\le t\le b,\ \delta\le u\le t-\delta\}\).
For the omitted left tail, the bounds in Step~2 give
\[
\sum_{0\le k\le\delta N}r_{k,N}r_{s-k,N}
\le\frac{C_{a,b}}{\sqrt N}
\sum_{0\le k\le\delta N}(k+1)^{-1/2}
\le C'_{a,b}\bigl(\sqrt\delta+N^{-1/2}\bigr).
\]
The right tail has the same bound, and each omitted integral tail is at most
\(C''_{a,b}\sqrt\delta\).  Taking \(N\to\infty\) and then
\(\delta\downarrow0\) proves the claim.  Combining it with Steps~1--2 gives
\[
\frac{F_{\lfloor Nt\rfloor}^{(N)}}{\sqrt N}
\longrightarrow\sqrt{\sin(\pi t)}\,I(t)
\quad\text{locally uniformly on }(0,1).
\]
To pass from mesh points to arbitrary \(t\) in a compact interval, apply the
mesh estimates on a slightly larger compact interval and use continuity of
the limit together with \(|\lfloor Nt\rfloor/N-t|\le1/N\).

\medskip\noindent 4)
Let \(T=\pi t\) and \(k=\sin(T/2)\).  With \(y=\pi(u-t/2)\),
\[
I(t)=\frac1\pi\int_{-T/2}^{T/2}
\frac{dy}{\sqrt{k^2-\sin^2y}}
=\frac2\pi\mathrm K(k),
\]
where the last substitution is \(\sin y=k\sin\theta\).  Thus the limit in
Step~3 is exactly the asserted limit with the normalization \(\sqrt{\pi N}\).

For completeness, the differential equation follows directly from the
elliptic integral, without any difference-equation asymptotic theorem.  The
binomial expansion under the integral gives
\[
H(z):=\frac2\pi\mathrm K(\sqrt z)
=\sum_{j=0}^{\infty}b_j^2z^j,
\qquad |z|<1.
\]
The coefficient ratio gives
\(z(1-z)H''+(1-2z)H'-H/4=0\) by termwise differentiation.
For \(Y(T):=H(\sin^2(T/2))\), the chain rule yields
\[
Y''(T)+\cot T\,Y'(T)-\tfrac14Y(T)=0.
\]
Multiplication by \(\sqrt{\sin T}\) therefore gives
\[
(\sqrt{\sin T}\,Y)''+
\frac1{4\sin^2T}\sqrt{\sin T}\,Y=0.
\]
Rescaling \(T=\pi t\) proves the differential equation for \(f\).
Since \(H(0)=1\), one has \(f(t)\sim\sqrt t\), and its integral formula
shows \(f>0\).  Reduction of order supplies a second independent solution
\(f(t)\int_{t_0}^t f(u)^{-2}\,du\), with fixed \(t_0\in(0,1)\), which is
asymptotic to \(\sqrt t\log t\) at zero.  Every solution is a linear
combination of these two; the condition \(f(t)\sim\sqrt t\) fixes the first
coefficient and excludes the second.  Thus selection of the non-logarithmic
branch follows from the exact convolution, rather than from an additional
endpoint assumption.

\medskip\noindent 5)
The function \(g\) in Step~2 is decreasing on \([0,1)\), since
\(g'(x)=\pi\cot(\pi x)-1/x<0\) for \(0<x<1\).
For \(k\le N/2\), the half-mesh intervals in
\[
\log(r_{k,N}/b_k)
=\sum_{j=1}^k\int_{(j-1/2)/N}^{j/N}-g'(x)\,dx
\]
are disjoint subsets of \([0,k/N]\).  Hence
\[
0\le\log(r_{k,N}/b_k)\le-g(k/N)\le\log(\pi/2),
\qquad b_k\le r_{k,N}\le(\pi/2)b_k.
\]
The generating function \(\sum_{k\ge0}b_kz^k=(1-z)^{-1/2}\) gives
\(\sum_{k=0}^s b_kb_{s-k}=1\).  Also \(b_s\ge1/(2\sqrt s)\) for
\(s\ge1\): start from \(b_1=1/2\) and use
\((s+1/2)/(s+1)\ge\sqrt{s/(s+1)}\).
The convolution in Step~1 therefore implies, including \(s=0\),
\[
0<F_s^{(N)}\le\frac{\pi^2}{2}\sqrt{s+1},
\qquad 0\le s\le\lfloor N/2\rfloor.
\]
The \(s+1\) is essential at the first mesh cell.

For the rest of the sequence, use positivity and concavity.  Positivity for
\(s<N\) follows from Step~1.  At \(s=N\), the finite spectral measure in
\Cref{lem:root-unity-rogers} gives
\[
P_N(z)=\prod_{j=1}^N\left(z-\cos\frac{(2j-1)\pi}{2N}\right)
=2^{1-N}T_N(z),
\qquad F_N^{(N)}=2^NP_N(1)=2>0.
\]
The convolution is not used at \(s=N\), where its product factors would be
singular.  This terminal identity uses only finite spectral data.
The sine quotient in the recurrence is strictly positive; because of the
minus sign, every interior second difference of \(F_s^{(N)}\) is negative.
Its piecewise linear interpolation is therefore concave on \([0,N]\).
Let \(m=\lfloor N/2\rfloor\).  For \(s\le m\), apply concavity at \(m\)
between \(s\) and \(N\); for \(s\ge m\), apply it at \(m\) between \(0\)
and \(s\).  Dropping the positive endpoint terms gives
\[
F_s^{(N)}\le
\begin{cases}
\dfrac{N-s}{N-m}F_m^{(N)},&s\le m,\\[5pt]
\dfrac{s}{m}F_m^{(N)},&s\ge m.
\end{cases}
\]
Both factors are at most \(3\) for \(N\ge2\), while \(m+1\le N\).
The preceding half-chain bound yields
\(F_s^{(N)}\le(3\pi^2/2)\sqrt N\), as asserted.

\medskip\noindent 6)
For fixed \(0<t<1\), put \(s=\lfloor Nt\rfloor\) and
\(u_N=(N-s-1)/N\to1-t\).  The local uniform convergence already proved,
applied at \(s/N\) and \(u_N\), gives
\(G_N(t)\to f(t)f(1-t)\).  These two mesh points lie in compact subintervals
of \((0,1)\) for all sufficiently large \(N\).
Step~5 gives the uniform bound
\[
0\le G_N(t)\le M^2/\pi,
\qquad 0\le t<1,\quad N\ge2.
\]
The pointwise limit has the same bound.  Dominated convergence on the finite
interval proves convergence in every finite \(L^p\), including \(L^2\).
No bound that vanishes at \(t=0\) for fixed \(N\) is required.
\end{proof}

\begin{lemma}[Endpoint asymptotics for the finite sine lattice]
\label{lem:sine-endpoint-asymptotic}
Let
\(F_s^{(N)}:=2^sP_s^{(N)}(1)\),
where \(P_s^{(N)}\) is the monic sine-Jacobi polynomial.  Let \(f\) be the
positive non-logarithmic solution of
\[
f''(t)+\frac{\pi^2}{4\sin^2(\pi t)}f(t)=0,
\;
f(t)\sim\sqrt t\quad(t\downarrow0).
\]
With the explicit choice \(C_N=\sqrt\pi\) for every \(N\ge2\), one has
\(\frac{F_{\lfloor tN\rfloor}^{(N)}}{C_N\sqrt N}\longrightarrow f(t)\)
locally uniformly on \((0,1)\).  Moreover, if
\[
G_N(t):=
\frac{
F_{\lfloor tN\rfloor}^{(N)}
F_{N-\lfloor tN\rfloor-1}^{(N)}
}{C_N^2N},
\; 0\le t<1,
\]
then
\(G_N\to f(t)f(1-t) \; \text{in }L^2(0,1)\).
In particular,
\[
\frac1N\sum_{s=0}^{N-1}
\frac{F_s^{(N)}F_{N-s-1}^{(N)}}{C_N^2N}
\longrightarrow
\int_0^1 f(t)f(1-t)\,dt
\]
and
\[
\frac1N\sum_{s=0}^{N-1}
\left|
\frac{F_s^{(N)}F_{N-s-1}^{(N)}}{C_N^2N}
\right|^2
\longrightarrow
\int_0^1 f(t)^2f(1-t)^2\,dt.
\]
\end{lemma}

\begin{proof}[Proof of \Cref{lem:sine-endpoint-asymptotic}]
By \Cref{lem:root-unity-rogers}, the sine-Jacobi recurrence has
\(b_{s,N}=0\) and
\[
a_{s,N}^{2}
=
\frac14
\frac{\sin^2(\pi s/N)}
{\sin(\pi(s-\frac12)/N)\sin(\pi(s+\frac12)/N)}.
\]
The monic recurrence at \(x=1\) is
\(P_{s+1}^{(N)}(1) = P_s^{(N)}(1)-a_{s,N}^2P_{s-1}^{(N)}(1)\).
Multiplying by \(2^{s+1}\), and using
\(F_s^{(N)}=2^sP_s^{(N)}(1)\),
gives
\[
F_{s+1}^{(N)}-2F_s^{(N)}+F_{s-1}^{(N)}
=
-4\left(a_{s,N}^2-\frac14\right)F_{s-1}^{(N)}.
\]
Since
\[
\sin((s-\tfrac12)\pi/N)\sin((s+\tfrac12)\pi/N)
=
\sin^2(\pi s/N)-\sin^2(\pi/(2N)),
\]
we have
\[
4N^2\left(a_{s,N}^2-\frac14\right)
=
N^2
\frac{\sin^2(\pi/(2N))}
{\sin^2(\pi s/N)-\sin^2(\pi/(2N))}.
\]
Thus the endpoint sequence \(F_s^{(N)}\) satisfies exactly the recurrence in
\Cref{thm:discrete-hard-edge-LG}, with
\(F_0^{(N)}=1, \; F_1^{(N)}=2\).
Applying \Cref{thm:discrete-hard-edge-LG} with \(C_N=\sqrt\pi\) gives the local uniform convergence
and the reflected-product convergence in \(L^2(0,1)\).  The two displayed
Riemann-sum limits are the \(L^1\) and \(L^2\) consequences of that convergence
for the corresponding step functions.
\end{proof}

\begin{lemma}[Persymmetric inverse-spectral criterion]\label{lem:persym-criterion}
Let \(N\ge1\), let \(\lambda_1<\cdots<\lambda_N\) be real numbers, and set
\(P_N(z)=\prod_{j=1}^N(z-\lambda_j)\).
Let \(J\) be the \(N\times N\) Jacobi matrix with positive off-diagonal
coefficients whose first spectral measure is
\(\rho=\sum_{j=1}^Nw_j\delta_{\lambda_j}\).
Then \(w_j=c/|P_N'(\lambda_j)|\), \(j=1,\ldots,N\), with \(c>0\)
the normalizing constant, if and only if \(J\) is persymmetric:
\(RJR=J, \; R e_s=e_{N-1-s}\).
Equivalently,
\(b_s=b_{N-1-s}, \; s=0,\ldots,N-1\),
and
\(a_s=a_{N-s}, \; s=1,\ldots,N-1\).
\end{lemma}

\begin{proof}[Proof of \Cref{lem:persym-criterion}]\refstepcounter{prooflink}\label{proof:lem:persym-criterion}
Write the Jacobi matrix in the form
\[
J=
\begin{pmatrix}
b_0&a_1&0&\cdots&0\\
a_1&b_1&a_2&\ddots&\vdots\\
0&a_2&b_2&\ddots&0\\
\vdots&\ddots&\ddots&\ddots&a_{N-1}\\
0&\cdots&0&a_{N-1}&b_{N-1}
\end{pmatrix},
\;
a_s>0.
\]
For \(N=1\), the conclusion is immediate.  Assume \(N\ge2\).

Let \(v^{(j)}\) be a real normalized eigenvector of \(J\) associated with
\(\lambda_j\).  Since \(J\) is a Jacobi matrix with positive off-diagonal
entries, \(e_0\) is a cyclic vector.  Hence every eigenvalue is simple and
\(\langle e_0,v^{(j)}\rangle\ne0\).
The first spectral weight is
\(w_j=\abs{\langle e_0,v^{(j)}\rangle}^2\).
Define the last spectral weights by
\(\widetilde w_j:=\abs{\langle e_{N-1},v^{(j)}\rangle}^2\).

We first prove the product identity
\(w_j\widetilde w_j = \frac{(a_1a_2\cdots a_{N-1})^2}{P_N'(\lambda_j)^2}\).
Consider the resolvent
\((zI-J)^{-1}\).
By the spectral decomposition,
\((zI-J)^{-1} = \sum_{j=1}^N\frac{v^{(j)}(v^{(j)})^\top}{z-\lambda_j}\).
Therefore the residue of the \((0,N-1)\)-entry at \(z=\lambda_j\) is
\(\langle e_0,v^{(j)}\rangle\langle e_{N-1},v^{(j)}\rangle\).
On the other hand, Cramer's rule gives
\[
\bigl((zI-J)^{-1}\bigr)_{0,N-1}
=
\frac{(-1)^{0+N-1}\det M_{N-1,0}(z)}{\det(zI-J)},
\]
where \(M_{N-1,0}(z)\) is the minor obtained from \(zI-J\) by deleting row
\(N-1\) and column \(0\).  Because \(zI-J\) is tridiagonal, this minor is
triangular after the deletion, and its only nonzero full diagonal product is
\((-a_1)(-a_2)\cdots(-a_{N-1})\).
Thus
\((-1)^{N-1}\det M_{N-1,0}(z)=a_1a_2\cdots a_{N-1}\).
Since
\(\det(zI-J)=P_N(z)\),
we get
\(\bigl((zI-J)^{-1}\bigr)_{0,N-1} = \frac{a_1a_2\cdots a_{N-1}}{P_N(z)}\).
Taking the residue at the simple pole \(z=\lambda_j\) yields
\[
\langle e_0,v^{(j)}\rangle\langle e_{N-1},v^{(j)}\rangle
=
\frac{a_1a_2\cdots a_{N-1}}{P_N'(\lambda_j)}.
\]
Squaring gives the product identity.

Now use the special form of the first weights.  Since
\(w_j=\frac{c}{|P_N'(\lambda_j)|}\),
the product identity implies
\[
\widetilde w_j
=
\frac{(a_1\cdots a_{N-1})^2}{c}\,
\frac1{|P_N'(\lambda_j)|}.
\]
Thus there is a constant \(\widetilde c>0\), independent of \(j\), such that
\(\widetilde w_j=\frac{\widetilde c}{|P_N'(\lambda_j)|}\).
Both \((w_j)_{j=1}^N\) and \((\widetilde w_j)_{j=1}^N\) are probability
vectors, because they are the squared coordinates of normalized eigenvectors
against a fixed unit vector:
\(\sum_{j=1}^N w_j=1, \; \sum_{j=1}^N \widetilde w_j=1\).
Therefore
\[
c\sum_{j=1}^N\frac1{|P_N'(\lambda_j)|}
=
\widetilde c\sum_{j=1}^N\frac1{|P_N'(\lambda_j)|},
\]
and hence \(c=\widetilde c\).  Consequently
\(\widetilde w_j=w_j, \; j=1,\ldots,N\).

The first spectral measure of the reversed matrix \(RJR\) is precisely
\(\sum_{j=1}^N\widetilde w_j\delta_{\lambda_j}\).
Indeed, if \(Jv^{(j)}=\lambda_jv^{(j)}\), then
\((RJR)(Rv^{(j)})=\lambda_j(Rv^{(j)})\),
and the first coordinate of \(Rv^{(j)}\) is the last coordinate of
\(v^{(j)}\).  Hence \(J\) and \(RJR\) have the same eigenvalues and the same
first spectral weights.

It remains only to recall why this spectral data determines the Jacobi matrix.
Given distinct points \(\lambda_j\) and positive weights \(w_j\), define an
inner product on polynomials of degree at most \(N-1\) by
\(\langle p,q\rangle_\rho = \sum_{j=1}^N w_jp(\lambda_j)q(\lambda_j)\).
Applying Gram--Schmidt to
\(1,z,z^2,\ldots,z^{N-1}\)
produces a unique sequence of orthonormal polynomials with positive leading
coefficients.  In this orthonormal basis, multiplication by \(z\) is represented
by a unique Jacobi matrix with positive off-diagonal coefficients.  Its first
spectral measure is exactly \(\rho\).  Therefore two Jacobi matrices with the
same distinct support points and the same positive first weights must coincide.

Applying this uniqueness to \(J\) and \(RJR\) gives
\(RJR=J\).
The coefficient identities follow by comparing the diagonal and off-diagonal
entries of \(J\) with those of \(RJR\):
\(b_s=b_{N-1-s}, \; s=0,\ldots,N-1\),
and
\(a_s=a_{N-s}, \; s=1,\ldots,N-1\).

Conversely, if \(RJR=J\), the first and last spectral measures agree, so
\(\widetilde w_j=w_j>0\).  The product identity already proved yields
\[
w_j^2=\frac{(a_1\cdots a_{N-1})^2}{P_N'(\lambda_j)^2},
\qquad
w_j=\frac{a_1\cdots a_{N-1}}{|P_N'(\lambda_j)|}.
\]
Since the weights sum to one, the numerator is precisely the normalizing
constant \(c\).  For \(N=1\), the same conclusion holds with the empty
product equal to one.
\end{proof}

\begin{lemma}[Symmetry and final polynomial of the sine-Jacobi matrix]\label{lem:sine-symmetry}
For the sine weights, the Jacobi matrix \(J_N^{\sin}\) satisfies
\(b_s=0, \; s=0,\ldots,N-1\),
and
\(a_s=a_{N-s}, \; s=1,\ldots,N-1\).
Its characteristic polynomial is
\(P_N(z)=\det(zI-J_N^{\sin})=2^{1-N}T_N(z)\).
\end{lemma}

\begin{proof}[Proof of \Cref{lem:sine-symmetry}]\refstepcounter{prooflink}\label{proof:lem:sine-symmetry}
The support points of the sine measure are
\[
\zeta_j=\cos\theta_j,
\;
\theta_j=\frac{(2j-1)\pi}{2N},
\;
j=1,\ldots,N.
\]
These are precisely the zeros of \(T_N\).  Since \(T_N\) has leading
coefficient \(2^{N-1}\) for \(N\ge1\), the monic polynomial with these zeros is
\(\prod_{j=1}^N(z-\zeta_j)=2^{1-N}T_N(z)\).
Thus the characteristic polynomial of the Jacobi matrix associated with the
measure \(\rho_N^{\sin}\) is
\(P_N(z)=\det(zI-J_N^{\sin})=2^{1-N}T_N(z)\),
because the spectrum of this Jacobi matrix is exactly the support of the
measure.

We next verify the hypothesis of \Cref{lem:persym-criterion}.  Differentiating
the preceding identity gives
\(P_N'(z)=2^{1-N}T_N'(z)\).
Using
\(T_N'(z)=NU_{N-1}(z)\)
and
\(U_{N-1}(\cos\theta)=\frac{\sin(N\theta)}{\sin\theta}\),
we obtain, at \(\zeta_j=\cos\theta_j\),
\(P_N'(\zeta_j) = 2^{1-N}N \frac{\sin(N\theta_j)}{\sin\theta_j}\).
Since
\(N\theta_j=\frac{(2j-1)\pi}{2}\),
we have
\(\abs{\sin(N\theta_j)}=1\).
Also \(\theta_j\in(0,\pi)\), so \(\sin\theta_j>0\).  Hence
\(\abs{P_N'(\zeta_j)} = 2^{1-N}N\frac1{\sin\theta_j}\).
The sine weights satisfy
\(w_j = \frac{\sin\theta_j}{\sum_{k=1}^N\sin\theta_k}\).
Therefore
\(w_j = \frac{c}{\abs{P_N'(\zeta_j)}}\)
for a positive normalizing constant \(c\).  By
\Cref{lem:persym-criterion}, the Jacobi matrix \(J_N^{\sin}\) is
persymmetric.  Consequently
\(a_s=a_{N-s}, \; s=1,\ldots,N-1\).
The same criterion also gives \(b_s=b_{N-1-s}\), although the stronger
identity \(b_s=0\) will now be proved directly from the evenness of the
measure.

The sine measure is even.  Indeed,
\(\theta_{N+1-j} = \frac{(2(N+1-j)-1)\pi}{2N} = \pi-\theta_j\),
and therefore
\[
\zeta_{N+1-j}
=
\cos(\pi-\theta_j)
=
-\zeta_j,
\;
\sin\theta_{N+1-j}=\sin\theta_j.
\]
Hence
\(w_{N+1-j}=w_j\).
Equivalently, for every polynomial \(F\),
\(\sum_{j=1}^N w_jF(\zeta_j) = \sum_{j=1}^N w_jF(-\zeta_j)\).

Let \(P_s\) be the monic orthogonal polynomial of degree \(s\) for this
measure, \(0\le s\le N\).  We prove its parity.  Define
\(Q_s(z):=(-1)^sP_s(-z)\).
Then \(Q_s\) is also monic of degree \(s\).  If \(q\) is any polynomial of
degree \(<s\), then by the evenness of the measure,
\[
\langle Q_s,q\rangle_{\sin,N}
=
(-1)^s\langle P_s(-z),q(z)\rangle_{\sin,N}
=
(-1)^s\langle P_s(z),q(-z)\rangle_{\sin,N}.
\]
Since \(q(-z)\) has degree \(<s\), the last inner product is zero by the
orthogonality of \(P_s\).  Therefore \(Q_s\) is a monic degree-\(s\)
polynomial orthogonal to all lower-degree polynomials.  By uniqueness of the
monic orthogonal polynomial of degree \(s\), we have
\(Q_s=P_s\).
Thus
\(P_s(-z)=(-1)^sP_s(z)\).

For \(s=0,\ldots,N-1\), the diagonal Jacobi coefficient is
\[
b_s=
\frac{\langle zP_s,P_s\rangle_{\sin,N}}
     {\langle P_s,P_s\rangle_{\sin,N}}.
\]
The denominator is positive.  The numerator is zero because
\(zP_s(z)^2\)
is an odd polynomial and the measure is even.  Hence
\(b_s=0, \; s=0,\ldots,N-1\).
This proves all stated identities.
\end{proof}

\begin{lemma}[Positivity of the normalization at \(\eta\)]\label{lem:positive-eta-normalization}
Let \(J\) be an \(N\times N\) Jacobi matrix with spectrum contained in
\([-1,1]\).  Let \(J^{(s)}\) denote its \(s\times s\) leading principal
submatrix, and define the monic polynomials
\(P_0(z):=1, \; P_s(z):=\det(zI_s-J^{(s)}), \; s=1,\ldots,N\).
Then
\(P_s(\eta)>0, \; s=0,\ldots,N, \; \eta>1\).
In particular, the normalizing factors \(P_s(\eta)\) of the sine-Jacobi
method are strictly positive.
\end{lemma}

\begin{proof}[Proof of \Cref{lem:positive-eta-normalization}]\refstepcounter{prooflink}\label{proof:lem:positive-eta-normalization}
For \(s=0\), the claim is immediate from
\(P_0(\eta)=1\).
Fix \(1\le s\le N\).  The polynomial
\(P_s(z)=\det(zI_s-J^{(s)})\)
is monic of degree \(s\), and its zeros are exactly the eigenvalues of the
self-adjoint matrix \(J^{(s)}\).  Denote these eigenvalues by
\(\xi_1^{(s)},\ldots,\xi_s^{(s)}\).
By Cauchy's interlacing theorem for Hermitian principal submatrices,
\(\min\spec(J)\le \xi_i^{(s)}\le \max\spec(J), \; i=1,\ldots,s\).
Since
\(\spec(J)\subset[-1,1]\),
we have
\(\xi_i^{(s)}\in[-1,1], \; i=1,\ldots,s\).
Therefore, for every \(\eta>1\),
\(\eta-\xi_i^{(s)}>0, \; i=1,\ldots,s\).
Using the factorization of the monic characteristic polynomial,
\(P_s(\eta) = \prod_{i=1}^s\bigl(\eta-\xi_i^{(s)}\bigr)\),
we obtain
\(P_s(\eta)>0\).
This proves the assertion for all \(s=0,\ldots,N\).
\end{proof}

\begin{proof}[Proof of \Cref{thm:sine-lattice-package}]\refstepcounter{prooflink}\label{proof:thm:sine-lattice-package}
For \(N=1\) all assertions are immediate: the recurrence and asymptotic statements are vacuous and the maximum principle says \(|P_0|\le P_0(1)=1\).  Assume henceforth \(N\ge2\).

The identities in (i) follow directly from \Cref{lem:sine-symmetry}.

We prove (ii).  Set
\(h:=\frac{\pi}{N}\).
By \Cref{lem:root-unity-rogers}, the monic orthogonal polynomials of the
measure \(\rho_N^{\sin}\) satisfy
\[
P_{s+1}^{(N)}(x)
=
xP_s^{(N)}(x)
-
a_{s,N}^2P_{s-1}^{(N)}(x),
\; 1\le s\le N-1,
\]
with
\[
a_{s,N}^{2}
=
\frac14
\frac{\sin^2(sh)}
{\sin((s-\frac12)h)\sin((s+\frac12)h)}.
\]
Since \(h=\pi/N\), this is exactly
\[
a_{s,N}^{2}
=
\frac14\frac{\sin^2(\pi s/N)}
{\sin(\pi(s-\frac12)/N)\sin(\pi(s+\frac12)/N)}.
\]
This proves (ii).

We next prove (iii).  Write
\(t:=\frac{s}{N}\).
Using
\(\sin(A-B)\sin(A+B)=\sin^2 A-\sin^2 B\)
with
\(A=\frac{\pi s}{N}, \; B=\frac{\pi}{2N}\),
the formula in (ii) becomes
\[
a_{s,N}^2
=
\frac14
\frac{\sin^2(\pi t)}
{\sin^2(\pi t)-\sin^2(\pi/(2N))}.
\]
Let \(K\subset(0,1)\) be compact.  Then there is a constant \(m_K>0\) such
that
\(\sin^2(\pi t)\ge m_K \; \text{for all }t\in K\).
Moreover,
\(\sin^2\frac{\pi}{2N} = \frac{\pi^2}{4N^2}+O(N^{-4})\).
Uniformly for \(t\in K\),
\(\frac{\sin^2(\pi/(2N))}{\sin^2(\pi t)} = O_K(N^{-2})\).
Therefore, using \((1-u)^{-1}=1+u+O(u^2)\) uniformly for
\(|u|\le C_KN^{-2}\),
\[
\begin{aligned}
a_{s,N}^2
&=
\frac14
\left(
1-
\frac{\sin^2(\pi/(2N))}{\sin^2(\pi t)}
\right)^{-1} \\
&=
\frac14
\left(
1+
\frac{\pi^2}{4N^2\sin^2(\pi t)}
+
O_K(N^{-4})
\right) \\
&=
\frac14+
\frac1{N^2}\frac{\pi^2}{16\sin^2(\pi t)}
+
O_K(N^{-4}).
\end{aligned}
\]
This proves both the stated asymptotic expansion and the displayed limit.

It remains to prove (iv).  This is the only place where the positivity part of \Cref{lem:root-unity-rogers} is used.  By that lemma, for
\(0\le s\le N-1\), the endpoint-normalized polynomial has the positive
trigonometric expansion
\[
\frac{P_s^{(N)}(\cos\theta)}{P_s^{(N)}(1)}
=
\sum_{k=0}^s c_{s,k}^{(N)}\cos((s-2k)\theta),
\]
where
\(c_{s,k}^{(N)}\ge0, \; \sum_{k=0}^s c_{s,k}^{(N)}=1\).
Hence, for every real \(\theta\),
\[
\left|
\frac{P_s^{(N)}(\cos\theta)}{P_s^{(N)}(1)}
\right|
\le
\sum_{k=0}^s c_{s,k}^{(N)}
\abs{\cos((s-2k)\theta)}
\le
1.
\]
Since every \(x\in[-1,1]\) can be written as \(x=\cos\theta\), this proves
\(|P_s^{(N)}(x)|\le P_s^{(N)}(1), \; -1\le x\le1\).

We now prove the endpoint identity for the two-variable sum.  By
\Cref{lem:sine-symmetry}, \(J_N^{\sin}\) is persymmetric.  Therefore the right
tail block on the sites \(s+1,\ldots,N-1\) is unitarily identified, by reversal
of order, with the leading principal block of size \(N-s-1\).  Consequently
\(Q_{N-s-1}^{(s+1)}(x)=P_{N-s-1}^{(N)}(x), \; 0\le s\le N-1\).
Applying the already proved maximum principle to \(P_{N-s-1}^{(N)}\), we get
\(|Q_{N-s-1}^{(s+1)}(x)| \le Q_{N-s-1}^{(s+1)}(1), \; -1\le x\le1\).
Together with
\(|P_s^{(N)}(y)|\le P_s^{(N)}(1), \; -1\le y\le1\),
this gives
\[
\sum_{s=0}^{N-1}|P_s^{(N)}(y)Q_{N-s-1}^{(s+1)}(x)|^2
\le
\sum_{s=0}^{N-1}
P_s^{(N)}(1)^2Q_{N-s-1}^{(s+1)}(1)^2.
\]
The reverse inequality follows by evaluating the left-hand side at
\(x=1, \; y=1\).
Thus the supremum equals the stated endpoint value.
\end{proof}

\subsection{Endpoint profile, continuum constant, and upper bound}

\begin{proof}[Proof of \Cref{thm:sine-endpoint-profile}]\refstepcounter{prooflink}\label{proof:thm:sine-endpoint-profile}
Let
\(P_s(z)=P_s^{(N)}(z), \; s=0,\ldots,N\),
be the monic leading polynomials of \(J_N^{\sin}\).  For \(0\le s\le N-1\),
let \(Q_{N-s-1}^{(s+1)}\) denote the monic characteristic polynomial of the
right tail block on the sites \(s+1,\ldots,N-1\).

The diagonal cofactor formula for the resolvent gives, for \(z\notin\spec(J_N^{\sin})\),
\[
\bigl[(zI-J_N^{\sin})^{-1}\bigr]_{ss}
=
\frac{P_s(z)Q_{N-s-1}^{(s+1)}(z)}{P_N(z)}.
\]
Indeed, deleting the \(s\)-th row and the \(s\)-th column of the tridiagonal
matrix \(zI-J_N^{\sin}\) splits the minor into the product of the left leading
block and the right tail block.  The spectrum of \(J_N^{\sin}\) consists of the zeros of \(T_N\), all of which
belong to \((-1,1)\).  Hence \(1\notin\spec(J_N^{\sin})\).  Therefore
\[
d_s=
\bigl[(I-J_N^{\sin})^{-1}\bigr]_{ss}
=
\frac{P_s(1)Q_{N-s-1}^{(s+1)}(1)}{P_N(1)}.
\]

By \Cref{lem:sine-symmetry}, \(J_N^{\sin}\) is persymmetric.  Hence the right
tail block on sites \(s+1,\ldots,N-1\) is the reversal of the leading principal
block of size \(N-s-1\), and so
\(Q_{N-s-1}^{(s+1)}(z)=P_{N-s-1}(z)\).
Also, by the same lemma,
\(P_N(z)=2^{1-N}T_N(z)\),
and therefore
\(P_N(1)=2^{1-N}\).
Define
\(F_s^{(N)}:=2^sP_s(1)\).
Then
\(d_s = \frac{P_s(1)P_{N-s-1}(1)}{2^{1-N}} = F_s^{(N)}F_{N-s-1}^{(N)}\).

Set
\[
G_N(t):=
\frac{F_{\lfloor Nt\rfloor}^{(N)}F_{N-\lfloor Nt\rfloor-1}^{(N)}}{\pi N},
\qquad 0\le t<1.
\]
By \Cref{lem:sine-endpoint-asymptotic}, with \(C_N=\sqrt\pi\),
\[
G_N\longrightarrow g\quad\text{in }L^2(0,1),
\qquad g(t):=f(t)f(1-t).
\]
In particular, this convergence also holds in \(L^1(0,1)\).

The normalization can now be computed exactly.  The logarithmic derivative
of the characteristic polynomial gives
\[
\sum_{s=0}^{N-1}d_s
=\tr(I-J_N^{\sin})^{-1}
=\frac{P_N'(1)}{P_N(1)}
=\frac{T_N'(1)}{T_N(1)}
=N^2.
\]
Since \(d_s=F_s^{(N)}F_{N-s-1}^{(N)}\), it follows that
\[
\int_0^1G_N(t)\,dt
=\frac1{\pi N^2}\sum_{s=0}^{N-1}d_s
=\frac1\pi.
\]
Passing to the \(L^1\) limit gives
\[
I_{\sin}=\int_0^1f(t)f(1-t)\,dt=\frac1\pi,
\qquad
\varphi_{\sin}(t)=\pi f(t)f(1-t).
\]
For \(0\le t<1\),
\(D_N(t)=d_{\lfloor Nt\rfloor}/N=\pi G_N(t)\).  Hence
\(D_N\to\varphi_{\sin}\) in \(L^2(0,1)\).  Taking squared \(L^2\) norms
therefore yields
\[
\frac1{N^3}\sum_{s=0}^{N-1}d_s^2
=\int_0^1D_N(t)^2\,dt
\longrightarrow
\pi^2\int_0^1f(t)^2f(1-t)^2\,dt
=\int_0^1\varphi_{\sin}(t)^2\,dt
=c_{\sin}.
\]
\end{proof}

\begin{proof}[Proof of \Cref{thm:sine-upper}]\refstepcounter{prooflink}\label{proof:thm:sine-upper}
Final-Chebyshev exactness follows from \Cref{thm:sine-lattice-package}(i) and
\Cref{thm:jacobi-method}.

It remains to prove the finite identity for \(A_N(J_N^{\sin})\).  The endpoint
supremum identity is precisely \Cref{thm:sine-lattice-package}(iv).  The
conversion to resolvent diagonals is as follows.
By \Cref{thm:sine-lattice-package}(iv),
\[
\sup_{x,y\in[-1,1]}
\sum_{s=0}^{N-1}|P_s(y)Q_{N-s-1}^{(s+1)}(x)|^2
=
\sum_{s=0}^{N-1}P_s(1)^2Q_{N-s-1}^{(s+1)}(1)^2.
\]

We now convert the endpoint value into the resolvent-diagonal expression.
Since \(\spec(J_N^{\sin})\) consists of the zeros of \(T_N\), the point \(1\)
is not an eigenvalue of \(J_N^{\sin}\).  By Cramer's rule, the diagonal entry
of \((I-J_N^{\sin})^{-1}\) is the corresponding diagonal cofactor divided by
\(\det(I-J_N^{\sin})=P_N(1)\).  Deleting row and column \(s\) in the tridiagonal
matrix splits the minor into the leading block and the right-tail block.
Therefore, for every \(s=0,\ldots,N-1\),
\[
d_s
=
\bigl[(I-J_N^{\sin})^{-1}\bigr]_{ss}
=
\frac{P_s(1)Q_{N-s-1}^{(s+1)}(1)}{P_N(1)}.
\]
Equivalently,
\(P_s(1)Q_{N-s-1}^{(s+1)}(1)=P_N(1)d_s\).
Hence
\[
\sum_{s=0}^{N-1}P_s(1)^2Q_{N-s-1}^{(s+1)}(1)^2
=
P_N(1)^2\sum_{s=0}^{N-1}d_s^2.
\]

By \Cref{thm:jacobi-first-variation-kernel},
\[
A_N(J_N^{\sin})
=
\frac{2}{\Delta |P_N(\eta)|}
\sup_{x,y\in[-1,1]}
\left(
\sum_{s=0}^{N-1}
\abs{Q_{N-s-1}^{(s+1)}(x)P_s(y)}^2
\right)^{1/2}.
\]
Using the endpoint identity just proved, this becomes
\[
A_N(J_N^{\sin})
=
\frac{2|P_N(1)|}{\Delta |P_N(\eta)|}
\left(\sum_{s=0}^{N-1}d_s^2\right)^{1/2}.
\]
By \Cref{lem:sine-symmetry},
\(P_N(z)=2^{1-N}T_N(z)\).
Therefore
\(P_N(1)=2^{1-N}T_N(1)=2^{1-N}\),
and
\(P_N(\eta)=2^{1-N}T_N(\eta)\).
Since \(\eta>1\), one has \(T_N(\eta)>0\), and hence
\(\frac{|P_N(1)|}{|P_N(\eta)|} = \frac1{T_N(\eta)} = \epsilon_N^\star\).
Substitution gives
\[
A_N(J_N^{\sin})
=
\frac{2\epsilon_N^\star}{\Delta}
\left(\sum_{s=0}^{N-1}d_s^2\right)^{1/2}.
\]

By \Cref{thm:sine-endpoint-profile},
\(\frac1{N^3}\sum_{s=0}^{N-1}d_s^2 \longrightarrow c_{\sin}\).
Since
\(c_{\sin}=\int_0^1\varphi_{\sin}(t)^2\,dt>0\),
we may rewrite the last convergence as
\(\sum_{s=0}^{N-1}d_s^2 = c_{\sin}N^3(1+o(1))\).
Taking square roots gives
\[
\left(\sum_{s=0}^{N-1}d_s^2\right)^{1/2}
=
\sqrt{c_{\sin}}\,N^{3/2}(1+o(1)).
\]
Substituting this into the exact identity for \(A_N(J_N^{\sin})\), we obtain
\[
A_N(J_N^{\sin})
=
\frac{2\epsilon_N^\star}{\Delta}
\sqrt{c_{\sin}}\,N^{3/2}(1+o(1)).
\]
This is the stated asymptotic formula.
\end{proof}

\subsection{Finite drift for the sine chain}

\begin{lemma}[Fourier positivity of the sine blocks]
\label{lem:sine-block-positivity}
Let \(J=J_N^{\sin}\) and \(0\le a\le b\le N-1\).  Write
\(D_{a:b}(\cos\theta)=\sum_jd_j\mathrm e^{ij\theta}\).  Then \(d_j\ge0\) for
all \(j\).  Consequently \(|D_{a:b}(x)|\le D_{a:b}(1)\) for
\(x\in[-1,1]\).
\end{lemma}

\begin{proof}
The consequence follows from
\(|D_{a:b}(\cos\theta)|\le\sum_jd_j=D_{a:b}(1)\).

\emph{End blocks.}  The full block is
\(D_{0:N-1}=2^{1-N}T_N\), with coefficients \(2^{-N}\) at
\(\mathrm e^{\pm iN\theta}\).  For \(b\le N-2\),
\(D_{0:b}=P_{b+1}\) with \(b+1\le N-1\), and
\Cref{lem:root-unity-rogers} gives a cosine expansion with nonnegative
coefficients.  By \Cref{lem:sine-symmetry}, \(a_i=a_{N-i}\) and the diagonal
vanishes, so the block on \(a,\ldots,N-1\) is the leading block of length
\(N-a\) in reversed order; hence \(D_{a:N-1}=P_{N-a}\), which is covered for
\(a\ge1\).

\emph{Interior blocks.}  Let \(1\le a\le b\le N-2\) and
\(\ell:=b-a+1\le N-2\).  We first prove the continuant identity
\begin{equation}\label{eq:continuant-identity}
P_N\,D_{a:b}
=P_{b+1}\,D_{a:N-1}
-\Bigl(\prod_{k=a}^{b+1}a_k^2\Bigr)D_{0:a-2}\,D_{b+2:N-1},
\end{equation}
with \(D_{0:-1}=D_{N:N-1}=1\).  Fix \(a\) and, for \(a-2\le m\le N-2\), let
\(F(m):=P_{m+1}D_{a:N-1}-P_ND_{a:m}\) and
\(G(m):=(\prod_{k=a}^{m+1}a_k^2)D_{0:a-2}D_{m+2:N-1}\).  By
\eqref{eq:block-recurrence} and the recurrence for \(P_s\),
\(F(m+1)=xF(m)-a_{m+1}^2F(m-1)\).  Expanding \(D_{m+1:N-1}\) along its first
row gives
\(a_{m+2}^2D_{m+3:N-1}=xD_{m+2:N-1}-D_{m+1:N-1}\), hence
\(G(m+1)=xG(m)-a_{m+1}^2G(m-1)\).  At \(m=a-2\), \(D_{a:a-2}=0\) and
\(F=G=P_{a-1}D_{a:N-1}\).  At \(m=a-1\), \(D_{a:a-1}=1\), and splitting
\(P_N\) across the coupling \(a_a\) gives
\(P_N=P_aD_{a:N-1}-a_a^2P_{a-1}D_{a+1:N-1}\), so \(F=G\).  Hence \(F=G\) for
all \(m\), which is \eqref{eq:continuant-identity} at \(m=b\).

Now put \(x=\cos\theta\) and \(z=\mathrm e^{i\theta}\).  The left side of
\eqref{eq:continuant-identity} is
\(2^{-N}(z^N+z^{-N})\sum_{|j|\le\ell}d_jz^j\).  Its exponents lie in
\([N-\ell,N+\ell]\cup[-N-\ell,-N+\ell]\), and these sets are disjoint because
\(\ell<N\).  Thus \(2^{-N}d_j\) is the coefficient of \(z^{N+j}\).  The last
term of \eqref{eq:continuant-identity} has degree
\((a-1)+(N-b-2)=N-\ell-2\) in \(x\), so it contains no exponent
\(\ge N-\ell\).  Therefore
\[
d_j=2^N\,[z^{N+j}]\bigl(P_{b+1}(\cos\theta)\,P_{N-a}(\cos\theta)\bigr),
\qquad |j|\le\ell,
\]
where \(D_{a:N-1}=P_{N-a}\) was used.  Since \(b+1\le N-1\) and
\(N-a\le N-1\), both factors have nonnegative coefficients by
\Cref{lem:root-unity-rogers}, and so does their product.  Hence
\(d_j\ge0\).
\end{proof}

\begin{proof}[Proof of \Cref{cor:sine-finite-drift}]\refstepcounter{prooflink}\label{proof:cor:sine-finite-drift}
The scaled blocks \(\widetilde D_{a:b}=2^{\,b-a+1}D_{a:b}\) differ from
\(D_{a:b}\) by positive factors, so \Cref{lem:sine-block-positivity} shows
that every block attains its maximum on \([-1,1]\) at \(x=1\).
\Cref{thm:finite-drift}(iii) and (iv) then give the bounds and their
attainment.
\end{proof}

\begin{proof}[Proof of \Cref{cor:lipschitz-hessian}]\refstepcounter{prooflink}\label{proof:cor:lipschitz-hessian}
Let \(e_s=x_s-x^\star\) and \(H=\nabla^2f(x^\star)\).  Since
\(\nabla f(x^\star)=0\),
\(\nabla f(x_s)=\bar H_se_s\) with
\(\bar H_s=\int_0^1\nabla^2f(x^\star+te_s)\,dt\), and
\(\|\bar H_s-H\|\le\int_0^1\rho t\|e_s\|\,dt\le\rho R/2\).  Both methods are
linear in the queried gradients, so their errors obey the recurrence of
\Cref{thm:jacobi-method} with \(H_s=\bar H_s\); hence
\(e_N=\TT(\bar H_0,\ldots,\bar H_{N-1})e_0\).  Apply
\Cref{cor:prefix-finite-drift} and \Cref{cor:sine-finite-drift} with
\(E_s=\bar H_s-H\) and \(\delta=\rho R/2\), together with
\(\|p_N^\star(H)\|\le\epsilon_N^\star\).
\end{proof}

\section{Additional experimental diagnostics}
\label{app:experiments}

This appendix records the diagnostics behind the main experimental figures.  The same exact prefix and sine--Jacobi implementations are used throughout.

\subsection{Stochastic collapse across intermediate horizons}
\label{app:experiments-stochastic-collapse}

The stochastic endpoint-coupled oracle of Subsection~\ref{subsec:experiments-frontier} is evaluated over the full grid.  The underlying deterministic task is always the same diagonal quadratic \(f_H(x)=\frac12x^\top Hx\) with \(H=\operatorname{diag}(1,\sqrt{120},120)\), initial vector proportional to \((1,0.25,1)\), and stochastic Hessian channel \(B=e_1e_3^\top+e_3e_1^\top\).  The only quantities varied in this diagnostic are the horizon and the batch/noise budget:
\(N\in\{16,18,\ldots,100\}, \; b\in\{64,72,80,\ldots,2048\}\).
The natural collapsed scale is
\(x=\frac{N^3}{b\Delta^2}\).
The leading first-variation variance predicts a contribution proportional to \(x\).  Since the stress test uses \(\sigma=2\) and includes moderately small batches and large horizons, the measured overhead also contains the next perturbative contribution.  The fit in Figure~\ref{fig:app-stochastic-collapse} therefore uses
\(\mathrm{oh}_R(N,b)=c_{1,R}x+c_{2,R}x^2\).
The fitted coefficients are
\[
\begin{aligned}
        c_{1,{\rm pref}}&=20.59, & c_{2,{\rm pref}}&=25.51,\\
        c_{1,{\rm sine}}&=16.84, & c_{2,{\rm sine}}&=24.24.
\end{aligned}
\]
The pooled robust median slopes are \(21.85\) for prefix and \(18.78\) for sine, giving ratio \(0.860\).  This agrees with the fixed-horizon design-tail ratio \(0.861\) in Figure~\ref{fig:exp-stochastic}.

\begin{figure}[ht!]
    \centering
    \includegraphics[width=0.75\linewidth]{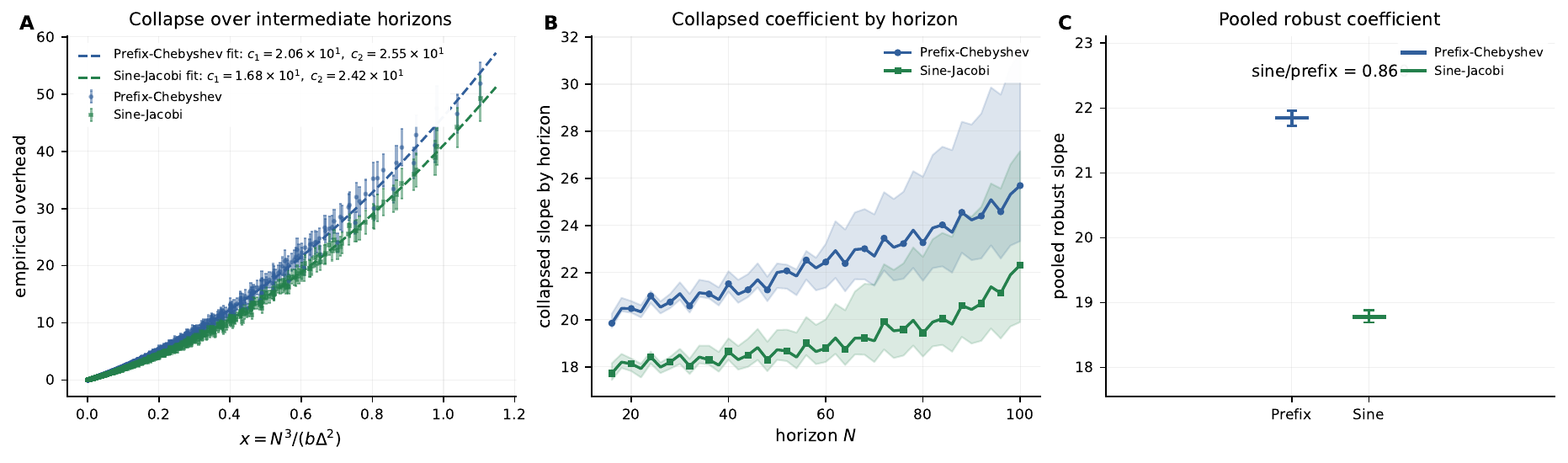}
    \caption{Stochastic curvature collapse over intermediate horizons.  Panel A plots measured overhead against \(N^3/(b\Delta^2)\) and fits the perturbative form \(c_1x+c_2x^2\).  Panel B gives horizon-wise robust collapsed slopes.  Panel C compares the pooled robust coefficients with bootstrap intervals.}
    \label{fig:app-stochastic-collapse}
\end{figure}

The curvature visible in Panel A has the expected sign under the independent,
centered oracle model.  For fixed \(N\), the terminal error is multilinear in
the variables \(\sigma z_s/\sqrt b\).  Terms indexed by distinct subsets of
insertion times are orthogonal in expectation, because some centered
independent variable occurs only once in their cross product.  Hence the
expected squared norm has nonnegative coefficients in powers of
\(\sigma^2/b\); the first-variation term is linear and the next coefficient
is nonnegative.  The pooled fit in the collapsed scale is an empirical
summary, not an exact quadratic identity.  The coefficient gap remains visible in the linear terms and in the pooled robust slopes.  The quadratic coefficients are comparable for the two realizations, which indicates that the main realization-level separation in this stress test is already present in the leading first-variation component.

\subsection{Restart-distribution diagnostics for GLM blocks}
\label{app:experiments-glm-restart-survival}

Figure~\ref{fig:app-glm-restart-survival} gives a distributional view of the
restarted-continuation experiment from Subsection~\ref{subsec:experiments-glm}.
The underlying objective, endpoint-coupled directions, acceptance rule, horizon
cap, and target radius are the same as in Figure~\ref{fig:exp-glm}.  Instead of
plotting the restart statistics as functions of the initial radius, the figure
summarizes their empirical survival functions over the tested starts.  For a
threshold \(t\), the plotted value is the fraction of starts whose diagnostic
value is at least \(t\).  The solid and dashed curves are medians over fixed
logarithmic initial-radius shells, the shaded regions are interquartile ranges
over those shells, and the dotted curves show the pooled empirical survival
functions over all tested starts.

\begin{figure}[ht!]
    \centering
    \includegraphics[width=0.75\linewidth]{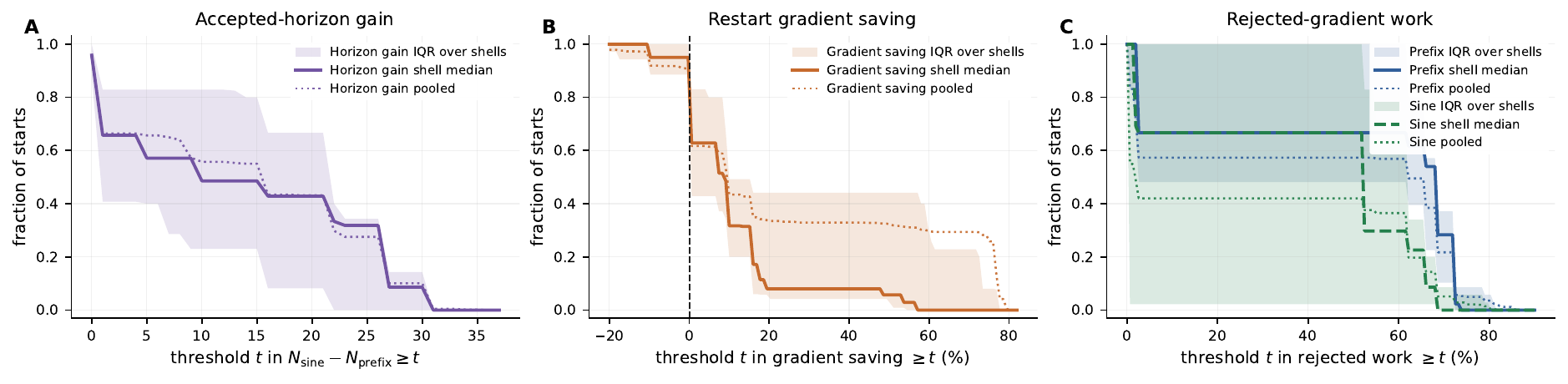}
    \caption{Distributional restart diagnostics for endpoint-coupled logistic GLM
    blocks.  Panel A shows the survival function of the first accepted horizon
    gain.  Panel B shows the survival function of the end-to-end gradient
    saving.  Panel C shows the survival function of the rejected-gradient
    fraction.  Solid and dashed curves are shell-wise medians, shaded regions
    are interquartile ranges over fixed logarithmic initial-radius shells, and
    dotted curves are the pooled empirical survival functions over all tested
    starts.}
    \label{fig:app-glm-restart-survival}
\end{figure}

This distributional view removes the ordering of the initial radii from the
visualization and records the same restart outcomes as empirical proportions.
It shows that the sine--Jacobi realization shifts the tested GLM starts toward
larger accepted horizons, larger end-to-end gradient savings, and smaller
rejected-gradient fractions.  The shell-wise interquartile bands also show that
the effect is not concentrated in a single radius shell.

\subsection{First-variation operator diagnostics}
\label{app:experiments-operator}

Figure~\ref{fig:app-operator} examines the first-variation operator itself at \(N=96\).  This is not a trajectory experiment: it is a direct computation of the kernels from Theorems~\ref{thm:prefix-kernel} and~\ref{thm:jacobi-first-variation-kernel}.  The spectral variables \(\lambda_i,\nu_j\) are placed on a Chebyshev grid in \([\mu,L]\), and the operator is represented as the matrix whose rows are the time profiles
\(\big(K_0^R(\lambda_i,\nu_j),\ldots,K_{N-1}^R(\lambda_i,\nu_j)\big)\).
Panel A plots the normalized singular spectrum.  Panels B and C focus on the worst spectral pair and plot the time-sensitivity density and its cumulative energy.

\begin{figure}[ht!]
    \centering
    \includegraphics[width=0.75\linewidth]{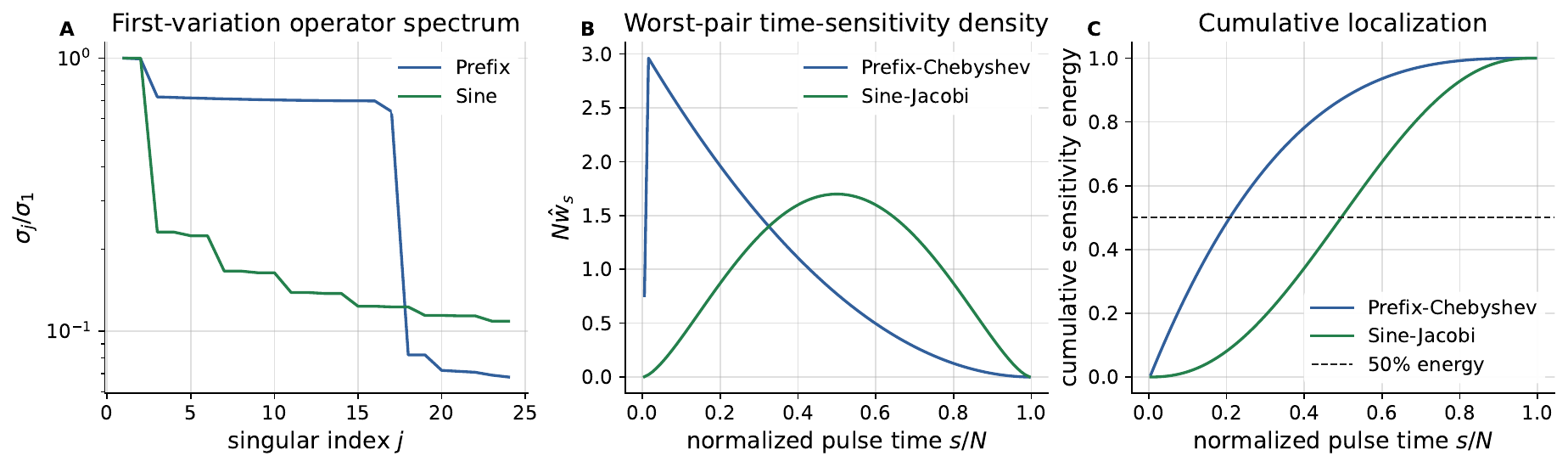}
    \caption{Operator-level first-variation diagnostics at \(N=96\).  Panel A compares normalized singular spectra.  Panel B shows the worst-pair time-sensitivity density \(N\widehat w_s\).  Panel C plots the cumulative time-sensitivity energy.}
    \label{fig:app-operator}
\end{figure}

The operator diagnostic separates two effects.  First, the dominant singular spectrum is more concentrated for the prefix realization, reflecting the stronger endpoint-dominated time-ordered response.  Second, the worst-pair time energy of the sine--Jacobi chain is less concentrated near the boundary times.  This is the mechanism behind the smaller \(\ell_2\)-in-time first-variation gain: the sine weights do not change the terminal Chebyshev polynomial, but they redistribute the intermediate sensitivity profile of the ordered product.

\subsection{Smooth nonlinear transfer family}
\label{app:experiments-nonlinear-transfer}

The nonlinear transfer experiment uses a two-dimensional smooth strongly convex family.  It is designed as a deterministic analogue of endpoint curvature coupling: the quadratic part fixes the same spectral endpoints used in the packet experiment, while the log-cosh ridge terms introduce smooth state-dependent curvature along directions \((1,q_i)\).  The objective is
\[
        f_{\rho,i}(x)
        =\frac12(\lambda x_1^2+\nu x_2^2)
        +\rho\sum_{m=1}^{25}w_{i,m}\tau^2
        \log\cosh\!\left(\frac{x_1+q_i x_2-b_{i,m}}{\tau}\right),
\]
with \(\lambda=200\), \(\nu=8000\), \(\tau=0.20\), horizon \(N=256\), and initial point \((0,1)\).  Both realizations use coefficients calibrated to \([\mu,L]=[200,8000]\).
This is the interval for the quadratic baseline; the added positive
semidefinite ridge Hessian need not stay below \(L\).
For profile \(i\), the centers \(b_{i,m}/q_i\) are placed on \(25\) equally spaced points in \([0.05,0.95]\), and the weights are a normalized Gaussian envelope with profile-dependent center and width.  The five triples \((\text{center},\text{width},q_i)\) are
\[
        (0.55,0.13,2.20),\quad (0.62,0.15,2.45),\quad
        (0.70,0.17,2.65),\quad (0.78,0.14,2.85),\quad
        (0.66,0.11,2.55).
\]
The coupling parameter is sampled on \(121\) values of \(\rho\in[90,260]\).  For each \((\rho,i)\), the minimizer is computed by L-BFGS-B with a warm start along the \(\rho\)-grid, and both realizations are run from the same initial point.

\begin{figure}[ht!]
    \centering
    \includegraphics[width=0.75\linewidth]{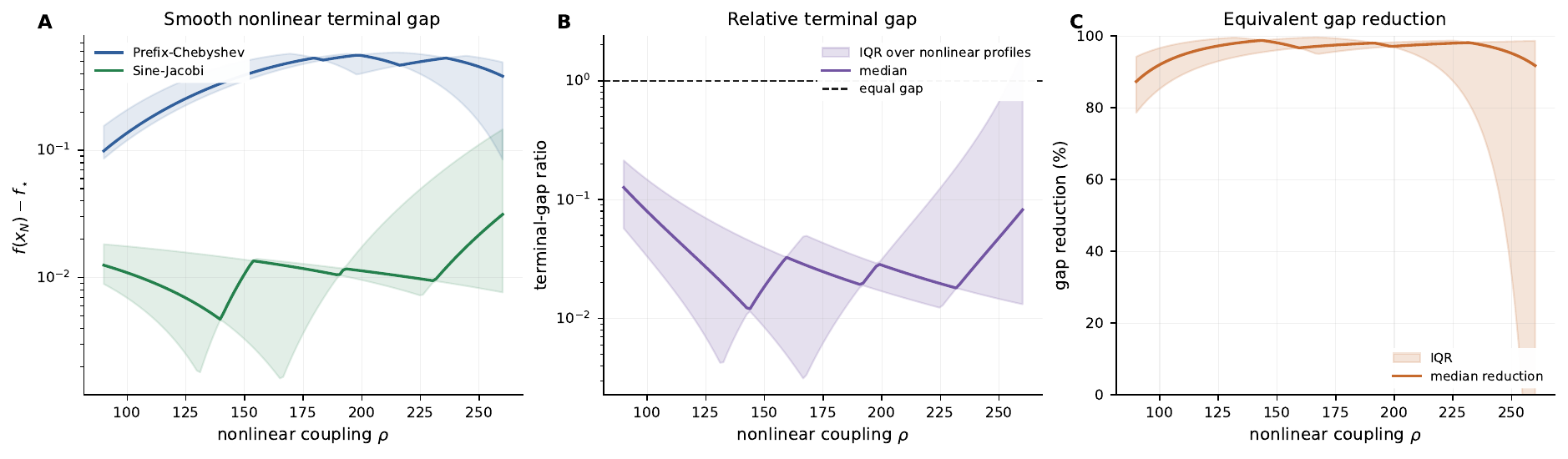}
    \caption{Smooth nonlinear transfer family.  Panel A shows the terminal objective gaps for the two realizations.  Panel B shows the sine/prefix terminal-gap ratio.  Panel C rewrites the same comparison as an equivalent gap reduction.  Curves are medians over five nonlinear profiles and shaded regions are interquartile ranges.}
    \label{fig:app-nonlinear-transfer}
\end{figure}

For each fixed profile, the functions plotted in Figure~\ref{fig:app-nonlinear-transfer} are smooth in \(\rho\).  The visible bends in the median and interquartile curves arise from two deterministic sources.  First, the median over five profiles is a pointwise order statistic; as \(\rho\) changes, the profile attaining the median can switch, producing a piecewise-smooth aggregate even though every individual profile is smooth.  Second, the nonlinear term changes both the minimizer and the local Hessian along the trajectory.  The two realizations still have the same terminal Chebyshev polynomial at the fixed-Hessian linearization, so both terminal gaps vary on comparable \(\rho\)-scales.  The ratio remains well below one over the central range: the overall median sine/prefix gap ratio is \(2.79\cdot10^{-2}\), and the median equivalent gap reduction is above \(87\%\) across the plotted range.  At \(\rho=90\), the median prefix and sine gaps are approximately \(9.89\cdot10^{-2}\) and \(1.25\cdot10^{-2}\); at \(\rho=200.5\), they are approximately \(5.53\cdot10^{-1}\) and \(1.13\cdot10^{-2}\).  The similar large-scale shape of the two terminal-gap curves reflects the common smooth strongly convex family, while the vertical separation reflects the smaller realization-level curvature response of the sine--Jacobi chain.


\begin{thebibliography}{99}
\bibitem{SpiridonovZhedanov}
V. Spiridonov and A. Zhedanov,
\(q\)-Ultraspherical polynomials for \(q\) a root of unity,
Letters in Mathematical Physics 37 (1996), 173--180.
\url{https://arxiv.org/pdf/q-alg/9605033}.

\bibitem{KS98}
R. Koekoek and R. F. Swarttouw, The Askey-scheme of hypergeometric orthogonal polynomials and its $q$-analogue. Delft University of Technology, Report 98-17, 1998.
\url{https://tudelft.roelofkoekoek.nl/documents/as98.pdf}.

\bibitem{Polyak1964}
B. T. Polyak, Some methods of speeding up the convergence of iteration methods.
USSR Computational Mathematics and Mathematical Physics 4 (1964), no. 5, 1--17.

\bibitem{Nesterov1983}
Y. E. Nesterov, A method for solving the convex programming problem with convergence rate
\(O(1/k^2)\). Soviet Mathematics Doklady 27 (1983), 372--376.

\bibitem{dAspremontScieurTaylor}
A. d'Aspremont, D. Scieur, and A. Taylor, Acceleration Methods.
Foundations and Trends in Optimization 5 (2021), no. 1--2, 1--245.
\url{https://arxiv.org/pdf/2101.09545v4}.

\bibitem{DroriTeboulle2014}
Y. Drori and M. Teboulle, Performance of first-order methods for smooth convex minimization:
a novel approach. Mathematical Programming 145 (2014), 451--482.

\bibitem{AndrewsEkhadZeilberger1993}
G. E. Andrews, S. B. Ekhad, and D. Zeilberger,
A short proof of Jacobi's formula for the number of representations
of an integer as a sum of four squares.
The American Mathematical Monthly 100 (1993), no.~3, 273--276.
\url{https://arxiv.org/pdf/math/9206203v1}.


\bibitem{TaylorHendrickxGlineur2017}
A. B. Taylor, J. M. Hendrickx, and F. Glineur, Smooth strongly convex interpolation and exact
worst-case performance of first-order methods. Mathematical Programming 161 (2017), 307--345.
\url{https://arxiv.org/pdf/1502.05666v6}.

\bibitem{AltschulerParriloSilverI}
J. M. Altschuler and P. A. Parrilo,
Acceleration by Stepsize Hedging: Multi-Step Descent and the Silver
Stepsize Schedule.
Journal of the ACM 72 (2025), no.~2, Article~12, 1--38.
\url{https://arxiv.org/pdf/2309.07879v1}.

\bibitem{AltschulerParriloSilverII}
J. M. Altschuler and P. A. Parrilo,
Acceleration by stepsize hedging: Silver Stepsize Schedule for smooth
convex optimization.
Mathematical Programming 213 (2025), 1105--1118.
\url{https://link.springer.com/content/pdf/10.1007/s10107-024-02164-2.pdf}.

\bibitem{AgarwalGoelZhang2021}
N. Agarwal, S. Goel, and C. Zhang,
Acceleration via Fractal Learning Rate Schedules.
In Proceedings of the 38th International Conference on Machine Learning,
Proceedings of Machine Learning Research 139 (2021), 87--99.
\url{https://proceedings.mlr.press/v139/agarwal21a/agarwal21a.pdf}.

\bibitem{KimFessler2016}
D. Kim and J. A. Fessler, Optimized first-order methods for smooth convex minimization.
Mathematical Programming 159 (2016), 81--107.
\url{https://arxiv.org/pdf/1406.5468v4}.

\bibitem{nesterov2004introductory}
Y.~Nesterov,
\newblock {\em Introductory Lectures on Convex Optimization: A Basic Course},
\newblock Springer, New York, 2004.
\newblock \url{https://doi.org/10.1007/978-1-4419-8853-9}.

\bibitem{arjevani2015lower}
Y.~Arjevani, S.~Shalev-Shwartz, and O.~Shamir,
\newblock On lower and upper bounds for smooth and strongly convex optimization problems,
\newblock arXiv:1503.06833, 2015.
\newblock \url{https://arxiv.org/abs/1503.06833}.

\bibitem{DevolderGlineurNesterov2014}
O. Devolder, F. Glineur, and Y. Nesterov, First-order methods of smooth convex optimization
with inexact oracle. Mathematical Programming 146 (2014), 37--75
\url{https://optimization-online.org/wp-content/uploads/2010/12/2865.pdf}.

\bibitem{CohenDiakonikolasOrecchia2018}
M. B. Cohen, J. Diakonikolas, and L. Orecchia, On acceleration with noise-corrupted gradients.
Proceedings of the 35th International Conference on Machine Learning, PMLR 80 (2018), 1019--1028.
page references count from the first page of
\url{https://proceedings.mlr.press/v80/cohen18a/cohen18a.pdf}.

\bibitem{LessardRechtPackard2016}
L. Lessard, B. Recht, and A. Packard, Analysis and design of optimization algorithms via integral
quadratic constraints. SIAM Journal on Optimization 26 (2016), no. 1, 57--95.
\url{https://arxiv.org/pdf/1408.3595v7}.

\bibitem{AybatFallahGurbuzbalabanOzdaglar2020}
N. S. Aybat, A. Fallah, M. Gurbuzbalaban, and A. Ozdaglar, Robust accelerated gradient methods
for smooth strongly convex functions. SIAM Journal on Optimization 30 (2020), no. 1, 717--751.
\url{https://arxiv.org/pdf/1805.10579v4}.

\bibitem{MohammadiRazaviyaynJovanovic2021}
H. Mohammadi, M. Razaviyayn, and M. R. Jovanovic, Robustness of accelerated first-order
algorithms for strongly convex optimization problems. IEEE Transactions on Automatic Control
66 (2021), no. 6, 2480--2495.
\url{https://arxiv.org/pdf/1905.11011v2}.

\bibitem{Teschl}
G. Teschl, Jacobi Operators and Completely Integrable Nonlinear Lattices.
American Mathematical Society, 2000.
\url{https://www.mat.univie.ac.at/~gerald/ftp/book-jac/jacop.pdf}.

\bibitem{DLMF}
F. W. J. Olver, A. B. Olde Daalhuis, D. W. Lozier, B. I. Schneider, R. F. Boisvert, C. W. Clark, B. R. Miller, B. V. Saunders, H. S. Cohl, and M. A. McClain, eds.,
NIST Digital Library of Mathematical Functions. Release 1.2.7 (2026-06-15),
\url{https://dlmf.nist.gov/}.

\bibitem{OnoRobinsWahl}
K. Ono, S. Robins, and P. T. Wahl,
On the representation of integers as sums of triangular numbers.
Aequationes Mathematicae 50 (1995), 73--94.
\url{https://uva.theopenscholar.com/files/ken-ono/files/006_8.pdf}.

\end{thebibliography}
\end{document}